\documentclass[11pt,amsfonts, epsfig]{amsart}
\usepackage{amsmath, amscd, amssymb}
\usepackage{graphicx, psfrag}
\usepackage{graphpap, color}
\usepackage{mathrsfs}
\usepackage{pstricks}
\usepackage{color}
\usepackage{cancel}
\usepackage[mathscr]{eucal}

\usepackage{pstricks}
\usepackage{color}
\usepackage{cancel}
\usepackage{verbatim}
\usepackage{latexsym}
\usepackage[all]{xy}

\def\contri{\text{Contr}}

\usepackage{lipsum}
\makeatletter
\g@addto@macro{\endabstract}{\@setabstract}
\newcommand{\authorfootnotes}{\renewcommand\thefootnote{\@fnsymbol\c@footnote}}%
\makeatother

\usepackage{upgreek}
\usepackage{wasysym}
\usepackage{ esint }

\def\fc{\mathfrak c}

\def\gg{{\mathfrak g}}

\def\vdim{\mathrm{vir}.\dim}

\def\virt{^{\vir}}
\def\virtloc{\virt\loc}

\def\bcM{\overline{\cM}}

\numberwithin{equation}{section}

\def\gm{\GG_m}

\def\loc{_1^\circ}

\def\sO{{\mathscr O}}
\def\sC{{\mathscr C}}

\def\sM{{\mathscr M}}
\def\sN{{\mathscr N}}
\def\sL{{\mathscr L}}
 
\def\sI{{\mathscr I}}
\def\sO{\mathscr{O}}
\def\sI{\mathscr{I}}

\def\sF{\mathscr{F}}
\def\sV{\mathscr{V}}
\def\sA{\mathscr{A}}
\def\sU{\mathscr{U}}
\def\sR{\mathscr{R}}
\def\sX{\mathscr{X}}

\newcommand{\CC}{\mathbb{C}}

\newcommand{\EE}{\mathbb{E}}

\newcommand{\PP}{\mathbb{P}}
\newcommand{\QQ}{\mathbb{Q}}

\newcommand{\ZZ}{\mathbb{Z}}
\newcommand{\GG}{\mathbb{G}}

\newcommand{\TT}{\mathbb{T}}

\def\lwd{_{\cW/\cD}}

\def\umv{^{\mathrm{mv}}}

\newcommand{\bL}{\mathbf{L}}
\newcommand{\bT}{\mathbf{T}}

\newcommand{\fw}{{\mathfrak w}}

\def\redd{{\mathrm{red}}}

\def\Gm{{\bG_m}}

\def\upmo{^{-1}}

\newcommand{\Def}{ \mathrm{Def} }
\newcommand{\Obs}{ \mathrm{Obs} }
\newcommand{\vir}{ {\mathrm{vir}} }
\newcommand{\ev}{ \mathrm{ev} }
\newcommand{\vg}{\vec{g}}

\newcommand{\vS}{\vec{S}}
\newcommand{\vGa}{\Ga}
\def\cA{\mathcal A}
\def\Wfix{\cW_\Ga} 
\def\fixW{\cW^{\,\Gm}}

\newcommand{\cal}{\mathcal}
\def\ee{\mathfrak e}

\def\cC{{\cal C}}
\def\cD{{\cal D}}
\def\cE{{\cal E}}
\def\cF{{\cal F}}
\def\cH{{\cal H}}

\def\cL{{\cal L}}
\def\cM{{\cal M}}
\def\cN{{\cal N}}
\def\cO{{\cal O}}
\def\cP{{\cal P}}
\def\cQ{{\cal Q}}

\def\cV{{\cal V}}
\def\cW{{\cal W}}

\def\cS{{\cal S}}
\def\cX{X} 

\def\fC{\mathfrak{C}}

\def\ft{\mathfrak{t}}

\def\v1{{\vec{1}}}

\def\sP{{\mathscr P}}

\def\fl{{\mathrm{fl}}}

\newcommand{\Mbar}{\overline{\cM}}

\def\mapright#1{\,\smash{\mathop{\lra}\limits^{#1}}\,}

\newcommand{\vd}{\vec{d}}

\def\dual{^{\vee}}

\def\sta{^\ast}
\def\st{^{\mathrm{st}}}
\def\virt{^{\mathrm{vir}}}
\def\upmo{^{-1}}
\def\sta{^{\ast}}
\def\dpri{^{\prime\prime}}

\def\sta{^*}

\def\lra{\longrightarrow}

\def\lsta{_{\ast}}

\newcommand{\Del}{\Delta}
\newcommand{\Si}{\Sigma}
\newcommand{\Ga}{\Gamma}

\newcommand{\ep}{\epsilon}
\newcommand{\lam}{\lambda}

\def\sQ{\mathscr Q}

\def\lrga{_{(\ga)}}

\def\begeq{\begin{equation}}
\def\endeq{\end{equation}}
\def\and{\quad{\rm and}\quad}
\def\bl{\bigl(}
\def\br{\bigr)}
\def\defeq{:=}

\def\sub{\subset}
\def\Ao{{\mathbb A}^{\!1}}

\def\Po{{\mathbb P^1}}
\def\and{\quad\text{and}\quad}

\def\reg{{\mathrm{reg}}}

 \DeclareMathOperator{\Ext}{Ext}
  
 \DeclareMathOperator{\Aut}{Aut}
 
\DeclareMathOperator{\id}{id} 
 \DeclareMathOperator{\res}{res}
 
 \DeclareMathOperator{\rank}{rank}
\DeclareMathOperator{\spec}{Spec}

\def\lggd{_{g,\gamma,\bd}}
\def\cWg{\cW\lggd}
\def\lgd{_{g,\bd}}

\newtheorem{prop}{Proposition}[section]
\newtheorem{proposition}[prop]{Proposition}
\newtheorem{theo}[prop]{Theorem}
\newtheorem{lemm}[prop]{Lemma}
\newtheorem{coro}[prop]{Corollary}
\newtheorem{rema}[prop]{Remark}
\newtheorem{exam}[prop]{Example}
\newtheorem{defi}[prop]{Definition}
\newtheorem{definition}[prop]{Definition}
\newtheorem{conj}[prop]{Conjecture}

\newtheorem{cons}[prop]{Construction}

\newtheorem{defi-prop}[prop]{Definition-Proposition}
\newtheorem{defi-theo}[prop]{Definition-Theorem}

\DeclareMathOperator{\coker}{coker}

\def\sS{\mathscr S}

\def\Ob{\cO b}

\def\nar{^{\mathrm{na}}}
\def\bro{^{\mathrm{br}}}

\def\loc{_{\mathrm{loc}}}

\def\ev{\text{ev}}

\def\sta{^\ast}

\def\sO{{\mathscr O}}

\def\sR{{\mathscr R}}
\def\sD{{\mathscr D}}

\def\beq{\begin{equation}}
\def\eeq{\end{equation}}

\def\vsp{\vskip3pt}
\def\Pf{{\PP^4}}

\def\bee{\begin{equation}}
\def\eeq{\end{equation}}

\def\sC{{\mathscr C}}

\def\bd{{\mathbf d}}

\let\eps=\epsilon

\def\ti{\tilde}

\def\barM{{\overline{M}}}

\def\mapright#1{\,\smash{\mathop{\lra}\limits^{#1}}\,}

\def\un{^{\mathrm{un}}}

\def\mufive{{\boldsymbol\mu_5}}

\let\ofth=\fo
\def\lred{_{\mathrm{red}}}
\def\fractf{\frac{2}{5}}

\def\bmu{{\boldsymbol \mu}}

\def\Gm{T}

\let\ga=\Ga

\def\lorho{_{^{(1,\rho)}}}
\def\lophi{_{^{(1,\varphi)}}}

\begin{document}

\title[GW and FJRW Invariants of quintic Calabi-Yau manifolds]{An effective
theory of GW and FJRW invariants of quintics Calabi-Yau manifolds}

\author[Huai-Liang Chang]{Huai-Liang Chang$^1$}
\address{Department of Mathematics, Hong Kong University of Science and Technology, Hong Kong} \email{mahlchang@ust.hk}
\thanks{${}^1$Partially supported by Hong Kong GRF grant 600711,  GRF 16301515 and  GRF 16301717}

\author[Jun Li]{Jun Li$^2$}
\address{Department of Mathematics, Stanford University,
USA; \hfil\newline 
\indent Shanghai Center for Mathematical Sciences, Fudan University, China} \email{jli@stanford.edu}
\thanks{${}^2$Partially supported by  NSF grant DMS-1564500 and DMS-1601211.}

\author[Wei-Ping Li]{Wei-Ping Li$^3$}
\address{Department of Mathematics, Hong Kong University of Science and Technology, Hong Kong} \email{mawpli@ust.hk}
\thanks{${}^3$Partially supported by Hong Kong GRF grant 602512 and  GRF 16301515.}

\author[Chiu-Chu Melissa Liu]{Chiu-Chu Melissa Liu$^4$}
\address{Mathematics Department, Columbia University}
 \email{ccliu@math.columbia.edu}
\thanks{${}^4$Partially supported by  NSF grant DMS-1206667, DMS-1159416 and DMS-1564497.
}


\maketitle

\begin{abstract}  We analyze the torus fixed loci of Mixed Spin P fields moduli, and deduce its localization formulas with explicit factors.  An algorithm toward evaluating quintic's Gromov-Witten and Fan-Jarvis-Ruan-Witten invariants is derived. 
\end{abstract}

\section{Introduction}

In \cite{CLLL}, we introduced the notion of Mixed-Spin-P (MSP) fields
\beq\label{MSP0}
\xi=( {\Si^\sC}, \sC, \sL, \sN,\varphi,\rho, \nu)
\eeq
consisting of pointed twisted curves $\Sigma^\sC\sub\sC$, line bundles $\sL$ and $\sN$,
and fields $\varphi\in H^0(\sL^{\oplus 5})$,
$\rho\in H^0(\sL^{\vee\otimes5}\otimes \omega^{\log}_{\sC})$, and  
$\nu\in H^0(\sL\otimes\sN\oplus \sN)$, subject
to conditions of non-degeneracy.
Its numerical invariants are the genus $g$ of $\sC$, the
monodromy $\gamma_i$ of $\sL$ at the marking $\Si^\sC_i\sub\Si^\sC$,
and the degrees $d_0=\deg \sL\otimes\sN$ and $d_\infty=\deg \sN$.

In the same paper, we constructed the moduli $\cW\lggd$ of the MSP fields of data $g$,
$\gamma=(\gamma_1,\cdots,\gamma_\ell)$ and $\bd=(d_0,d_\infty)$,
and proved that they are $\GG_m$-DM stacks,
have $\GG_m$-equivariant relative perfect obstruction theories, and admit
$\GG_m$-invariant cosections of their obstruction sheaves with proper degeneracy locus.
Applying the cosection localized virtual cycle construction,
we obtained  properly supported $\GG_m$-equivariant cycles $[\cW\lggd]\virtloc$.

\vsp

In this paper, we will derive a class of vanishings using these cycles. Applying the
virtual localization formula to these vanishings, we find polynomial
relations among  GW invariants and  FJRW invariants of Fermat quintic polynomials:
\beq\label{0}
\sum_\Gamma \res_{\ft=0}\Bigl(\ft^{\delta(g,\gamma,\bd)-1} \cdot\frac{[(\cW\lggd)^{\GG_m}_\Gamma]\virtloc}
{e(N_{(\cW\lggd)^{\GG_m}_\Gamma/\cW\lggd})}\Bigr)_0=0,\ \text{when}\ \ \delta(g,\gamma,\bd)>0.
\eeq
(See \S \ref{Sub5.1} for the notations for $\Gamma$ and \S \ref{subsection2.4} for the notations for $(\cW\lggd)_\Gamma$.)
Here $\ft$ is the generator in $A^{\GG_m}\lsta (B\GG_m)=\QQ[{\ft}]$;
$\delta(g,\gamma,\bd)$ is the virtual dimension of $\cW\lggd$.

These relations provide an effective algorithm to evaluate all genus GW invariants of
quintic threefolds, and all genus FJRW invariants of Fermat quintics with
$\fractf$-insertions,
provided that a range of ``initial" GW invariants
of quintic threefolds and FJRW invariants of  Fermat quintics are known.

More precisely,  let $N_{g,d}$ be the genus $g$ degree $d$ GW invariants of the quintic threefold, let $\fw_5=x_1^5+\cdots+x_5^5$, and let $\Theta_{g,k}$ be the genus $g$ FJRW invariants of $([\CC^5/\ZZ_5],\fw_5)$ with
$k$ many $\fractf$-insertions. (A $\fractf$-insertion is a marking $\Si_i\sub\Si^\sC$ so that the monodromy representation
of $\sL$ along $\Si_i$ is $\zeta_5^2$, where $\zeta_5=\exp(\frac{2\pi \sqrt{-1}}{5})$. For more details, see \S \ref{2.1}.)

\vsp
Using relations \eqref{0}, we obtain, for $g\geq 1$,


\begin{theo}\label{main-1} Let $d\ge g\geq 1$. The relations  \eqref{0} with $\gamma=\emptyset$ and $d_\infty=0$ provide an effective algorithm evaluating $N_{g,d}$, provided  \begin{enumerate}
 \item $N_{g',d'}$ are known for $g'\leq g$ and $d'<d$
\item $\varTheta_{g,k}$ are known for $k\leq 2g-2$ and $\varTheta_{g',k}$ are known for $g' < g, k\leq 2g-4$.
\end{enumerate}
\end{theo}

\begin{theo}\label{main-2}
For $g\geq 1$ and $k\ge 7g-2$, the relations \eqref{0} with $\gamma=\emptyset$, $d_0=0$ 
provide an effective algorithm to evaluate $\Theta_{g,k}$,
provided that $\{\varTheta_{g,k' }\}_{k'< k}$ and $\{\Theta_{g',k'}\}_{g'<g,k'\le k-2}$ are known.
\end{theo}

As a special case, in $g=1$ all FJRW invariants $\Theta_{1,k}$ for the  quintic singularity of $([\CC^5/\ZZ_5],\fw_5)$  can be evaluated.
(Recall that FJRW invariants $\Theta_{0,k}$ were calculated via GRR formula in \cite{CR}.) 
\medskip

The previous two theorems provide  a symmetric form of the algorithms.

\begin{coro}\label{symmetry} We let $g\ge 1$, and let 
$A_g=\{\Theta_{g,k}\}_{k\leq 2g-2}$, $P_g =\{\Theta_{g,k}\}_{2g-2<k<7g-2}$, and 
$N_g=\{N_{g,d}\}_{0<d<g}$.
Suppose all $N_{g',d'}$ and $\Theta_{g',k'}$ are known for $g'<g$. Then the algorithm based on \eqref{0} can
\begin{itemize}
\item[(i)] determine all $N_{g,d}$ based on $A_g$ and $N_g$, and 

\item[(ii)] determine all $\Theta_{g,k}$ based on $A_g$ and $P_g$. 
\end{itemize}
\end{coro}

In the end, using relation \eqref{0} to determine all $N_{g,d}$ and $\Theta_{g,k}$ amounts to (1) find a way to evaluate 
$N_g$, $A_g$ and $P_g$; and (2) organize (package) the relation to solve the generating function of
$N_{g,d}$ and $\Theta_{g,k}$, as explicitly as possible.

Towards  (1) , we want to reduce the initial value sets $N_g$, etc.. 
As $\Theta_{g,k}$ is defined only when $2g-2\equiv k (5)$, the sizes of the mentioned sets are
 $|N_g|=|P_g|=g-1$, and $|A_g| =[2(g-1)/5]+1$. We conjecture that the relations in \eqref{0} will make the  sets  
$N_g$ and $P_g$ equivalent. 

\begin{conj}\label{conj-1}  Let $g\ge 1$. Suppose all $N_{g'<g,d'}$, $\Theta_{g'<g,k'}$ are known, and  
$A_g$ is also known. 
The relations \eqref{0}, for $d_0$, $d_\infty>0$, provide an effective algorithm to determine
the set $N_g$ from $P_g$, and vice versa.
\end{conj}



\vsp

This work was inspired by the work of Fan-Jarvis-Ruan on their analytic construction of
FJRW invariants \cite{FJR1} and by Witten's vision that ``Calabi-Yau and Landau-Ginzburg (are) separated
by a true phase transition" \cite{GLSM}. 

\vsp
After introducing the GW invariants of stable maps with $p$-fields and proving their equivalence
with the GW invariants of quintic CY threefolds by the first two named authors \cite{CL},
our goal is to search for a geometric theory of wall crossings envisioned by Witten.
In \cite{CLL}, the first three named authors reconstructed the (narrow) FJRW
invariants via cosection localized virtual cycles. In the first part of this project \cite{CLLL}, the authors
introduced the theory of Mixed-Spin-P fields. 
It is a field theory with an added field $\nu$, which is a
``quantization" of Witten's parameter in his phase transition between Calabi-Yau and Landau-Ginzburg
theories; or in the setting of this paper, between GW invariants of stable maps with $p$-fields and FJRW
invariants via cosection localized virtual cycles.


The current paper sets up the foundation for applying localization to 
MSP field theory, in order to study GW and FJRW invariants of Fermat quintics.

\vsp


Around the time of the completion of the first draft of this paper, there were  other works to study all genus GW invariants of quintics.    Ciocan-Fontanine and Kim  \cite{CK}(see also \cite{CJR}, \cite{Zh}) 
provided wall crossings of GW invariants with other  quasimap invariants, by varying $\epsilon$-stabilities.  
Fan-Jarvis-Ruan \cite{FJR3} proposed a version of GLSM theory. (See also the work of Tian-Xu \cite{TX}.)
Choi and Kiem \cite{ChK}
introduced additional $\delta$-stability to study wall-crossings. 
\vsp

  There had been a few notable progress obtained by applying the 
 localization of MSP relations \eqref{0} developed by this paper.  
In \cite{GR1,GR2}, Guo-Ross
used Theorem \ref{main-1} to package $g=1$ FJRW invariants of the Fermat quintic,
confirming the mirror symmetry prediction and verifying the $g=1$ Landau-Ginzburg/Calabi-Yau correspondence.  
In \cite{CGLZ}, Chang-Guo-Li-Zhou used Theorem \ref{main-1} to obtain an explicit formula of $g=1$ GW invariants of the quintic CY threefold, recovering Zinger's formula \cite{Zi} (also recovered by \cite{KLh}). 



  Later,  there appeared some other approaches toward determining quintic's (and other CY's) higher genus GW. Fan-Lee \cite{FL} applied  localization to a degeneration of $\Pf$ to the quintics, and   obtained a  recursion determining $N_{g,d}$'s up to initial conditions. Also, Chen-Janda-Ruan   \cite{CJRS}  and    Guo-Janda-Ruan   \cite{GJR}, and \cite{GJR18} used log compactifications of p fields to approach the problem.

Recently, the MSP approach obtains fruitful results on quintic's high genus GW invariants' structures. In \cite{NMSP1}, a NMSP field  is the same as an MSP field in \eqref{MSP0} except  $\nu\in H^0(\sL\otimes \sN\oplus \sN)$ replaced by $\nu \in H^0( (\sL\otimes\sN)^{\oplus N} \oplus \sN)$. 
The moduli space of NMSP fields has  the same properties  as this paper, along with an action by $(\CC\sta)^{\times N}$.  Its localization (\cite{NMSP1}) follows from the arguments provided by this paper. In \cite{NMSP2,NMSP3}, by packaging the explicit form of Theorem \ref{main-1}, Chang-Guo-Li proved a sequence of results regarding quintic  GW potentials $F_g:=\sum N_{g,d}q^d$.  They are the Yamaguchi-Yau \cite{YY} finite generation conjecture, the Yamaguchi-Yau (functional) equations \cite{YY, ASYZ}\footnote{ also called Holomorphic Anomaly Equations(HAE) by Pandharipande and Lho \cite{LP}.  }, the convergence of  $F_g$ with positive radius, and BCOV's Feynman rules \cite{BCOV} which determines every $F_g$  diagrammatically and recursively, based on $3g-3$ initial data.  As a consequence, the same authors recover $F_2$'s closed formula in \cite{NMSP3}.   An analogous package of Theorem \ref{main-2} should lead to Feynman structures of FJRW's potentials. 

   After these works, determining the initial data becomes one of the unsolved important questions on determining $F_g$'s. Corollary \ref{symmetry} and Conjecture \ref{conj-1} provide  parts of the relations between the initial conditions (both GW and FJRW). We shall come back to this topic in the future. 
 
 The series of works on NMSP moduli rely on  three results: (i) properness of degeneracy loci \cite{CLLL}, (ii) vanishing of irregular graphs' contribution
 to localization formula \cite{CL2},  and (iii) the fixed loci and localization formulas in this paper that can lead to   packages via cohomological field theories.

The paper is organized as follows. In \S2, after reviewing the basic properties of MSP fields,
we give  the topological characterization of a partition of the $\CC\sta$-fixed locus of the moduli space.
In \S3, we work out the cosection localized virtual cycle of the fixed locus. The data in the
virtual localization formula associated to the moving parts is worked out in \S4. In \S5, we derive a vanishing relation,
and prove the main theorems of this paper using this relation. 


\vsp 

\section{Moduli of $T$-invariant MSP-fields}
\def\st{{\mathrm{st}}}
\def\wt{\mathrm{wt}}

In this section, we introduce MSP fields and 
a torus action on their moduli spaces.  We analyze their $\CC\sta$-fixed locus.

\subsection{Definition of MSP fields}\label{2.1}
Let $\bmu_5\le\GG_m$ be the subgroup of $5^{\text{th}}$-roots of unity, generated by $\zeta_5=\exp(\frac{2\pi\sqrt{-1}}{5})$. Let
\beq\label{gamma}\bmu_5\bro=\{(1,\rho),(1,\varphi)\}\cup\bmu_5\and
\bmu_5\nar=\bmu_5\bro-\{1\}.\footnote{Here the superscript ``br" stands for broad, and ``na" stands for narrow.}
\eeq

Here $(1,\rho)$ and $(1,\varphi)$ are merely symbols, with the convention
$\langle (1,\rho)\rangle=\langle (1,\varphi)\rangle=\{1\}\le \gm$.
We call $\gamma\in (\bmu_5\bro)^{\times\ell}$ broad, and call $\gamma\in (\bmu_5\nar)^{\times\ell}$ narrow.
We call an $\ell$-pointed twisted curve $\Si^\sC\sub \sC$ $\gamma$-pointed if $\Si_i^\sC$ is labeled by $\gamma_i$.
Let
$$\omega^{\log}_{\sC/S}=\omega_{\sC/S}(\Si^\sC),\and \Si^\sC_\alpha=\coprod_{\gamma_i=\alpha}\Si^\sC_i
\quad\hbox{ for $\alpha\in \bmu_5\nar$ or $\bmu_5\bro$}.
$$


\begin{definition}\label{def-curve}
A $(g,\gamma,\bd)$ MSP-field is a collection as in \eqref{MSP0}
such that

\begin{enumerate}
\item[(1)] $\cup_{i=1}^\ell\Si_i^\sC = \Si^\sC\subset \sC$ is a genus $g$, $\gamma$-pointed
twisted curve such that the $i$-th marking $\Si^\sC_i$ is banded by the group $\langle\gamma_i\rangle\le \gm$;
\item[(2)] $\sL$ and $\sN$ are invertible sheaves on $\sC$, $\sL\oplus \sN$ representable,
$\deg \sL\otimes \sN=d_0$,   $\deg \sN=d_\infty$, and
the monodromy of $\sL$ along $\Si^\sC_i$ is
$\gamma_i$ when $\langle \gamma_i\rangle\ne\langle 1\rangle$;
\item[(3)] $\nu=(\nu_1, \nu_2)\in H^0( \sL\otimes\sN)\oplus  H^0( \sN)$, and $(\nu_1,\nu_2)$ is nowhere zero;
\item[(4)]
$\varphi=(\varphi_1,\ldots, \varphi_{5}) \in H^0(\sL)^{\oplus 5}$,  $(\varphi,\nu_1)$ is nowhere zero, and $\varphi|_{\Si^\sC_{(1,\varphi)}}=0$;
\item[(5)] $\rho \in H^0(\sL^{\vee\otimes 5}\otimes \omega^{\log}_{\sC/S})$,
$(\rho,\nu_2)$ is nowhere zero, and $\rho|_{\Si^\sC_{(1,\rho)}}=0$.
\end{enumerate}
We call $\xi$ broad if $\gamma$ is broad ($\gamma\in (\bmu_5\bro)^{\times\ell}$); we
call $\xi$ narrow if $\gamma$ is narrow. 
\end{definition}

The distinction between
$(1,\varphi)$ and $(1,\rho)$ lies in item (4) and (5). When $\gamma_i=(1,\varphi)$ or $(1,\rho)$,
by the representable requirement, $\Si^\sC_i$ is a scheme point.

Note that a $\fractf$-insertion is a marking with monodromy representation the multiplication
by $\zeta_5^2$. A local example is as follows. 
Consider $\sC=[\Ao/\bmu_5]$, where $\bmu_5$ acts on $\Ao=\spec \CC[x]$ via
$\zeta_5\cdot x=\zeta_5\upmo x$.
Then the $\sO_\sC$-module $x^{-2}\CC[x]$ has monodromy representation via multiplication by $\zeta_5^2$.
\vsp


Throughout this paper, unless otherwise mentioned,  for any closed point $\xi\in\cW\lggd$ we understand
that $\xi$ is $(\Si^\sC,\sC,\sL,\cdots)$ given in \eqref{MSP0}.


\subsection{Decorated graphs of torus fixed MSP fields}
\label{Sub5.1}
By the main theorem of \cite{CLLL}, the category $\cW\lggd$
of families of MSP-fields of data $(g,\gamma,\bd)$ is a separated DM stack.
Let $T=\GG_m$, and endow $\cW\lggd$ a $T$-structure, via 
\beq\label{action}t\cdot (\sC,\Si^\sC, \sL,\sN,\varphi,\rho, \nu_1,\nu_2)=(\sC,\Si^\sC, \sL,\sN,\varphi,\rho, t\nu_1,\nu_2).
\eeq
In this subsection, we study the $\Gm$-fixed locus $(\cWg)^\Gm\sub \cWg$.

In the following, we fix a $(g,\gamma,\bd)$. For notational simplicity,
whenever $(g,\gamma,\bd)$ is understood, we will adopt the convention
\beq\label{convention}
\cW=\cWg,\quad \cW^-=\cWg^-,\and \fixW=(\cWg)^T.
\eeq

Let $\xi\in \fixW$. 
Following the definition of $\Gm$ fixed points,
and after a standard algebraic argument, we conclude that
there are homomorphisms $h$ and $\Gm$-linearizations $\tau$ and $\tau'$
(satisfying the obvious relations) such that 
$$h: \Gm\to \text{Aut}(\sC,\Si^\sC);\quad
\tau_t: h_{t\ast}\sL\to \sL\and \tau'_t:
h_{t\ast}\sN\to \sN,
$$
\beq\label{inv}
t\cdot (\varphi,\rho, \nu_1,\nu_2)=
(\tau_t,\tau_t')(h_{t\ast} \varphi,h_{t\ast}\rho,t \cdot h_{t\ast}\nu_1,h_{t\ast}\nu_2),\quad
t\in\Gm.\footnote{For convenience we allow the $T$-acting on curves and etc. with rational weights.}
\eeq
We call such $\Gm$-actions and linearizations those induced from $\xi\in \fixW$.
Using that $\xi\in \cW^T$ is stable, one sees that such $(h,\tau_t,\tau_t')$ 
is unique.

We let $\bL_k$ be the one-dimensional weight $k$ $T$-representation. We let $\sL^{\log}=\sL(-\Si^\sC\lophi)$,
and $\sP^{\log}=\sL^{-5}\otimes\omega_\sC^{\log}(-\Si^\sC\lorho)$. Then \eqref{inv} can be
rephrased as
\beq\label{inv-section}
(\varphi,\rho, \nu_1,\nu_2)\in H^0\bl(\sL^{\log})^{\oplus 5}\oplus\sP^{\log}\oplus\sL\otimes\sN\otimes\bL_1\oplus\sN\br^T.
\eeq

We now investigate the structure of any $T$-equivariant MSP fields $\xi\in \cW^T$.


\begin{defi} \label{Ca}
Given $\xi\in \fixW$,  let
$\sC_0=\sC\cap (\nu_1=0)_\redd$,  $\sC_\infty=\sC\cap (\nu_2=0)_\redd$,
$\sC_1=\sC\cap (\rho=\varphi=0)_\redd$,
 $\sC_{01}$ (resp. $\sC_{1\infty}$) be the union of irreducible components of
$\overline{\sC-\sC_0\cup\sC_1\cup\sC_\infty}$ in
$(\rho=0)$ (resp. in $(\varphi=0)$), and  $\sC_{0\infty}$ be the union of irreducible components of
$\sC$ not contained in $\sC_0\cup\sC_1\cup\sC_\infty\cup\sC_{01}\cup\sC_{1\infty}$.
\end{defi}

As shown in \cite{CLLL}, $\sC_0$, $\sC_1$ and $\sC_\infty$ are mutually disjoint. By definition,
all $\sC_{01}$, $\sC_{0\infty}$ and $\sC_{1\infty}$ are pure one-dimensional, and are
$T$-equivariant.
%

%
%
%
%

\begin{lemm}\label{C-structure}
Let $\xi\in \fixW$. Then
\begin{enumerate}
\item the action $h:\Gm\to \Aut(\sC,\Si^\sC)$ acts trivially on $\sC_0$, $\sC_1$ and $\sC_\infty$;
\item every irreducible component $\sA\sub\sC_{01}$ (resp. $\sA\sub \sC_{1\infty}$; resp. $\sA\sub\sC_{0\infty}$)
is a smooth rational twisted curve with two $\Gm$-fixed points 
lying on $\sC_0$ and $\sC_1$ (resp. $\sC_1$ and $\sC_\infty$; resp. $\sC_0$ and $\sC_\infty$).
\item if $\xi\in \cW^-$, then $\sC_{0\infty}=\emptyset$.
\end{enumerate}
\end{lemm}

\begin{proof}
We prove that $h$ fixes $\sC_0$. Let $\sA\sub\sC_0$ be an irreducible component over which $h|_\sA$ is non-trivial. Then $\sA$ must be a rational curve. Let $x\in\sA$ be fixed by $h$. Since both $\varphi(x)$ and $\nu_2(x)$ are non-zero, $\sL|_x=\sN|_x=\bL_0$; thus $\sL|_{\sA}\cong\sN|_{\sA}\cong\sO_\sA$. Therefore the component $\sA$ makes $\xi$ unstable, a contradiction. Thus $h$ fixes $\sA$. Likewise, $h$ also fixes $\sC_1$ and $\sC_\infty$.

We next fix an irreducible component $\sA\sub \sC_{01}$ and investigate the action $h|_\sA$.
First, by definition of $\sC_{01}$, $\nu_2|_\sA$ is nowhere vanishing, thus $\sN|_\sA\cong \sO_\sA$.
Let $x\in\sA$. Suppose $h|_\sA$ fixes $x$. It is easy to see that one of $\varphi(x)$ and $\nu_1(x)$ must be zero.
If $x$ is a general point of $\sA$,
then by definition of $\sC_{01}$, both $\varphi(x)$ and $\nu_1(x)$ are non-zero.
This proves that $h|_\sA$ acts non-trivially on $\sA$, implying that $\sA$ is rational.

We claim that   $\sA\sub\sC_{01}$ must be smooth. Otherwise, let $x\in\sA$ be a node of $\sA$, which
is fixed by $h|_\sA$. Then apply  Lemma 3.2 in \cite{CLLL} 
and using that either $\nu_1(x)\ne 0$ or $\varphi(x)\ne 0$, we obtain  $\sL|_\sA\cong \sO_\sA$  forcing 
$\xi$ unstable, a contradiction.

As $\sA$ is smooth, $h|_\sA$ has two fixed points, say $x_1,x_2\in\sA$. Like before, if both $\varphi(x_1)$ and
$\varphi(x_2)$ are non-zero, then $\sL|_\sA\cong \sO_\sA$, forcing $\xi$ unstable, a contradiction;
if both $\varphi(x_1)$ and $\varphi(x_2)$ are zero, then both $\nu_1(x_1)$ and $\nu_1(x_2)$ are non-zero,
forcing $\sL|_\sA\cong \sO_\sA$ and $\varphi|_\sA=0$, implying that $\sA\sub \sC_1$, impossible.
This proves that one of $x_1$ and $x_2$ lies in $\sC_0$ and the other lies in $\sC_1$.

A similar argument shows that the parallel conclusion holds for $\sA\sub \sC_{1\infty}$ and $\sC_{0\infty}$. Finally, (3) follows from that both $\rho$ and $\varphi$ are non-zero on every irreducible component of $\sC_{0\infty}$.
\end{proof}

We now associate a decorated graph to each $\xi\in \fixW$.
For a graph $\Gamma$, besides its vertices $V(\Ga)$, edges
$E(\Ga)$ and legs $L(\Ga)$, the set of its flags is
$$
F(\Ga)=\{(e,v)\in E(\Ga) \times V(\Ga): v\ \text{incident to}\ e\}.
$$

\vsp
Given $\xi\in\cW^\Gm$, 
 let $\pi:\sC^{\text{nor}}\to\sC$ be its normalization. 
For any $y\in \pi\upmo(\sC_{\text{sing}})$, we denote by $\gamma_y$ the monodromy of $\pi\sta\sL$ along $y$.

\begin{defi}\label{graph1}
To each $\xi\in \cW^\Gm$ we associate  a graph $\Ga_\xi$ as follows:

\begin{enumerate}
\item (vertex) let $V_0(\Ga_\xi)$, $V_1(\Ga_\xi)$, and $V_\infty(\Ga_\xi)$ be
the set of connected components of $\sC_0$, $\sC_1$, $\sC_\infty$ respectively, and
let $V(\Ga_\xi)$ be their union; 

\item (edge)   let $E_0(\Ga_\xi)$, $E_\infty(\Ga_\xi)$ and $E_{0\infty}(\Ga_\xi)$
be the set of irreducible components
of $\sC_{01}$, $\sC_{1\infty}$ and $\sC_{0\infty}$ respectively, and  let $E(\Ga_\xi)$
be their union; 

\item (leg) let $L(\Ga_\xi)\cong \{1,\cdots,\ell\}$ be the ordered set of markings of $\Si^\sC$, $\Si_i^\sC\in L(\Ga_\xi)$ is attached to
$v\in V(\Ga_\xi)$ if $\Si_i^\sC\in \sC_v$;

\item (flag) $(e,v)\in F(\Ga_\xi)$ if and only if $\sC_e\cap \sC_v\ne \emptyset$. 

%
\end{enumerate}
(Here $\sC_a$ is  the curve associated to the symbol $a\in V(\Ga_\xi)\cup E(\Ga_\xi)$.)
We call $v\in V(\Ga_\xi)$ stable if $\sC_v\sub\sC$ is 1-dimensional, otherwise  unstable.
\end{defi}


In the following,  let $V^S(\Ga_\xi)\subset V(\Ga_\xi)$ be the set of stable vertices and $V^U(\Ga_\xi)\subset V(\Ga_\xi)$ be
that of unstable vertices.
Given $v\in V(\Ga_\xi)$,  let $E_v=\{e\in E(\Ga_\xi): (e,v)\in F(\Ga_\xi)\}$.
For $v\in V^S(\Ga_\xi)$, we define
$$\Si^{\sC_v}_{\text{inn}}=\Si^\sC\cap \sC_v,\quad
\Si^{\sC_v}_{\text{out}}=\overline{(\sC-\sC_v)}\cap\sC_v,\and \Si^{\sC_v}=\Si^{\sC_v}_{\text{inn}}\cup \Si^{\sC_v}_{\text{out}},
$$
called the inner, the outer, and the total
markings of $\sC_v$, respectively.


\begin{defi}\label{graph2} Let $\xi\in\cW^T$.
We endow the graph $\Ga_\xi$ 
a decoration as follows:
\begin{itemize}
\item[(a)] (genus) Define $\vg: V(\Ga_\xi)\to \ZZ_{\geq 0}$ via $\vg(v)=h^1(\sO_{\sC_v})$.
\item[(b)] (degree) Define $\vd: E(\Ga_\xi)\cup V(\Ga_\xi)\to \QQ^{\oplus 2}$ via $\vd(a)=(d_{0a},d_{\infty a})$,
where $d_{0a}= \deg \sL\otimes\sN|_{\sC_a}$ and $d_{\infty a}= \deg\sN|_{\sC_a}$.
\item[(c)] (marking) Define $\vS: V(\Ga_\xi)\to 2^{L(\Ga_\xi)}$ via
$v\mapsto S_v\subset L(\Ga_\xi)$, where $S_v$ is the set of $\Si^{\sC}_i\in \sC_v$.
\item[(d)] (monodromy) Define $\vec{\gamma}: L(\Ga_\xi)\to\mufive$ via $\vec{\gamma}(\Si_i^\sC)=\gamma_i$, which is
the monodromy of $\sL$ along $\Si^{\sC}_i$.
\end{itemize}
\end{defi}

For convenience, we denote $d_{a}=d_{0a}-d_{\infty a}$, i.e., $d_a=\deg\sL|_{\sC_a}$. For $c\in \{0, 1,\infty\}$,  let
$V_c^S(\Ga_\xi)=V_c(\Ga_\xi)\cap V^S(\Ga_\xi)$.
For $v\in V(\Ga_\xi)$, the valency of $v$ is the cardinality of $E_v$: $\mathrm{val}(v)= |E_v|$.
Let $n_v = |E_v\cup S_v|$.

For $v\in V_\infty^S(\Ga_\xi)$, we decorate elements in $E_v\cup S_v$ as follows.
For $a\in S_v$, we decorate it by the same $\gamma_a$ that decorates $a\in L(\Ga)$;
for $e\in E_v$, we decorate it by $\gamma_{(e,v)}:=e^{-2\pi\sqrt{-1}d_e} \in \mufive$, ($d_e=\deg\sL|_{\sC_e}$,)
which is the monodromy
of $\sL|_{\sC_v}$ along $y_{(e,v)}=\sC_e\cap \sC_v$.
Define
\beq\label{gammav}
\gamma_v=\{\gamma_a: a\in S_v\}\cup\{\gamma_{(e,v)}: e\in E_v\}.
\eeq
Elements in $\gamma_v$ are indexed by $S_v\cup E_v$.  We form
\beq\label{VV0}
V^{a,b} (\Ga_\xi)=\{v\in V(\Ga_\xi)-V^S(\Ga_\xi)\, |\,  |S_v|=a, |E_v|=b\}.
\eeq
%
We also abbreviate $V_c^{a,b}(\Ga_\xi)=V_c(\Ga_\xi)\cap V^{a,b}(\Ga_\xi)$. Let
\beq\label{Delta}
\Delta\lggd=\{
\Ga_\xi \mid \xi\in\fixW\}/\cong,\quad \Del=\coprod\Del\lggd,
\eeq
where $\cong$ is the equivalence induced by isomorphisms of graphs preserving datum in Definitions
\ref{graph1} and \ref{graph2}, including the partition $V(\ga)=V_0(\ga)\cup V_1(\ga)\cup V_\infty(\ga)$, etc..
The disjoint union is taken over all possible $(g,\gamma,\bd)$.


\subsection{Decompositions of torus fixed MSP-fields}
\def\inn{\mathrm{inn}}
\def\out{\mathrm{out}}
Let $\xi \in \fixW$.
We will characterize the structure of $\sC_a$ and $\xi|_{\sC_a}$ for each $a\in V^S(\Ga_\xi)\cup E(\Ga_\xi)$.

\begin{rema}\label{weights}
If  $p\in \sC_0$, as $\varphi|_p$ and $\nu_2|_p$ are nonzero, then
$\sL|_p\cong \sN|_p\cong \bL_0$. If  $p\in \sC_1$, then $\nu_1|_p$ and $\nu_2|_p$ are nonzero and hence
$\sL\otimes\sN\otimes\bL_1|_p\cong \sN|_p\cong \bL_0$.  If  $p\in \sC_\infty$, then $\rho|_p$ and
$\nu_1|_p$ are nonzero and thus  $\sL^{\vee\otimes 5}\otimes\omega_\sC^{\log}|_p\cong \sL\otimes\sN\otimes\bL_1|_p\cong \bL_0$.
\end{rema}

\begin{exam} [Stable maps with $p$-fields]\label{stablep}
A stable MSP-field $\xi\in\cW$ having $\nu_1=0$ will have $\sN\cong\sO_\sC$, $\nu_2=1$, and
then $\xi$ 
is a pair of a stable map $f=[\varphi]: \Si^\sC\sub \sC\to\Pf$
and a $p$-field $\rho\in H^0(f\sta \sO_\Pf(-5)\otimes\omega_\sC)$.
\end{exam}

\begin{exam} [$5$-spin curves with five $p$-fields]\label{spinp}
A stable MSP-field $\xi\in\cW$ having $\nu_2=0$ will have $\sN\cong\sL\dual$, $\nu_1=1$ and
then $\xi$ is a pair of a $5$-spin curve $(\Si^\sC,\sC, \rho: \sL^{\otimes 5}\cong \omega_\sC^{\log})$
and five $p$-fields $\varphi\in H^0(\sL)^{\oplus 5}$, with $\varphi|_{\Si^\sC\lophi}=0$.
\end{exam}
\begin{lemm}\label{fix1}
Let $\xi\in \fixW$, and suppose $E_{0\infty}(\Ga_\xi)=\emptyset$.  
\begin{itemize}
\item[(a)] For $v\in V^S_0(\Ga_\xi)$, the sections $(\varphi_1,\cdots,\varphi_5)|_{\sC_v}$ define a morphism
$f_v: \sC_v \to \PP^4$; together with the marking $\Si^{\sC_v}=\Si^{\sC_v}_\inn\cup\Si^{\sC_v}_\out$
and $\rho_v=\rho|_{\sC_v}$, they
form an $E_v\cup S_v$-pointed, genus $g_v$, degree $d_v$ stable map to $\PP^4$ with p-field.
\item[(b)] For $v\in V_1^S(\Ga)$, 
$\Si^{\sC_v}\sub \sC_v$ is an $E_v\cup S_v$-pointed, genus $g_v$ stable curve.
\item[(c)] For $v\in V_\infty^S(\Ga)$, let $\gamma_v=
\{\gamma_i: i\in S_v\}\cup \{\gamma_{(e,v)}:e\in E_v\}$. Suppose all $\gamma_i$ and $\gamma_{(e,v)}$ in
$\gamma_v$ lie in $\bmu_5-\{1\}$, then $(\Si^{\sC_v}, \sC_v)$ with
$(\sL,\rho,\varphi)|_{\sC_v}$ is a $5$-spin $E_v\cup S_v$-pointed twisted curve with five $p$-fields.
\end{itemize}
\end{lemm}

\begin{proof}
The proof is straightforward, and will be omitted. Note that $E_{0\infty}(\Ga_\xi)=\emptyset$
guarantees that $\varphi|_{y_{(e,v)}}=0$ in case (c) and $\rho|_{y_{(e,v)}}=0$ in case (a).
\end{proof}
We remark that without assuming $E_{0\infty}(\ga_\xi)=\emptyset$, then in case (a) we may not have
$\rho_v|_{q_{(e,v)}}=0$ when $e\in E_{0\infty}(\ga_\xi)$ (here $q_{(e,v)}$ is the node associate to the flag
$(e,v)$); and in case (c) we may not have $\varphi|_{q_{(e,v)}}=0$, either.

We now move to $e\in E_0(\Ga_\xi)$. By Lemma \ref{C-structure}
we know $\sC_e\cong\PP^1$.
Let $v\in V_0(\Gamma_\xi)$ and $v'\in V_1(\Gamma_\xi)$ with  $(e,v)$ and $(e,v')\in F(\Gamma)$.
Let $q=\sC_e\cap\sC_v$ and $q'=\sC_e\cap \sC_{v'}$. Give $\sC_e$ the  coordinate
$[x,y]$ so that $[1,0]=q$ and $[0,1]=q'$.
Recall $d_e=d_{0e}-d_{\infty e}$.

\begin{lemm}\label{fix3}
Let the notations be as stated.
Then $\sN|_{\sC_e}\cong\cO_{\PP^1}$,  $\sL|_{\sC_e}\cong\cO_{\PP^1}(d_e)$,
and for some $(c_1,\ldots,c_5)\in \CC^5-0$ and
$c\in \CC^\times$,
$$(\varphi_1,\cdots, \varphi_5;\rho;\nu_1,\nu_2)|_{\sC_e}=(c_1x^{d_e},\cdots,c_5x^{d_e};0; cy^{d_e},1).
$$
\end{lemm}

\begin{proof}
Because $\deg\sL|_{\sC_e}=d_e$, $\sL|_{\sC_e}\cong \sO_{\Po}(d_e)$.
Since $q\in\sC_0$, the linearization $\tau_t|_q=\id$. Thus the induced $\Gm$-linearization on
$H^0(\sL|_{\sC_e})\cong H^0(\sO_{\Po}(d_e))$ is via $t\cdot x^{d_e-i}y^i = t^{ik} x^{d_e-i}  y^i$,
for some $k\ne 0$,
and $H^0(\sO_{\Po}(d_e))^\Gm=\CC\cdot x^{d_e}$.
Since $\varphi_i|_{\sC_e}$ are $\Gm$-invariant,
and $\nu_1|_{\sC_e}$ is $\Gm$-equivariant (of certain weight) and non-vanishing at $q'$,
there are $c,  c_i\in\CC$ so that  $\phi_i|_{\sC_e} = c_i x^{d_e}$ and $\nu_1|_{\sC_e}= c y^{d_e}$ .
As $\nu_1(q')\ne 0$, $c\ne 0$; as $\varphi(q)\ne 0$, $(c_1,\ldots, c_5)\neq 0$. This proves the lemma.
\end{proof}

We next consider $e\in E_\infty(\Ga_\xi)$. By Lemma \ref{fix3} we know that $\sC_e$ is a smooth rational curve.
It is isomorphic to $\PP^1$ when $d_{e}\in\ZZ$, and isomorphic to $\PP(5,1)$, and has
$\sC_e\cap\sC_\infty$ its only stacky point when $d_e\not\in\ZZ$.
Since $\nu_1|_{\sC_e}$ is nowhere vanishing,  any linearization on $\sL|_{\sC_e}$
will induce a linearization on $\sN|_{\sC_e}$ keeping $\nu_1|_{\sC_e}$ a $\Gm$-invariant
section of $\sL\otimes\sN|_{\sC_e}\otimes\bL_1$.

\vsp
\noindent
{\bf Convention on $\PP(5,1)$.}
{\sl Let $(\sX,q,q')=(\PP(5,1),[1,0],[0,1])$,
with  coordinates $[\ti x, \ti y]$, where $q=[1,0]$ is its only stacky point. And
$x=\ti x$ and $y=\ti y^5$ descend to its coarse moduli $X=\Po$. 
For $c\in \ofth\ZZ$,
we agree $\sO_{\sX}(c)\cong\sO_{\sX}(5c\,q)$ and $\deg \sO_{\sX}(q)=1/5$.  Thus $H^0(\sO_{\sX}(c))\cong H^0(\sO_X(\lfloor c\rfloor))$.
In the notation of $\sM_m$ before \eqref{Adirect}, near $q$ $\sO_{\sX}(c)$ is the module $\ti y^{-m}\CC[\ti y]$,
where $m=5c-5\lfloor c\rfloor$.}
\vsp

Let 
$v\in V_\infty(\Ga_\xi)$ and $v'\in V_1(\Ga_\xi)$ be vertices so that
$(e,v)$ and $(e,v')\in F(\Ga_\xi)$, and  $q=\sC_e\cap\sC_v$ and $q'=\sC_e\cap \sC_{v'}$.
We endow $\sC_e$ a coordinate $[x,y]$ (resp. $[\ti x,\ti y]$) so that
$q=[1,0]$ and $q'=[0,1]$ when $\sC_e\cong\PP^1$ (resp. $\PP(5,1)$).

We set $\delta=-1$ when $\dim\sC_v=0$, $S_v= \emptyset$ and $|E_v|=1$ (equivalently  $v\in V_\infty^{0,1}(\Ga_\xi)$),
and $\delta=0$ otherwise, and set
$\delta'=-1$ when $v'\in V_1^{0,1}(\Ga_\xi)$, and $\delta'=0$ otherwise.
 Note that when $d_{e}=d_{0e}-d_{\infty e}\not\in\ZZ$, $\delta=0$. 

\begin{lemm}\label{Lem5.6}
Let $e\in E_\infty(\Ga_\xi)$ be as stated. Then $\varphi|_{\sC_e}=0$,
$\sN|_{\sC_e}\cong \sL\dual|_{\sC_e}$,
and $\nu_1|_{\sC_e}$ is never zero.
Furthermore, $\omega_\sC^{\log}|_{\sC_e}\cong
\sO_{\sC_e}(\delta')$, $\sL|_{\sC_e}\cong \sO_{\sC_e}(d_{e} )$, and 
$(\rho, \nu_2)|_{\sC_e}=(c x^{-5d_{e} +\delta+\delta'}, c'y^{\lfloor -d_{e}\rfloor})$ for some $c,c'\in\CC^\times$.
\end{lemm}



\begin{proof}
The part $\varphi|_{\sC_e}=0$, etc., is merely restating the known facts. We now consider
the case  $\sC_e\cong \Po$. Again, $\omega_\sC^{\log}|_{\sC_e}\cong
\sO_{\Po}(\delta +\delta')$ and $\sL|_{\sC_e}\cong \sO_{\Po}(d_e)$ follow from the definition.
By Remark \ref{weights}, 
we have $\sL^{\vee5}\otimes\omega_\sC^{\log}|_q\cong \bL_0\cong \sN|_{q'}$ as $T$-representations.  Thus
$$H^0(\sL^{\vee5}\otimes\omega_\sC^{\log}|_{\sC_e})^\Gm=\CC \cdot x^{-5d_e+\delta+\delta'}
\and
H^0(\sN|_{\sC_e})^\Gm=\CC\cdot y^{-d_e}.
$$
This proves the furthermore part of the lemma in this case. The remaining part of the furthermore has the similar argument. 

\end{proof}

\subsection{Flattening a graph}

In this subsection, we investigate if the graph $\ga_\xi$ remains unchanged when we deform
$\xi\in\cW^T$.

We begin with edges in $E_{0\infty}(\Ga_\xi)$.
Let $(e,v)$ and $(e,v')\in F(\Ga_\xi)$ with $v\in V_\infty(\Ga_\xi)$
and $v'\in V_0(\Ga_\xi)$, thus $e\in E_{0\infty}(\Ga_\xi)$. Let $q=\sC_e\cap \sC_v$, and $q'=\sC_e\cap \sC_{v'}$.

\begin{lemm}
Let $e\in E_{0\infty}(\Ga_\xi)$. Then $\sC_e\cong \Po$,
$\sL|_{\sC_e}\cong \omega^{\log}_\sC|_{\sC_e}\cong\sO_{\sC_e}$,
and $\sN|_{\sC_e}\cong \sO_{\sC_e}(d_{\infty e})$.
\end{lemm}

\begin{proof}
Since $\varphi(q')\ne 0$, then $\deg\sL|_{\sC_e}\ge 0$.
Since $\rho(q)\ne 0$, we have
$\deg \sL^{\vee\otimes 5}\otimes \omega_\sC^{\log}|_{\sC_e}\ge 0$. Furthermore,
since $\deg \omega_\sC^{\log}|_{\sC_e}\le 0$, we have $\deg\sL|_{\sC_e}=\deg \omega_\sC^{\log}|_{\sC_e}=0$.
This proves the lemma.
%
\end{proof}

We introduce the notion of $\Gm$-balanced nodes.

\begin{defi}
Let $\sC$ be a $\Gm$-twisted curve and $q$ be a node of $\sC$. Let $\hat \sC_1$ and $\hat \sC_2$ be the
two branches of
the formal completion of $\sC$ along $q$. We call $q$ $\Gm$-balanced 
if $T_q\hat \sC_1\otimes T_q\hat \sC_2\cong \bL_0$ as $\Gm$-representations.
\end{defi}


Let $v\in V_1(\Ga_\xi)$ be an unstable vertex with two distinct $(e,v)$ and $(e',v)\in F(\Ga_\xi)$,
and let $q_v=\sC_e\cap \sC_{e'}$ be the node of $\sC$ associated with $v$.

\begin{lemm} \label{unstable-q}
Let $v \in V_1(\Ga_\xi)$ be an unstable vertex, etc., as stated.
Then $q_v$ is $\Gm$-balanced
if and only if $d_{e}+d_{e'}=0$, and
$(\sC_e\cup\sC_{e'})\cap\sC_\infty$ is a node or a marking.
\end{lemm}

\begin{proof}
We first consider the case where $e\in E_\infty(\Ga_\xi)$ and $e'\in E_0(\Ga_\xi)$.
Let $q=\sC_e\cap\sC_\infty$ and $q'=\sC_{e'}\cap\sC_0$.
By Remark \ref{weights},
\beq\label{weight0}
\sL|_{q'}\cong  \sL^{\vee\otimes 5}\otimes\omega_\sC^{\log}|_q\cong \bL_0.
\eeq

Let $\delta=0$ when $q$ is a node or a marking of $\sC$, and $=-1$ otherwise. Then since
$q_v$ is a node, $\omega_\sC^{\log}|_{q_v}\cong \bL_0$,  and thus
$\sL^{\vee\otimes 5}|_{q_v}\cong \sL^{\vee\otimes 5}\otimes\omega_\sC^{\log}|_{q_v}$ as $T$-representations.
Since $q_v$ is a scheme point, and since $\sL^{\vee\otimes 5}\otimes\omega_\sC^{\log}|_{\sC_e}$ is a
pullback of an invertible sheaf on the coarse moduli of $\sC_e$, combined with \eqref{weight0}, we conclude that
$q_v$ is a $\Gm$-balanced
node if and only if
$$\deg \sL^{\vee\otimes 5}\otimes\omega_\sC^{\log}|_{\sC_e}+\deg\sL^{\vee\otimes 5}|_{\sC_{e'}}=0,
$$
which is $-5d_{e}+\delta-5d_{e'}=0$. We claim that then $\delta=0$ and $d_{e}+d_{e'}=0$.
Indeed, if $\delta=-1$, then $\sC_e$ is a scheme and $d_{e}\in\ZZ$ (and $d_{e'}\in\ZZ$), impossible.
Thus $\delta=0$ and the lemma follows.

The other cases where both $e$ and $e'$ lie in $E_0(\Ga_\xi)$ or in $E_\infty(\Ga_\xi)$ can be ruled out
 by a similar argument. This proves the lemma.
\end{proof}

Let
\beq\label{NGa}
N(\Ga_\xi)=V^{0,2}(\Gamma_\xi)\cup\{(e,v)\in F(\Ga_\xi)\mid v\in V^S(\Ga_\xi)\}.
\eeq
Note that every $a\in N(\Ga_\xi)$ has its associated node $q_a$ of $\sC$.

\begin{defi}\label{def5.12}
We call $a\in N(\Ga_\xi)$ $\Gm$-balanced if 
the associated node $q_a$ is a $\Gm$-balanced node in $\sC$.
Let $N(\ga_\xi)\un\sub N(\Ga_\xi)$ be the subset of $T$-unbalanced elements.
\end{defi}

\begin{prop}\label{un}
Let $\xi\in\fixW$. Then $q_a$,  for $a\in N(\ga_\xi)$, is $T$-balanced if and only if
$a\in V^{0,2}_1(\Ga_\xi)$ and satisfies the condition in Lemma \ref{unstable-q}. Furthermore, the set of $T$-unbalanced nodes of $\sC$ is
$\{q_a\mid a\in N(\ga_\xi)\un\}$.
\end{prop}

\begin{proof} Let $q_a$ be a node indexed by $a\in N(\Ga_\xi)$. Then a repetition of the proof of
Lemma \ref{unstable-q} shows that $q_a$ is $T$-balanced only if $a\in V^{0,2}_1(\Ga_\xi)$, and satisfies the
criterion in Lemma \ref{unstable-q}.

To complete the proof, we need to show that any node $q$ of $\sC$ not listed in $\{q_a \, |\, a\in N(\Ga_\xi)\}$,
must be $T$-balanced. Indeed, such $q$ must be a node of one of $\sC_0$, $\sC_1$ and $\sC_\infty$.
Since $T$ acts trivially on these curves, $q$ must be a balanced node. This proves the proposition.
\end{proof}

Although  a $\Gm$-balanced $a\in N(\Ga_\xi)$    is characterized by $q_a$ being $\Gm$-balanced,
by Lemma \ref{unstable-q} it 
can also be characterized by the information of the (decorated) graph $\Ga_\xi$.
Thus for any $\ga\in\Del$, we can talk about $N(\ga)\un\sub N(\ga)$.

\begin{defi} A decorated graph $\Gamma$ is called a flat graph if $N(\ga)\un= N(\ga)$.
\end{defi}

For reasons to be clear in the next subsection,
we will construct a flattening of $\Ga_\xi$ when $N(\Ga_\xi)$ contains $T$-balanced nodes.

\begin{cons}[Flattening]
For each $\Gm$-balanced $v\in N(\Ga_\xi)$, which necessarily is an unstable vertex in $V_1(\Ga_\xi)$,
we eliminate the vertex $v$ from $\Ga_\xi$,  replace the two edges $e\in E_\infty(\Ga_\xi)$ and $e'\in E_0(\Ga_\xi)$
incident to $v$ by
a single edge $\ti e$ incident to the other two  vertices that are incident to $e$ and $e'$,
and demand that $\ti e$ lies in $E_{0\infty}$. We decoration $\bar e$ via $\vec g(\ti e)=0$,
$(d_{0\ti e}, d_{\infty\ti e})=(d_{\infty e}, d_{\infty e})$, while keeping the rest unchanged.
The resulting decorated graph is called the flattening of $\ga_\xi$ at $a$.
\end{cons}

Let $\Ga_\xi^{\mathrm{fl}}$ be the graph
after applying this to all $\Gm$-balanced $v$ in
$N(\ga_\xi)$.
We call it the flattening of $\Ga_\xi$. As the result, $N(\Ga_\xi^{\mathrm{fl}})= N(\Ga_\xi)\un$.
Define
$$\Delta^{\mathrm{fl}}\lggd\defeq \{\Ga^{\mathrm{fl}} \mid \ga\in\Del\lggd\}/\cong,\quad
\Delta^{\mathrm{fl}}=\bigcup \Delta^{\mathrm{fl}}\lggd.
$$
Here $\cong$ is as that in \eqref{Delta}.

\subsection{$\Ga$-framed MSP fields}\label{subsection2.4}
\def\autxi{\cal Aut_{(\xi,\eps)}}



\begin{defi}Let $S$ be a scheme and let $\sC_S$ be a flat $S$-family of twisted curves.
We say that $\sC_S$ can be decomposed along nodes $\sQ_S\sub \sC_S$ if
$\sQ_S$ is a closed substack of $\sC_S$ that is a gerbe over $S$ and homeomorphic to $S$,
so that
there are a flat $S$-family of twisted curves $\ti\sC_S$, an
$S$-morphism $\ti\sC_S\to\sC_S$,  and two
disjoint closed $S$-embeddings $h_1$, $h_2:\sQ_S\to\ti\sC_S$ so that
$$\hom_S(\sC_S,\cdot)=\{f\in \hom_S(\ti\sC_S,\cdot)\mid h_1\circ f\approx h_2\circ f\}.
$$
(See \cite[A.1]{AGV} for precise formulation of $\approx$.)
If $\sC_S$ can be decomposed along $\sQ_S$, we call $\sQ_S$ an $S$-family of nodes.
\end{defi}

If $\sC$ is a twisted curve with only one node $q\in \sC$, then the decomposition of $\sC$ along $q$ is
the normalization $\ti\sC$ of $\sC$, and $h_1$, $h_2: q\to \ti\sC$ are the two preimages of $q$ via the projection
$\ti\sC\to \sC$. We can decompose along several disjoint $S$-families of nodes by applying this procedure repeatedly
to each $S$-family of nodes.

We now assume that $\sC_S$ is an $S$-family of $\Gm$-curves, meaning that $\Gm$ acts
on $\sC_S$ and  commutes  with the trivial $\Gm$-action on $S$.
We have the following well-known existence of decomposition along families of nodes.

\begin{lemm}\label{stable-node}
Let $s\in S$ be a closed point in a DM stack, 
let $\sC_S$ be a flat $S$-family of nodal twisted $\Gm$-curves, and
let $q\in \sC_s$ be a $\Gm$-unbalanced node of $\sC_s$. 
Then there is an \'etale neighborhood $\ti s\in \ti S\to S$, $\ti s\mapsto s$, so that
we can extend $\ti q=q\times_{\ti S}\ti s\in \sC_s|_{\ti s}$ to an $\ti S$-family of nodes $\sQ_{\ti S}\sub \sC_S|_{\ti S}$.
\end{lemm}

A simple observation shows that such decomposition may not exists for $T$-balanced nodes.
%
%

\begin{defi}[Standard decomposition]\label{standard}
For any $\xi\in\fixW$ with the domain curve $\sC$, we can decompose $\sC$ along its $T$-unbalanced nodes
$q_a$, $a\in N(\ga_\xi)\un=N(\ga_\xi^\fl)$, to obtain connected subcurves $\sC_b$ indexed by
$$b\in \Xi(\ga_\xi^\fl)\defeq V^S(\Ga_\xi^\fl)\cup E(\Ga_\xi^\fl).
$$
\end{defi}

\begin{defi}
Let $\vGa\in\Delta^\fl$; let $S$ be a scheme.
A $\Ga$-framed $S$-family in $\fixW$ consists of a pair $(\xi_S, \eps_S)$, where
\beq\label{G-frame}
\xi_S\in \fixW(S) \and 
\eps_S=\{\eps_s: \Ga\cong \Ga_{\xi_s}^{\mathrm{fl}} \mid s\in S\ \text{ a closed point}\},
\eeq
such that if for each closed point $s\in S$ we let $q_{s,a}=q_{s,\eps_s(a)}$ for $a\in N(\Ga)$ and $\sC_{s,b}
=\sC_{s,\eps_s(b)}$
 for $b\in \Xi(\Ga)$  be the result of the standard decomposition of $\xi_S\times_S s=(\sC_s,\cdots)$,
then the unions
$\coprod_{s\in S} q_{s,a}$ and $ \coprod_{s\in S} \sC_{s,b}$
are closed substacks of the domain curve $\sC_S$ of $\xi_S$, and are flat over $S\lred$.
\end{defi}

Applying Lemma \ref{stable-node}, we see that given a connected scheme and any $\Ga$-framed $S$-family in $\fixW$
with $\sC_S$ its total space of base curves,
we have an $S$-family of nodes $\cQ_{S,a}\sub\sC_S$ indexed by $a\in N(\Ga)$ so that they decompose
$\sC_S$ into $S$-flat families of connected subcurves $\sC_{S,b}$ indexed by $b\in \Xi(\Ga)$. (Here
$\ga=\ga^\fl$ automatically when we speak of $\Ga$-framed families.)
Let $\Ga\in \Delta^{\mathrm{fl}}$,  and let
$\Wfix$ be the groupoid of $\Ga$-framed families in $\fixW$ with obviously defined arrows. Then $\Wfix$ is a
DM-stack, finite over $\fixW$ via the forgetful morphism $\Wfix\to \fixW$.

\begin{prop}\label{open-closed}
Let $\cW_{(\Ga)}$ be the image of the forgetful morphism $\iota_\Ga:\Wfix\to\fixW$.
Then $\cW_{(\Ga)}$ is an open and closed substack of $\fixW$. Let $\iota_\Ga$ be factored as 
\beq\label{ij}
\iota_\Ga: \Wfix\mapright{\imath_\ga} \cW\lrga\mapright{\jmath_\ga} \fixW.
\eeq
Then $\imath_\Ga$ is an $\Aut(\Ga)$-torsor.
\end{prop}

\begin{proof}
That $\cW_{(\Ga)}$ is open and closed follows from lemma \ref{stable-node} and the definition of $\Ga$-framed
MSP fields, the other parts
are straightforward to prove, and will be omitted.
\end{proof}

We end this subsection with the notion of $\Ga$-framed curves $(\sC,\Si^\sC,\sL,\sN)$.
Recall that given a flat $\Ga$ and a $(\xi,\eps)\in \cW_\ga$, where $\xi=(\sC,\cdots)$, etc.,
 we not only have an identification of the $T$-unbalanced nodes of $\sC$ with $N(\ga)$,
 but also have a tautological identification of branches of these $T$-unbalanced nodes:
Let $q\in \sC$ be a $T$-unbalanced node. Then we have
\begin{itemize}
\item[(I)] when $q$ is identified with an $(e,v)\in N(\ga)$, $v\in V^S(\ga)$, then the two branches of $\sC$ along $q$
are labeled by $v$ and $e$;
\item[(II)]
when $q$ is identified with a $v\in V^U(\ga)$, then the two branches of $\sC$ along $q$
are labeled by $(e_1,v)$ and $(e_2,v)\in F(\ga)$.
\end{itemize}


\begin{defi}\label{cDga}
A $\Ga$-framed curve consists of $T$-equivariant $(\sC,\Si^\sC,\sL,\sN)$\footnote{By
$T$-equivariant we mean that $\Si^\sC\sub \sC$ comes with a $T$-action,
and both $\sL$ and $\sN$ are $T$-linearized.} such that
\begin{enumerate}
\item $\Si^\sC$ is labeled by legs in $\Ga$;
\item $T$-unbalanced nodes of $\sC$ are labeled by $N(\ga)$;
\item branches of $T$-unbalanced nodes of $\sC$ are labeled according to rule {\rm I} and {\rm II}.
\end{enumerate}
\end{defi}

Because the conditions in Definition \ref{cDga} are open, we can speak of flat family of $\Ga$-framed curves. Let
$\cD_\ga$ be the stack of flat families of $\Ga$-framed curves, where arrows are $T$-equivariant arrows in
$\cD$ that preserve the labeling in Definition \ref{cDga}.  Here  $\cD$ is the stack of data $(\sC,\Si^\sC,\sL,\sN)$, where
$\Si^\sC\sub \sC$ are pointed twisted curves, and $\sL$ and $\sN$ are invertible  sheaves on $\sC$.
Clearly, $\cD_\ga$ is a smooth Artin stack, and admits a forgetful morphism
$\cW_\ga\to\cD_\ga$.

\subsection{Decomposition via decoupling}
\def\gadconn{(\Ga^{\mathrm{d}})_{\mathrm{conn}}}
\def\gad{\Ga^{\mathrm{d}}}
\def\gadeconn{(\Ga^{\mathrm{de}})_{\mathrm{conn}}}
\def\gade{\Ga^{\mathrm{de}}}

We introduce the operation {\sl decoupling}
to a flat graph $\Ga$ to simplify the stack $\cW_\ga$.

Given an edge $e$ in $\ga$, we say a vertex $v$ of $e$ is a connecting vertex if either $v$ is stable or
$v$ is a two-valent vertex. We say a vertex $v$ of $e$ an end vertex if $v$ is a one-valent unstable
vertex. When $e$ has an end vertex (and one connecting vertex), we call $e$ a leaf edge of $\ga$.

\begin{defi}[Decoupling]
Let $e\in E_0(\ga)\cup E_\infty(\ga)$ be a non-leaf edge and $v\in V_1(\ga)$ be the (connecting)
vertex of $e$ in $V_1(\ga)$.
We define the decoupling of $\ga$ along $a=(e,v)$ be
the new graph $\Ga'$ after 
\begin{enumerate}
\item adding a new vertex $v'$ to $\Ga$ and replace the flag $(e,v)$ by $(e,v')$;
\item adding a leg $l$ to $v$ and a leg $l'$ to $v'$, both decorated by $(1,\varphi)$.
\end{enumerate}
\end{defi}

In geometric term, In case $\xi=(\sC,\cdots)\in \cW_\ga$, then the decoupling corresponds to
normalizing $\sC$ along the node associate to $a$, and adding two markings to the two new smooth
points, as specified in the definition.

\smallskip

Let $\ga$ be a flat graph.
We apply decoupling repeatedly to all pairs $(e,v)$ of non-leaves $e\in E_0(\ga)\cup E_\infty(\ga)$ and vertices $v$ of
$e$ in $V_1(\ga)$.
Let $\gad$ be the resulting graph and denote by
$\gadconn$ the set of connected components of $\gad$.

For each $A\in \gadconn$, we form the moduli $\cW_A$ of $A$-framed MSP fields. Let
$\cW_{\gad}$ be the moduli of $\gad$-framed MSP fields. Note that though $\gad$ could be disconnected,
the construction $\cW_\ga$ applies and $\cW_{\gad}$ is well-defined, resulting a stack naturally isomorphic
to the product of $\cW_A$ of all $A\in\gadconn$.

\begin{prop}\label{Phi0}
For each $A\in\gadconn$, the decoupling defines a natural morphism
$\Phi_A: \cW_\ga\lra \cW_A$, whose product defines an isomorphism
$$\Phi=\prod \Phi_A: \cW_\ga\lra \cW_{\gad}=\prod_{A\in \gadconn} \cW_A.
$$
\end{prop}

\begin{proof} The proof is a direct verification. 
\end{proof}

\subsection{Decomposition via trimming}
We will use the operation {\sl trimming} to further simplify
the graph $\gad$.

\begin{defi}[Trimming]
Let $e\in E_0(\ga)\cup E_\infty(\ga)$ be a leaf edge so that its connecting vertex $v$ is stable.
We define the trimming of $e$ from $\ga$ to be
the new graph $\Ga'$ after replacing the edge $e$ by a leg $l$ attached to the same vertex $v$,
and decorating $l$ via the rule; 
\begin{enumerate}
\item when $v\in V_\infty(\Ga)$ and $d_{e,0}\not\in\ZZ$, decorating $l$ by $\gamma_{(e,v})\upmo$;
\item when $v\in V_\infty(\Ga)$ and $d_{e,0}\in\ZZ$,
decorating $l$ by $(1,\varphi)$;
\item when $v\in V_0(\Ga)$, decorating $l$ by $(1,\rho)$;
\item when $v\in V_1(\Ga)$, decorating $l$ by $(1,\varphi)$;
\end{enumerate}
\end{defi}

We apply trimming to leaf edges in $\gad$ whose connecting vertices are stable. 
We denote the resulting graph by $\gade$. We will have a similar morphism as in Proposition \ref{Phi0}.
For precise statement of such morphism, we need to investigate $T$-invariants MSP associated to $e\in E_0(\ga)
\cup E_\infty(\ga)$.

Let $e\in E_0(\Ga)$, and let $\cW_e$ be the stack parameterizing families of
\beq\label{ee1}
(c_1 x^{d_{e}}, \cdots, c_5 x^{d_{e}}, c y^{d_{e}})\in H^0(\sO_{\Po}(d_{e}))^{\oplus 6},
\ (c_1,\cdots, c_5)\ne 0, \, c\ne 0,
\eeq
where $[x,y]$ are coordinates of $\Po$.  An arrow from $(c_i x^{d_{e}},c y^{d_{e}})$ to
$(c_i'  x^{d_{e}},c'y^{d_{e}})$
consists of an isomorphism $\sigma: \Po\to\Po$ fixing $[1,0]$ and $[0,1]$, and an isomorphism
$\sO_{\Po}(d_{e})\cong \sigma\sta\sO_{\Po}(d_{e})$ that identifies $c_i x^{d_{e}}$ with $c_i '  x^{d_{e}}$,
and identifies $cy^{d_{e}}$ with $c'   y^{d_{e}}$.

\begin{lemm}\label{fix4}
For $e\in E_0(\Ga)$, we have the following morphism and isomorphism
defined by Lemma \ref{fix3}:
$$\Phi_e: \Wfix\lra \cW_e\cong \sqrt[d_e]{\cO_{\PP^4}(1)/\PP^4}.
$$
\end{lemm}

\begin{proof}
Firstly, we prove the isomorphism. 
We form $[(\CC^5 \setminus 0)/\gm]$, where $\alpha\in \gm$ acts on $\CC^5 \setminus 0$ by
$\alpha \cdot (c_1,\ldots, c_5) = (\alpha^{d_e} c_1,\ldots, \alpha^{d_e} c_5)$.
We first show that the map sending \eqref{ee1} to
$[c\upmo c_1,\cdots, c\upmo c_5]\in [(\CC^5 \setminus 0)/\gm]$
is well-defined. Firstly, it is independent of the choice of the isomorphism
$\sL|_{\sC_e}\cong \sO_{\Po}(d_e)$. Indeed, different isomorphisms will scale both $c$ and $(c_1,\ldots, c_5)$
by a non-zero constant leaving $(c\upmo c_1,\cdots, c\upmo c_5)$ invariant.
On the other hand, given an automorphism of $\Po$ of the form $\alpha\cdot [x,y]=[\alpha\upmo x, y]$,
then
$$\alpha\cdot(c_1 x^{d_{e}}, \ldots, c_5 x^{d_{e}}, c y^{d_{e}})=
(\alpha^{-d_e} c_1 x^{d_{e}}, \ldots, \alpha^{-d_e} c_5 x^{d_{e}}, c y^{d_{e}}).
$$
This proves that the map $\cW_e\to [(\CC^5 \setminus 0)/\gm]$ is well-defined and is an isomorphism. Using
$$
[(\CC^5 \setminus 0)/\gm] \cong \sqrt[d_e]{\cO_{\PP^4}(1)/\PP^4},
$$
we prove the isomorphism part of the Lemma.

We now define the morphism $\Phi_e$. Let $(\xi_S,\eps_S)$ be an $S$-family in $\Wfix$, and  let $s\in S$
be a closed point. Denote $\xi_S=(\sC_S,\Si^{\sC_S},\sL_S,\cdots)$.
Following the notation developed before and in the proof of Lemma \ref{fix3},
let $q=\sC_{s,e}\cap(\sC_s)_0$ and $q'=\sC_{s,e}\cap(\sC_s)_1$. Consider the case where  $q$ and $q'$  are the two nodes of $\sC_{s}$
 in $\sC_{s,e}$($=(\sC_s)_e$). The proof for other cases is similar.  As both $q$ and $q'$ are $\gm$-unbalanced,
and because $(\xi_S,\eps_S)$ is a $\Gamma$-framed MSP field,
$q$ and $q'$ extend to $S$-families of nodes $\sQ_S$,
$\sQ_S'\sub \sC_S$ containing $q$ and $q'$ respectively, which decompose $\sC_S$ along
$\sQ_S$ and $\sQ'_S$ into two families of curves, one of which is
a family of rational curves $\sC_{S,e}$ indexed by $e\in E_0(\Ga)$
with two sections $\sQ_S$ and $\sQ'_S\sub \sC_{S,e}$. By shrinking $S$ with $s\in S$ if necessary, we can find an
$S$-isomorphism
$\phi:\Po\times S\cong \sC_{S,e}$ so that $\phi\upmo( \sQ_S)=[1,0]\times S$ and
$\phi\upmo(\sQ_S')=[0,1]\times S$. We then fix an isomorphism
$\phi\sta \sL_S|_{\sC_{S,e}}\cong \pi_{\Po}\sta\sO_{\Po}(d_e)$, and, by Lemma \ref{fix3}, express
\beq\label{ee7}
\phi\sta \varphi_{i,S}|_{\sC_{S,e}}=c_i x^{d_{e}},\quad \phi\sta \nu_{1,S}|_{\sC_{S,e}}=c y^{d_{e}},
\eeq
where $c_i, c$ are regular functions on $S$ so that $(c_1,\cdots, c_5)$ and $c$ are nowhere vanishing.
The $S$-family \eqref{ee7} defines a morphism $S\to\cW_e$. Clearly, different choices of
isomorphisms $\phi$ and isomorphism $\phi\sta \sL_S|_{\sC_{S,e}}\cong \pi_{\Po}\sta\sO_{\Po}(d_e)$
produce another $S\to\cW_e$ that differs from the previous $S\to \cW_e$ by an arrow between them. Thus, we obtain a morphism $\Phi_e(\xi_S,\eps_S): S\to \cW_e$,
which patches to form a morphism $\Phi_e:\Wfix\to\cW_e$, as desired.\end{proof}
Let $e\in E_\infty(\Ga)$. Recall the statement before Lemma \ref{Lem5.6},
  let $v\in V_\infty(\Ga)$ and $v'\in V_1(\Ga)$ be so that
$(e,v)$ and $(e,v')\in F(\Ga)$. Set $\delta=-1$ when $v\in V_\infty^{0,1}(\Ga)$, and $\delta=0$ otherwise, set
$\delta'=-1$ when $v'\in V_1^{0,1}(\Ga)$ and $\delta'=0$ otherwise. Note that when $d_e=d_{0e}-d_{\infty e}\not\in\ZZ$,
$\delta=0$.

We introduce $\cW_e$. In case $d_{\infty e}\not\in\ZZ$, let $(\sX,q,q')=(\PP(5,1),[1,0],[0,1])$, etc., be
as in and before Lemma \ref{Lem5.6}.
Let $\sN=\sO_{\sX}(-d_{e})$, and let $\omega=\sO_{\sX}(\delta' q')\cong\sO_{\sX}(\delta' ) $. (Note $\delta=0$ in this case.)
We define $\cW_e$ to be the stack parameterizing families of
\beq\label{ee3}
(c x^{-5d_{e}+\delta+\delta'}, c' y^{\lfloor -d_{e}\rfloor})\in H^0(\sN^{\otimes 5}\otimes\omega\oplus \sN),
\quad c,c'\in \CC^\times,
\eeq
where an arrow from $(c_1 x^{-5d_e +\delta'}, c'_1 y^{\lfloor-d_e\rfloor })$ to
$(c_2 x^{-5d_e +\delta'}, c'_2 y^{\lfloor-d_e\rfloor })$
consists of an isomorphism $\sigma_1: \sX\to\sX$ fixing $q$ and $q'$, and an isomorphism
$\sigma_2: \sN\cong \sigma_1\sta\sN$ which  together with the tautological isomorphism\footnote{Here
the tautological map $\sigma_1\sta\omega\cong \omega$ is
 consistent with the tautological inclusion $\omega\sub \sO_\sX$ and the isomorphism
$\sigma_1\sta\sO_\sX\cong \sO_\sX$ fixing the global section $1$.} $\sigma_1\sta\omega\cong\omega$
identifies $(c_1 x^{-5d_e +\delta'}, c'_1 y^{\lfloor-d_e\rfloor })$
with $(c_2 x^{-5d_e +\delta'}, c'_2 y^{\lfloor-d_e\rfloor })$.

When $d_{e}\in \ZZ$,   let $(\cX,q,q')=(\Po,[1,0],[0,1])$
under  coordinates $[x, y]$. Let
$\sN=\sO_{\cX}(-d_{e})$, and$\omega=\sO_{\cX}(\delta q+\delta' q')$.
Define $\cW_e$ to be the moduli stack parameterizing families of \eqref{ee3},
where arrows between objects are defined similarly.

\begin{lemm}\label{lem227}
Let $e\in E_\infty(\Ga_\xi)$. Then $\cW_e$ is a $\bmu_{-5d_{e}+\delta}$-gerbe over a point.
\end{lemm}

\begin{proof}
First, by scaling the line bundle $\sL$ and applying automorphisms to $\cX$ preserving $(q,q')$, we can
make $(c,c')=(1,1)$. Thus $\cW_e$ as a set consists of a single point.

Let $G_e$ be the automorphism group of the single point in $\cW_e$. Thus $\cW_e$ is a $G_e$-gerbe over a single point.
We now prove that $G_e\cong \bmu_{-5d_{e}+\delta}$. We only prove the case where $d_{e}\not\in\ZZ$  while
the other case is  similar.
Let $\sigma\in G_e$,  let $\sigma_1:\sX\to\sX$ be the automorphism part of $\sigma$,
and let $\sigma_2:\sigma_1\sta\sN\to\sN$ be the linearization part of $\sigma$.
Since $c'y^{\lfloor -d_{e}\rfloor}|_{q'}\ne 0$ and $\sigma_1$ fixes $q'$, $\sigma\in G_e$ implies that
$\sigma_2|_{q'}: \sN|_{q'}=\sigma_1\sta\sN|_{q'}\to \sN|_{q'}$ is the identity map. In particular, the map $G_e\to\Aut(\sX)$ is injective. Furthermore, let $\sigma_2': \sigma_1\sta\sO_\sX\to\sO_\sX$ be the $G_e$-linearization leaving
$1\in H^0(\sO_\sX)$ invariant, then the linearization $\sigma_2$ is that induced by a $G_e$-isomorphism
$\sN\cong \sO_X(-5d_{e}q)$, the linearization $\sigma_2'$ on $\sO_\sX$,  and the tautological inclusion
\beq\label{ttt}
\sO_\sX\sub \sO_\sX(-5d_{e}q)\cong \sN.
\eeq

Then since the meromorphic section $x^{\delta'}$ of $\omega$ is the meromorphic section $1$ in the footnote below
\eqref{ee3}, $x^{\delta'}$  is $\CC\sta$-equivariant. Thus
$x^{-5d_{e}+\delta'}\in H^0(\sN^{\otimes 5}\otimes\omega)$ is $G_e$-invariant
if and only if $x^{-5d_{e}}\in H^0(\sN^{\otimes 5})$ is $G_e$-invariant.
Using the $G_e$-equivariant homomorphisms in \eqref{ttt},  we see that
$G_e=\bmu_{-5d_{e}}\le \CC\sta$, which is $\bmu_{-5d_{e}+\delta}$ as $\delta=0$ in this case.
\end{proof}


\begin{coro}\label{C2.30}
The families given in Lemma \ref{Lem5.6} induce a morphism
$$\Phi_e: \Wfix \lra \cW_e\cong B\bmu_{-5d_{e}+\delta}.
$$
\end{coro}

We let $G_e=\bmu_{d_e}$ when $e\in E_0(\ga)$; let $G_e=\bmu_{-5d_e+\delta}$ when $e\in E_\infty(\ga)$.

\begin{rema}\label{dinf}
The proof of Lemma \ref{lem227} shows that for $e\in E_\infty(\Ga_\xi)$
with $d_{e}\not\in \ZZ$, a generator of the automorphisms of $\xi|_{\sC_e}$ takes the following form near its stacky point $q$:
Let $[\ti x,\ti y]$ be a homogeneous coordinate of $\sC_e$ with $q=[1,0]$ and $q'=[0,1]$,
let $\sU_e=\sC_e-q'$, and  let 
$b=\lceil -d_e\rceil$ and $m=5b+5d_e$. 
Then $\sL|_{\sU_e}\cong \ti y^{-m}\CC[y]$ 
(compare with that before \eqref{Adirect}), and
a generator of $\Aut(\xi|_{\sC_e})$ acts on $\sL|_{\sU_e}$ via
\beq\label{Ce-gen}
[a,c]\mapsto [\zeta_l \cdot a,  c]\in C_e\and 1\mapsto \zeta_l^b\cdot 1, \quad\hbox{where $\ell=-5d_e$}.
\eeq

\end{rema}

The remark can be seen as follows:
The action on $\sC_e$ is the same as its action on the coarse moduli $C_e$ as stated. Using
$\sL|_{\sC_e}\cong \sO_{\sC_e}((m-5b)q)$, we see that $\sL|_{\sU_e}\cong \ti y^{-m}\CC[y]$ is equivalent to
$\sO_{\sC_e}(-5bq)|_{\sU_e}\cong \CC[y]$. As $TC_e|_{[1,0]}\cong \bL_{-1}$\footnote{By this we mean
it is a weight $-1$ representation of $\bmu_l$.}, $TC_e|_{[0,1]}\cong \bL_1$,
and $\sL|_{q'}\cong \bL_0$ from the proof of Lemma \ref{lem227}, we get $\sO_{\sC_e}(-5bq)|_q\cong \bL_{b}$.
As $1\in\CC[y]$ in the stated isomorphism spans $\sO_{\sC_e}(-5bq)|_q$, we get
$1\mapsto \zeta_l^b\cdot 1$.
\vsp

\begin{prop}\label{Phi1}
For each $A\in\gadeconn$, the trimming defines a natural morphism
$\Phi_A: \cW_\gad\lra \cW_A$, whose product defines a morphism
$$\Phi=\prod \Phi_A: \cW_\gad\lra \cW_{\gade}=\prod_{A\in \gadconn} \cW_A,
$$
which is a $G$-gerbe, where $G$ is the product of $G_e$ of those $e\in E(\ga)$
that have been trimmed from $\gad$ to obtain $\gade$
\end{prop}

Before proving the proposition, we state simple facts on lifting automorphisms of a twisted curve to
their invertible sheaves.
We recall notations used in \cite[Subsect.\,2.1]{CLLL}
(drawn from \cite{ACV,  AGV, Cad}.)

First a model of an  index $r$ balanced node is
$$\sV_r\defeq \left[ \spec\bl \CC[u,v]/(uv)\br\big/ \bmu_r\right],\quad \zeta\cdot(u,v)=(\zeta u,\zeta\upmo v).
$$
An index $r$ marking of a twisted curve looks like the model
$$\sU_r\defeq \left[ \spec\CC[u]\big/ \bmu_r\right],\quad \zeta\cdot u=\zeta u.
$$
Invertible sheaves (representable) on the model $\sV_r$ is a $\bmu_r$-module $\sM_{m}$, $0<m<r$:
$$
\sM_{m}\defeq u^{-(r-m)}\CC[u]\oplus_{[0]} v^{-m}\CC[v]\defeq \ker\{u^{-(r-m)}\CC[u]\oplus v^{-m}\CC[v]\to\CC_{m}\},
$$
where  the arrow is a homomorphism of $\bmu_r$-modules,
$\bmu_r$ leaves $1\in \CC[u]$ and $1\in \CC[v]$ fixed and acts on $1\in \CC_{m}\cong \CC$ via $\zeta\cdot1=\zeta^{m}1$,
and both $u^{-(r-m)}\CC[u]\to \CC_{m}$ and $v^{-m}\CC[v]\to\CC_{m}$ are surjective.
Under this convention, 
\beq\label{Adirect}
\pi_{r\ast}\sM_{m}=\CC[x]\oplus \CC[y],
\eeq
where $\pi_r$ is the projections to its coarse moduli space
\begin{equation}\label{Apir}
\begin{aligned}
&\pi_r: \sV_r\to V_r\defeq \spec\bl \CC[x,y]/(xy),\\
\and
&\pi_r: \sU_r\to U_r\defeq \spec \CC[x],\quad x=u^r, y=v^r.
\end{aligned}
\end{equation}
In the following, we will repeatedly use the fact  that if we let $1_u\in \CC[u]$ be the element $1\in\CC\sub\CC[u]$, and let
$1_v$, $1_x$ and $1_y$ be similarly defined elements, then $\pi_{r\ast}$ in \eqref{Adirect} sends $1_u$ and $1_v$
to $1_x$ and $1_y$, respectively.

Note that the relative automorphism groups $\Aut_{U_r}(\sU_r)\cong\{1\}$, $\Aut_{V_r}(\sV_r)\cong\bmu_r$,
where the latter is generated by ${\zeta_r}\cdot (u,v)=(u,\zeta_r\upmo v)$. These  groups are useful in our
study of MSP fields due to the following lifting property. We assume $r$ is an odd prime.
We will follow the notations
before and after \eqref{Apir}.

\begin{lemm}\label{lift-1} Let $m\in [1, r-1]$.
There is a unique $\Aut_{V_r}(\sV_r)$-linearization of the module $\sM_m$ so that it leaves $1_u\in\sM_m$ fixed.\end{lemm}

\begin{proof}
Let $\alpha\in \Aut_{V_r}(\sV_r)$ so that ${\alpha}\cdot (u,v)=(u,\zeta_r\upmo v)$.
To make $\sM_m$ an $\Aut_{V_r}(\sV_r)$-module that leaves $1_u$ fixed, we only need to find
$\alpha\cdot 1_v=\zeta_r^a 1_v$ so that $\alpha$ acts on $u^{-(r-m)}1_u$ and $v^{-m}1_v$ by identical factors.
Since $\alpha\cdot (u^{-(r-m)}1_u)=u^{-(r-m)}1_u$ and $\alpha\cdot (v^{-m}1_v)=\zeta_r^{m}\zeta_r^a 1_v$,
they are identical if and only if $1=\zeta_r^{m}\zeta_r^a$.
As $m\in [1,r-1]$, it is uniquely solvable (mod $r$) by $a=r-m$.
\end{proof}

A simple corollary of this is that the only $\alpha\in\Aut_{V_r}(\sV_r)$ that lifts to an automorphism $\sM_m$ fixing
$(1_u,1_v)$ is the trivial element $\alpha=1\in \Aut_{V_r}(\sV_r)$.

\begin{lemm}\label{lift-2}
Let $b\ge 0$ and $m\in [1,r-1]$ be integers so that $(m,r)=1$, and  let $l=rb+m$.
Given a $\bmu_l$-action on $V_r$ via ${\zeta_l}\cdot (x,y)=(x,\zeta_l\upmo y)$ and a $\bmu_l$-linearization
on $\pi_{r\ast}\sM_m=\CC[x]\oplus \CC[y]$ via ${\zeta_l}\cdot (1_x,1_y)=(1_x, \zeta_l^ b1_y)$, there is
a unique lifting of this $\bmu_l$-action to $\sV_r$ together with a $\bmu_l$-linearization of $\sM_m$ covering that of
$\pi_{r\ast}\sM_m$.
\end{lemm}

\begin{proof}
A general extension of ${\zeta_l}\cdot (x,y)=(x,\zeta_l\upmo y)$ is
${\zeta_l}\cdot (u,v)=(\zeta_r^c u, \zeta_{rl}\upmo \zeta_r^{c'} v)$ for some  $c,c'\in\ZZ$. Since the
$\bmu_l$-action on $\sV_r$ via $\zeta_r\cdot (u,v)=(\zeta_r^{c'} u,\zeta_r^{-c'}v)$ is the trivial action, by
replacing $c$ with $c+c'$, we can assume $c'=0$.
Then
$${\zeta_l}\cdot (u^{-(r-m)}1_u, v^{-m}1_v)=
(\zeta_r^{-c(r-m)} u^{-(r-m)}1_u, \zeta_{rl}^{m} \zeta_l^b v^{-m}1_v).
$$
It lifts to an action on $\sM_m$ if and only if $\zeta_r^{-c(r-m)}=\zeta_{rl}^{m} \zeta_l^b=\zeta_{rl}^{rb+m}$.
Using $l=rb+m$, the latter identity is equivalent to $\zeta_r^{cm}=\zeta_r$. Because $(m,r)=1$,
it is uniquely solvable for an integer $c\in [1,r-1]$. This proves the lemma.
\end{proof}

\begin{proof}[Proof of Proposition \ref{Phi1}]
Let $(\cC,\Si^\cC,\cL,\cN,\varphi,\rho,\nu)$ (together with $\eps$) be the universal family of $\cW_\ga$. 
$a=(e,v)$ associates to a
$\cW_\ga$ family of nodes $\cQ_{a}\sub \cC$ that decomposes $\cC$ into a partial normalization along $\cQ_a$:
$$p:\cC'\cup \cC_e\lra \cC.
$$
Since $e$ is a leaf edge this partial normalization is a union of two
$S$-families of subcurves
$\eta: \cC'\to\cC$ and $\cC_e\sub\cC$, of which $\cC_e$ is the family of rational curves associated to $e$.
We let $\cQ_a'=\cC'\cap p\upmo(\cQ_a)$. 

We let
$$\Si^{\cC'}=\Si^\cC\cup\cC'\cup\cQ_a'\and (\cL',\cN',\varphi',\rho',\nu')=\eta\sta
(\cL,\cN,\varphi,\rho,\nu).
$$
It is direct to check that
$(\cC',\Si^{\cC'}, \cL',\cdots)$ (together with that induced by $\eps$) is an $\cW_\ga$-family of $\Ga'$-framed MSP fields,
thus defines a morphism $\phi: \cW_\ga\to\cW_{\ga'}$,

On the other hand, the restriction of $(\cL,\cN,\varphi,\rho,\nu)$ to $\cC_e$ is a family in $\cW_e$, thus defines a
morphism $\phi_e:\cW_\ga\to\cW_e$. We prove that the map is an isomorphism:
\beq\label{phip}
\ti\phi\defeq \phi\times\phi_e: \cW_\ga\to\cW_{\ga'}\times\cW_e
\eeq

We will prove in detail the most involved case where $v\in V^S_\infty(\ga)$ and $d_e\notin \mathbb Z$.
The other cases are similar and will be omitted.
In this case, it is easy to see that $\phi$ is one-one, onto and
\'etale. Thus to prove that $\ti\phi$ is an isomorphism we only need to show that for any
$(\xi,\eps)\in\Wfix$,
\beq\label{auto-xi}
\Aut((\xi,\eps))\cong \Aut(\phi(\xi,\eps))\times \Aut(\phi_e(\xi,\eps)).
\eeq

Let $(\xi',\eps')=\phi((\xi,\eps))$,  $\xi_e=\phi_e((\xi,\eps))$,
 $\xi=(\sC,\Si^\sC,\sL,\cdots)$,
and  $\sC'=\sC\cap\sC$, and $\sC_e=\cC_e\cap\sC$. Then $\sC'$ and $\sC_e$ are the domain
curves of $\xi'$ and $\xi_e$, respectively.

We now prove the proposition. First, any $\iota\in \Aut((\xi,\eps))$ induces $\iota'\in \Aut((\xi',\eps'))$ and $\iota_e\in\Aut(\xi_e)$.
We claim that if $\iota'=\iota_e=\id$, then $\iota=\id$.
Indeed, let
$$\alpha_0\in \Aut(\sC),\quad \alpha_1: \alpha_{0\ast} \sL\to \sL,\quad \alpha_2: \alpha_{0\ast} \sN\to\sN
$$
be induced by $\iota$.
Assuming $\iota'=\iota_e=\id$,
then $\alpha_i|_{\sC'}=\alpha_i|_{\sC_e}=\id$ for all $i$. 
Therefore, $\alpha_0$ induces the identity in $\Aut(C)$, where $C$ is the coarse moduli of $\sC$.
Suppose $\alpha_0\ne \id$, then $\alpha_0\in \Aut_{C}(\sC)$.
Because $\xi$ is representable (cf. Definition \ref{def-curve}),
by Lemma \ref{lift-1}, $\alpha_0$ can not be lifted
to automorphisms of both $\sL$ and $\sN$ so that their restrictions over the scheme part of $\sC$ are identities.
Therefore, $\alpha_0=\id$. Then, applying the uniqueness part of
Lemma \ref{lift-1}, we see that $\alpha_1=\alpha_2=\id$.
This proves that $\iota'=\iota_e=\id$ imply that $\iota=\id$.

We now prove the converse. Let $(\kappa',\kappa_e)\in \Aut((\xi',\eps'))\times\Aut(\xi_e)$; we show that
we can find an $\iota\in \Aut((\xi,\eps))$ so that $\iota'=\kappa'$ and $\iota_e=\kappa_e$.
We first prove that we can extend any pair $\kappa'\ne\id$ and $\kappa_e=\id$ to an $\iota\in \Aut((\xi,\eps))$.
Let 
$$\alpha'_0\in\Aut(\sC'),\quad
\alpha'_{1}: \alpha'_{0\ast}\sL|_{\sC'}\to \sL|_{\sC'}\and
\alpha'_{2}: \alpha'_{0\ast}\sN|_{\sC'}\to \sN|_{\sC'}
$$
be that induced by $\kappa'$. 
Note that {since $d_e\notin \mathbb Z$}, $q=\sC_v\cap\sC_e$ is a stack point. Without loss of generality,
we can assume that there is an $\alpha'_{0}$-invariant \'etale chart
$\sU'=[U'/\bmu_5]\to \sC'$ (covering $q$) so that $\alpha'_0|_{\sU'}$ lifts to $\ti\alpha_{0}: U'\to U'$,
and that there is an $\ti x\in\Gamma(\sO_{U'})$ so that
$\zeta_5\cdot \ti x=\zeta_5 \ti x$
and $q=[(\ti x=0)/\bmu_5]$. Because $\rho|_{\sC'}$ is nonzero near $q$, 
we can assume that $\sL|_{U'}\cong \ti x^{-(5-m)} \CC[\ti x]$, for an $m\in [1,4]$,
so that $\rho|_{U'}=(\ti x^{-(5-m)} 1)^5$.
Since $\rho$ is invariant under $\kappa'$, $\bar\alpha'_{1}\colon=\alpha'_{1}\circ (\alpha'_{0})\lsta$ fixes
$(\ti x^{-(5-m)} 1)^5$. Thus there is a
$c\in [0,4]$ so that $\bar\alpha'_{1}$ sends $\ti x^{-(5-m)} 1$ to $\zeta_5^c \ti x^{-(5-m)} 1$.

Now let $\alpha_{0}$ be an extension of $\alpha'_{0}$ to $\sC=\sC'\cup\sC_e$ so that $\alpha_{0}|_{\sC_e}=\id$.
By an argument similar to the proof of Lemma \ref{lift-1}, we can extend $\alpha'_{1}$ 
to an $\alpha_{1}: 
\alpha_{0\ast}\sL|_{\sC}\to \sL|_{\sC}$.
Because $\bar\alpha'_{1}$ sends $\ti x^{-(5-m)} 1$ to $\zeta_5^c \ti x^{-(5-m)} 1$, by the proof of
the same Lemma,
we see that there is a $c'\in[0,4]$ so that $\alpha_{1}|_{\sC_e}=\zeta_5^{c'}\id$.

When $c'=0$, we are done. Otherwise, we apply
Lemma \ref{lift-1} to pick a $\tau\in\Aut_{C_{}}(\sC_{})$ so that it lifts to an
action $\tau$ (the same letter for simplicity) on $\sL|_{\sC_{}}$ so that $\tau|_{\sC'}=\id$ and
$\tau|_{\sC_e}=\zeta_5^{-c'}\id$. Let $\beta_{0}=\tau\circ\alpha_{0}$. Then $\beta_{0}$
lifts to a $\beta_{1}:\beta_{0\ast}\sL|_{\sC_{}}\to\sL_{\sC_{}}$ so that
$\beta_{1}|_{\sC'}=\alpha'_{1}$ and $\beta_{1}|_{\sC_e}=\id$. As to $\sN$, $\alpha'_{2}$ extends to
$\beta_{2}$ that is $\alpha'_{2}$ and $\id$ when restricted to $\sC'$ and $\sC_e$, respectively.
Because the field $\varphi$ vanishes along $\cC_e$, we see that $(\varphi,\rho,\nu_1,\nu_2)|_{\sC_{}}$
are invariant under
$\iota\defeq(\beta_{0},\beta_{1},\beta_{2})$. Combined, we conclude that we can extend
$\kappa'$ to an $\iota\in\Aut((\xi,\eps))$ so that $\iota_e=\id$. 

The other case to verify is that we can extend any $(\id,\kappa_e)\in \Aut((\xi',\eps'))\times\Aut(\xi_e)$
to an $\iota\in \Aut((\xi,\eps))$. The proof is similar  and will be omitted.
\end{proof}

\subsection{Regular graphs}\label{exceptional}

Regular graphs are flat graphs whose contribution to the localization formula are possibly non-zero.

We begin with more conventions.
For a $v\in V_\infty(\Ga)$ with
$\gamma_v=\{\zeta_5^{a_1},\cdots,\zeta_5^{a_c}\}$,
we abbreviate $\gamma_v=(0^{e_0}\cdots 4^{e_4})$,
where $e_i$ is the number of  $i$ in $\{a_1,\cdots,a_c\}$.
Call a vertex $v\in V^S_\infty(\Ga)$ exceptional
if $g_v=0$ and $\gamma_v=(1^{2+k}4)$ or $(1^{1+k}23)$ for some $k\ge 0$. 
Denote by $V^S_{\mathrm{exc}}(\Ga)\sub V^S_\infty(\Ga)$ the set of all {\sl exceptional} vertices.

\begin{defi}\label{regu}
We call a $v\in V_\infty(\ga)$ regular if the followings hold:
\begin{enumerate}
\item In case $v$ is stable, then either $v$ is exceptional, or for every $a\in S_v$ and $e\in E_v$ we have
$\gamma_a$ and $\gamma_{(e,v)}\in \{\zeta_5,\zeta_5^2\}$.
\item In case $v$ is unstable and $\sC_v$ is a scheme point, then $\sC_v$ is a non-marking smooth point
of $\sC$. 
\end{enumerate}
We call $\Ga$ regular if it is flat, $E_{0\infty}(\ga)=\emptyset$, and
that every $v\in V_\infty(\ga)$ is regular; otherwise we call it irregular.
\end{defi}

Let $\ga\in\Del^\reg$; let $v\in V^S(\Ga)$. We denote by $|v|$ the
one vertex graph consisting of the vertex $v$ and legs labeled by
$E_v\cup S_v$. Thus $|v|$ has inherited legs indexed by $S_v$ from $\Gamma$, and with new legs indexed by $E_v$.
Further, $v$ in $|v|$ is decorated by the same data as $v$ in $\Ga$;
legs inherited from $\Gamma$ have the same decorations as those in $\Gamma$;
when $v\in V_0(\ga)$, all new legs in $|v|$ are
decorated by $(1,\rho)$; when $v\in V_1(\ga)$, all news legs in $|v|$ are decorated by $(1,\varphi)$;
when $v\in V_\infty(\ga)$, the new leg in $|v|$ indexed by
$a\in E_v$ is decorated by $\gamma_a$.
(Here $\gamma_a$ is defined in the paragraph before \eqref{VV0}. Since $\ga$ is regular, all $\gamma_a\ne 1$.)

Let $\cW_v$ be the moduli space of $|v|$-framed MSP fields.

\begin{prop}\label{fix2}
We have canonical isomorphisms
$$\cW_{v}\cong \Mbar_{g_v, E_v\cup S_v}(\PP^4,d_v)^p,\quad
\cW_{v}\cong \Mbar_{g_v,E_v\cup S_v},\and  \cW_{v} \cong  \Mbar_{g_v,\gamma_v}^{1/5,5p}
$$ when $v\in V^S_0(\vGa)$, $v\in V_1^S(\ga)$, and $v\in V_\infty^S(\ga)$, respectively. Here    $\Mbar_{g_v,\gamma_v}^{1/5,5p}$  denotes the moduli of $5$-spin $\gamma_v$-pointed twisted curves with five p-fields. 
\end{prop}

By $E_v\cup S_v$-pointed
we mean    the markings  labeled by elements in $E_v\cup S_v$.

\begin{proof}
We look at the case $v\in V_0^S(\ga)$. Let $(\cC_v,\Si^{\cC_v},\cL_v,\cdots)$ be the universal family on $\cW_v$.
Then $(\nu_1,\nu_2)=(0,1)$; $\sN_v\cong\sO_\sC$, and the sections $(\varphi_1,\cdots,\varphi_5)$
defines a morphism
\beq\label{eqn:phi-v} \phi_v: \cC_v\lra \Pf.
\eeq
Since  $\rho$ is a section in
$\cL^{\vee 5}\otimes\omega_{\cC_v/\cW_v}^{\log}(-\Si^{\cC})=\cL^{\vee 5}\otimes\omega_{\cC_v/\cW_v}$, this morphism  with the markings indexed by $E_v\cup S_v$
makes it a family in $\Mbar_{g_v, E_v\cup S_v}(\PP^4,d_v)^p$, thus defines a morphism $\cW_{v}\to
\Mbar_{g_v, E_v\cup S_v}(\PP^4,d_v)^p$. It is direct to check that it is an isomorphism.
The other cases can be check similarly(c.f. Lemma \ref{fix1}), and will be omitted.
\end{proof}

For an unstable $v\in V^U_0(\Ga)$, we define $\cW_v=\Pf$.
For any $v\in V_0(\Ga)$, we define
$$ \cW_{[v]}\colon =\cW_{v}\times_{(\PP^4)^{E_v}} \prod_{e\in E_v} \cW_e,
$$
where $\cW_v\to (\PP^4)^{|E_v|}$ is the product of the
evaluation maps
$\ev_{(e,v)}:\cW_v\to \PP^4$.  For $e\in E_v$, $\cW_e\to\PP^4$ is the coarse moduli morphism
given by Lemma \ref{fix4}.





\begin{coro}
Let $\ga\in\Del^\reg$. Then we have a canonical isomorphism
$$
\cW_{\ga}\cong \prod_{v\in V_0(\Ga)}{\cW_{[v]}}
\times \prod_{v\in V^S_1(\Ga)}\cW_{v}
\times \prod_{e\in E_\infty(\Ga)} \cW_e
\times \prod_{v\in V^S_\infty(\Ga)}\cW_{v},
$$
and the induced
$\cW_\ga\to \prod_{v\in V_0\cup V_1^S\cup V_\infty^S} \cW_v$
is a $G$-gerbe, where $G=\prod_{e\in E}G_e$.
\end{coro}

%
%
%

\section{Virtual localization, part 1}  
Recall that for a DM $\Gm$-stack $W$  with a $\Gm$-equivariant perfect obstruction theory
$\phi_W\dual: \TT_W\to E_W$
and a $\Gm$-invariant cosection $\sigma: \Ob_W\to\cO_W$,
by \cite{GP}, the $\Gm$-invariant part of the perfect obstruction theory
induces a perfect obstruction theory of the fixed locus $W^\Gm$ with the obstruction sheaf
$(\Ob_W|_{W^\Gm})^T$ and  the cosection (the invariant part of $\sigma$)
$$\sigma^T: Ob_{W^\Gm}\defeq (\Ob_W|_{W^\Gm})^T\lra \cO_{W^\Gm}.
$$
Let $\coprod_{a\in \Lambda} W_a=W^\Gm$ be a decomposition into disjoint open and closed substacks.
 Let $\sigma_a=\sigma^T|_{W_a}$. Let $D(\sigma)$ be the degeneracy (non-surjective) part of
$\sigma$,
let $\iota_a: D(\sigma_a)\to D(\sigma)$ be the inclusion, and let
$$[W]\virtloc\in A\lsta^\Gm D(\sigma)\quad \text{(resp. } [W_a]\virtloc\in A\lsta^\Gm D(\sigma_a)\text{)}
$$
be the cosection localized equivariant virtual cycle of $(W,\sigma)$ (resp. $(W_a,\sigma_a)$).
The following virtual localization formula (analogous to \cite{GP}) is addressed in
\cite{CKL}.

\begin{theo}[Virtual localization formula] \label{thm3.1}Let the notations be as stated. Suppose each moving part
$(E_W|_{W_a})\umv$ is quasi-isomorphic to a $T$-equivariant two-term complex
of coherent locally free sheaves $[\sF_{a,0}\to \sF_{a,1}]$.
Then after inverting (the generator) $\ft\in H^2(B\Gm)$ and letting $e(N_a\virt)=e(\sF_{a,0})/e(\sF_{a,1})$, we have
$$[W]\virtloc=\sum_{a\in\Lambda} \iota_{a\ast}\Bigl( \frac{[W_a]\virtloc}{e(N_a\virt)}\Bigr)\in
A\lsta^\Gm D(\sigma) [\ft\upmo].
$$
\end{theo}

\subsection{Applying virtual localization formula}
To apply the localization formula, we first lift a $T$-equivariant relative obstruction theory of $\cW$($=\cW\lggd$) to
a $T$-equivariant perfect obstruction theory of $\cW$.

We recall the relevant notations. Recall that $\cD$ is the stack of $(\sC,\Si^\sC,\sL,\sN)$.
By forgetting $(\varphi,\rho,\nu)$, we get the forgetful morphism $q:\cW\to\cD$.
Let
$$(\cC_\cW,\Si^{\cC_\cW},\cL_\cW,\cN_\cW,\varphi,\rho,\nu)\and \pi:\cC_\cW\lra\cW
$$
be the universal family of $\cW$. Following \cite{CLLL}, we have the relative perfect obstruction theory
\beq\label{phiW}
\phi_{\cW/\cD}\dual:\TT_{\cW/\cD}\lra \EE_{\cW/\cD}\defeq
R\pi_{\ast}\cV_\cW,
\eeq
\beq\label{VV}\hbox{where}\quad 
\cV_\cW=(\cL^{\log}_\cW)^{\oplus 5}\oplus \cP_\cW^{\log}\oplus \cL_\cW\otimes\cN_\cW\otimes\bL_1\oplus\cN_\cW,
\eeq
where $\cL_\cW^{\log}=\cL_\cW(-\Si^{\cC_\cW}\lophi)$, and
$\cP_\cW^{\log}=\cL_\cW^{\vee\otimes 5}\otimes\omega_{\cC_\cW/\cW}^{\log}(-\Si^{\cC_\cW}\lorho)$.

It comes with a cosection
$\sigma_{\cW/\cD}:\Ob_{\cW/\cD}\to\sO_\cW$. 
Following the  constructions in \cite{CL}, both the obstruction theory and the cosection are $T$-equivariant,
and liftable to a cosection of $\Ob_\cW$, thus provide a
$T$-equivariant cycle $[\cW]\virtloc$.

To apply the virtual localization formula, we first construct a $T$-equivariant (absolute) obstruction theory $\phi_\cW\dual$.
Let $\theta$ be the arrow in the distinguished triangle (d.t. for short) in the lower line in \eqref{dia5.0}
induced by $q:\cW\to\cD$. 
\beq\label{dia5.0}
\begin{CD}
 @>>>q\sta \TT_{\cD}[-1]@>{\phi\lwd\dual\circ\theta}>>\EE\lwd@>>>  \EE_\cW  @>+1>>\\
 @.@|@AA{\phi\lwd\dual}A@AA{\phi_\cW\dual}A\\
  @>>>q\sta \TT_{\cD}[-1]@>\theta>> \TT\lwd@>>> \TT_\cW  @>+1>>.
\end{CD}
\eeq
 Here $\EE_\cW$ is the mapping cone of $\phi\lwd\dual\circ\theta$,
and  $\phi\dual_\cW$ (shown above) 
is an arrow in $D^+(\sO_{\cW})$ that makes above an arrow of d.t.'s.
A standard argument\footnote{Similar arguments can be found in Section 3. It will also be
addressed in \cite{CL2}.} shows that $\phi_\cW\dual$ is a perfect obstruction theory.

 To proceed, we argue that we can make objects and arrows in \eqref{dia5.0} lie in
$D^+_{\text{qcoh}}(\sO_{[W/T]})$,
 where $D^+_{\text{qcoh}}(\sO_{[W/T]})$ is the subcategory of $D(\sO_{[W/T]})$
consisting of complexes of
 $\sO_{[W/T]}$-modules with quasi-coherent cohomologies  \cite[page 126]{Champ},
in the \textit{lisse-\'etale} site.

  Applying \cite[Prop 13.2.6-iii]{Champ} to the 
$[\cC_\cW/T]\to [\cW/T]$ (induced by  $\pi_{\cW}:\cC_\cW\to\cW$), the construction of $\phi\lwd\dual$
shows that $\phi\lwd\dual$ can be represented as an arrow in $D^+_{\text{qcoh}}(\sO_{[W/T]})$ in the \textit{lisse-\'etale} site. 


By Illusie's construction of cotangent complexes,
the $\theta$ in \eqref{dia5.0} 
is an arrow in $D^+_{\text{qcoh}}(\sO_{[\cW/T]})$.
Thus $\phi_{\cW/\cD}\dual\circ  \theta$ is an arrow in $D^+_{\text{qcoh}}(\sO_{[\cW/T]})$.
By the mapping cone axiom of $D^+_{\text{qcoh}}(\sO_{[\cW/T]})$, we
obtain a lift $\phi_\cW\dual$ that is an arrow in $D^+_{\text{qcoh}}(\sO_{[\cW/T]})$.
This makes $\phi_\cW\dual$ a $T$-equivariant obstruction theory of $\cW$.

Next, we claim that the cosection $\sigma\lwd$ lifts to a cosection
$$\sigma_\cW: \Ob_\cW=H^1(\EE_\cW)\lra\sO_\cW.
$$
Indeed, that $H^1(\EE_\cW)$ is the absolute obstruction sheaf of $\cW/\cD$, and that $\sigma\lwd$ lifts
to a cosection of $H^1(\EE_\cW)$ shown in \cite[Lemma 2.10]{CLLL} allow  us to define $\sigma_\cW$ to be this lift,
which is $T$-equivariant since $\sigma\lwd$ is. Thus the $T$-equivariant
$(\phi_\cW\dual,\sigma_\cW)$ defines a $T$-equivariant cosection localized virtual cycle of $\cW$.

To apply virtual localization, we need to verify that for $\cW\lrga\sub\cW$, which is the image of
$\iota_\ga: \cW_\ga\to\cW$ (cf. Proposition \ref{open-closed}),
we can find a two-term complex of locally free  sheaves $[\cF_{(\ga),0}\to \cF_{(\ga),1}]$ on $[\cW\lrga/T]$ so that
the moving part is
\beq\label{T-F} (\EE_\cW|_{\cW\lrga})\umv\cong_{q.i.} [\cF_{(\ga),0}\to \cF_{(\ga),1}].
\eeq

First, as $\cW\lrga\sub\cW^T$ is an open and closed substack, we have
the direct sum decomposition as complexes of quasi-coherent sheaves on $[\cW_{(\ga)}/T]$:
\beq\label{T-dec}
\EE_\cW|_{\cW\lrga}= \bigoplus_k (\EE_\cW|_{\cW\lrga})^{(k)},
\eeq
where the superscript $(k)$ means the corresponding weight $k$ part.

Since $\EE_\cW|_{\cW\lrga}$ is a complex of sheaves of $\sO_{\cW_\ga}$-modules with coherent cohomologies,
the same is true for all  $(\EE_\cW|_{\cW\lrga})^{(k)}$.
By the explicit description of $\cW_{\ga}$, it is direct to check that $\cW_{\ga}$ has the  resolution property and
its coarse moduli is projective. Since $\cW_\Ga\to\cW_{(\ga)}$ is a finite quotient, the same is true for
$\cW_{(\ga)}$. Applying \cite[Prop 5.1]{Kr} and \cite[Prop  3.5 and Lem 3.6]{Huy},
and using that for any closed $\xi\in \cW_{(\ga)}$, $H^i(\EE_\cW|_\xi)=0$ for $i\not\in[0,1]$,
we conclude that each  $(\EE_\cW|_{\cW\lrga})^{(k)}\otimes\bL_{-k}$ is quasi-isomorphic to
a two-term complex of finite rank locally free sheaves of amplitude $[-1,0]$ on $\cW_{(\ga)}$.
Finally, since $\EE_\cW|_{\cW\lrga}$ has coherent cohomologies, all but finitely many of $(\EE_\cW|_{\cW\lrga})^{(k)}$'s
are quasi-isomorphic to $0$. This verifies the requirement to apply Theorem \ref{thm3.1}. 
Following the convention, we agree
$$e(N\virt\lrga)=e(\cF_{(\ga),0})/e(\cF_{(\ga),1}).
$$

\begin{prop} 
Let $[\cW]\virtloc$ (resp. $[\cW_{(\ga)}]\virtloc$) be the cosection localized virtual cycle
of $\phi\lwd\dual$ (resp. of $(\phi_\cW\dual|_{\cW\lrga})^T$),
let   $\jmath^-\lrga:\cW^-\lrga\to\cW^-$  be the inclusion. Then 
\beq\label{loc-form}[\cW]\virtloc=
\sum_{(\vGa)  
} \jmath^-_{(\ga)\ast}\Bigl( \frac{[\cW_{(\Ga)}]\virtloc}{e(N\lrga\virt)}\Bigr)\in
 A\lsta^\Gm (\cW^{-}) [\ft\upmo].
\eeq
\end{prop}

\begin{proof} By our construction of the $T$-equivariant $(\phi_\cW\dual,\sigma_\cW)$,
we can apply the virtual localization theorem Theorem \ref{thm3.1} to get
\eqref{loc-form}, with $[\cW]\virtloc$ the cosection localized virtual cycle
of $(\phi_\cW\dual, \sigma_\cW)$, and $[\cW_{(\ga)}]\virtloc$ as stated in the proposition.
To complete the proof of the proposition, we need to verify that the cosection localized virtual cycle of
$(\phi_\cW\dual, \sigma_\cW)$ and {that} of $(\phi\lwd\dual, \sigma\lwd)$ are identical. The proof of this
is parallel to that of Proposition \ref{prop3.6}, and will be omitted.
\end{proof}

\subsection{The fixed part}
In this subsection, we determine the cycle $[\cW_{(\ga)}]\virtloc$. 

Denote $\cW^-\sub \cW$ to be the set of
closed points $\xi\in\cW$ such that
$(\varphi=0)\cup(\rho=\sum \varphi_i^5=0)=\sC$ (cf. \cite[Lemma 2.11]{CLLL}). We let
$$\cW^\sim=\{\xi\in\cW\mid (\varphi=0)\cup(\rho=0)=\sC\}.
$$
By \cite[Coro 3.23]{CLLL}, $\cW^\sim$ is proper. 

\begin{defi}
We say $\beta\in A\lsta^T(\cW^{-})^T$ is weakly zero, denoted by $\beta\sim0$, if 
there is a closed proper substack $Z'\sub \cW^T$ with $(\cW^{-})^T\sub Z'$ so that 
$\beta$ is mapped to zero under the induced homomorphism $A\lsta^T(\cW^{-})^T\to A\lsta^T Z'$. \end{defi}
  We quote a vanishing proved in \cite{CL2}.
\begin{prop}[{\cite{CL2}}] \label{ir-vanishing}
In case $\Ga\in \Del^{\mathrm{fl}}$ is narrow and irregular, and not a loop, then
$[\cW_{(\Ga)}]\virtloc\sim 0$. 
\end{prop}

As a weakly zero class in localization formula can be treated as zero,
in using \eqref{loc-form} to get numerical relations, we only need to sum over all $\Ga\in \Del^\reg$.

%
%
%
%
%
%
%

In the following, we fix a $\Ga\in \Delta^\reg$. 
As $\Ga$ is fixed throughout the remaining  of this subsection,
we skip $\Ga$ from the notation $V(\ga)$, etc., and use $V$ and $E$ to denote $V(\Ga)$ and $E(\Ga)$ respectively.
Accordingly $V^S=V^S(\Ga)$, and $V^{0,1}_0=V^{0,1}_0(\Ga)$, etc..

Let $v\in V^S_0\cup V_\infty^S$.
Since $\cW_v$ defined before Proposition \ref{fix2} is the moduli of $T$-equivariant $|v|$-framed stable MSP fields,
applying \cite{CLLL} we obtain its cosection localized virtual cycle, denoted by $[\cW_v]\virtloc$.
For $v\in V_1^S$, $\cW_v\cong \overline{\cM}_{g_v,|E_v\cup S_v|}$, thus we set $[\cW_v]\virtloc=[\cW_v]$.
For unstable $v\in V_0^U$, we set $[\cW_v]\virtloc=-[Q_5]$, where $Q_5\sub \Pf$ is the Fermat quintic.

We  use $\imath_\ga: \cW_\ga\to \cW\lrga$ (cf. \eqref{ij}) and
use $\imath_\ga^-: \cW_\ga^-\to\cW\lrga^-$ induced by $\iota_\ga$.

\begin{proposition}\label{prop-loc}
Let $\Ga\in \Delta^\reg$. Then
$$[\cW_{(\Ga)}]\virtloc= \frac{1}{|\Aut(\Ga)|} \frac{1}{\prod_{e\in E} |G_e|}
(\imath_{\ga}^-)_\ast\biggl(\prod_{v\in V_0\cup V_1^S\cup V_\infty^S} [\cW_v]\virtloc\biggr).
$$
\end{proposition}

The rest of this subsection is devoted to prove this proposition.
Let $\ga\in\Del^\fl$, not necessarily regular.
We first 
reconstruct $[\cW_{\Ga}]\virtloc$.
Let $\cD_\ga$ be the smooth Artin stack of $\Ga$-framed curves $(\Si^\sC,\sC,\sL,\sN)$ as defined in and after Definition \ref{cDga}.
Let
$$\jmath: \cD_\ga\lra \cD\and q_\ga: \cW_\ga\lra \cD_\ga
$$
Then we have the Cartesian square via the tautological morphisms
\beq\label{candia}
\begin{CD} \cW_\ga @>{\iota_\ga}>> \cW\\
@VV{q_\ga}V @VV{q}V\\
\cD_\ga @>{\zeta_\ga}>> \cD.
\end{CD}
\eeq
Let $(\cC,\Si^{\cC},\cL,\cN,\cdots)$ with $\pi:\cC\to\cW_\ga$ be the
universal family of $\cW_\ga$. Let $\cL^{\log}$, $\cP^{\log}$ and $\cV$ be defined as in \eqref{VV} with $\cL_\cW$, etc.,
replaced by $\cL$, etc..
Parallel to the construction of the relative obstruction theory $\phi\lwd\dual$ of $\cW\to\cD$,
we obtain a $T$-equivariant relative obstruction theory
\beq\label{VVga}
\phi_{\cW_{\ga}/\cD_\ga}\dual: \TT_{\cW_{\ga}/\cD_\ga}\lra \EE_{\cW_{\ga}/\cD_\ga}=R\pi_{\ga\ast}^T \cV.
\eeq
We will show later that
\beq\label{ObOb}
\Ob_{\cW_\ga}=\coker\{ q_\ga\sta \TT_{\cD_\ga}[-1]\lra H^1(\EE_{\cW_\ga/\cD_\ga})\}=(\iota_\ga\sta\Ob_\cW)^T.
\eeq
Granting this, $(\iota_\ga\sta\sigma_\cW)^T$ induces a cosection of $\Ob_{\cW_\ga/\cD_\ga}$, liftable to a cosection of
its absolute obstruction sheaf.
\medskip

\begin{prop}\label{prop3.6}
The cosection localized $T$-equivariant cycle $[\cW_\ga]\virtloc$ using the pairs
$(\phi_{\cW_\ga/\cD_T}\dual, \sigma_{\cW_\ga/\cD_T})$
and $((\iota_\ga\sta\phi_\cW\dual)^T,(\iota_\ga\sta\sigma_\cW)^T)$ are identical.
\end{prop}


\begin{lemm}\label{lem3.7}
The tautological map $q_\ga\sta \TT_{\cD_\ga}\to (\iota_\ga\sta q\sta \TT_\cD)^T$ is an isomorphism in the derived category
$D_{qcoh}^+(\sO_{\cW_\ga})$.
\end{lemm}

\begin{proof}
It is sufficient to show that $\zeta^{\ast}_\ga\TT_{\cD}$ is canonically a $T$-complex,
and the tautological map $\TT_{\cD_\ga}\to \zeta^{\ast}_\ga\TT_{\cD}$ factors through a $T$-equivariant
isomorphism
\beq\label{tt}
\TT_{\cD_\ga}\mapright{\cong} (\zeta^{\ast}_\ga\TT_{\cD})^T.
\eeq

To this end, we transform any data $(\Si^\sC,\sC,\sL,\sN)$ in $\cD$ to a $(\CC\sta)^2$-pair $\sR\sub\sS$:
$$\sS=\PP_\sC(\sL\oplus\sN\oplus 1),\quad \sR=\sS\times_{\sC}\Si^\sC\sub\sS,\quad
(t_1,t_2)\cdot[a,b,1]=[t_1a,t_2b,1]. 
$$
Let $\cS$ be the stack of such $(\CC\sta)^2$-pairs $\sR\sub\sS$. Then the mentioned construction defines
an equivalence of stacks $\cD\cong\cS$. Furthermore, if we let $\cS_T$ be the stack of $T$-equivariant
$(\CC\sta)^2$-pairs in $\cS$, and let $\cD_T$ be the stack of $T$-equivariant $(\sC,\Si^\sC,\sL,\sN)$ in $\sD$,
then we have an equivalence of stacks $\cD_T\cong \cS_T$. Thus we have induced isomorphisms
$\TT_{\cD}\cong \TT_{\cS}$ and $\TT_{\cD_T}\cong \TT_{\cS_T}$.

Let $\eta: \cS_T\to\cS$ be the forgetful
morphism. Since $\eta\sta\TT_{\cS}$ has a canonical $T$-action, the isomorphism $\eta\sta\TT_{\cS}\cong
\zeta^{\ast}_\ga\TT_{\cD}$ provides the latter a $T$-action.
Furthermore, as
$\TT_{\cS_T}\to \eta\sta\TT_{\cS}$ factors through an isomorphism $\TT_{\cS_T}\to (\eta\sta\TT_{\cS})^T$,
$\TT_{\cD_T}\to \zeta^{\ast}_\ga\TT_{\cD}$ factors through a $T$-equivariant
isomorphism \eqref{tt}. This proves the lemma.
\end{proof}

\begin{proof}[Proof of Proposition \ref{prop3.6}]
We apply the construction of relative obstruction theory in \cite{CL} to the square \eqref{candia}
to obtain the following commutative diagrams
\beq\label{WgaDt}
\begin{CD}
\EE_{\cW_\ga/\cD_\ga}@>=>> ( \iota_\ga\sta\EE_{\cW/\cD})^T@>{\text{inj}}>> \iota_\ga\sta\EE_{\cW/\cD}\\
@AA{\phi_{\cW_\ga/\cD_T}\dual}A@AA{(\phi_{\cW/\cD}\dual)^T}A@AA{\phi_{\cW/\cD}\dual}A\\
\TT_{\cW_\ga/\cD_\ga}@>\gg_1>>( \iota_\ga\sta \TT_{\cW/\cD})^T@>{\text{inj}}>> \iota_\ga\sta \TT_{\cW/\cD},
\end{CD}
\eeq
where the upper-left arrow is an isomorphism, which follows from the definition of $\EE_{\cW_\ga/\cD_\ga}$.
We then take the $T$-invariant part of \eqref{dia5.0} to get the  upper two lines of the following commutative diagram in
$D^+_{qcoh}(\sO_{\cW_\ga})$,
where the lower two lines are induced by the square \eqref{candia}, and the lower-left isomorphism is due to Lemma \ref{lem3.7}:
\beq\label{dia5}
\begin{CD}
 @>>>(\iota_\ga\sta q\sta \TT_{\cD})^T[-1]@>\tau>>( \iota_\ga\sta\EE_{\cW/\cD})^T@>>>  (\iota_\ga\sta\EE_{\cW})^T  @>+1>>\\
 @.@|@AA{(\phi_{\cW/\cD}\dual)^T}A@AAA\\
@>>>(\iota_\ga\sta q\sta \TT_{\cD})^T[-1]@>>>(\iota_\ga\sta\TT_{\cW/\cD})^T@>>> (\iota_\ga\sta\TT_{\cW})^T@>+1>>\\
   @.@AA{\cong}A@AA{\gg_1}A@AAA\\
     @>>>(q_\ga\sta   \TT_{\cD_\ga})[-1]@>>>\TT_{\cW_\ga/\cD_\ga}@>>>\TT_{\cW_\ga}@>+1>>.
\end{CD}
\eeq

By \eqref{WgaDt}, the composition of the middle vertical arrows is $\phi\dual_{\cW_\ga/\cD_\ga}$.
As $\TT_{\cD_\ga}$ is locally perfect of amplitude $[-1,0]$,
we have an exact sequence of cone-stacks induced from the upper and lower lines of \eqref{dia5} and
\cite[Prop 2.7]{BF}:
\beq\label{exactstack0}
\begin{CD}
  h^1/h^0(q_\ga\sta   \TT_{\cD_\ga}[-1])@>>>h^1/h^0(\EE_{\cW_\ga/\cD_\ga})@>\lam >>h^1/h^0(\EE_{\cW_\ga})\\
  @AA{\cong}A@AA{(\phi\dual_{\cW_\ga/\cD_\ga})\lsta}A@AA{(\phi\dual_{\cW_\ga})\lsta}A\\
    h^1/h^0(q_\ga\sta   \TT_{\cD_\ga}[-1])@>>>h^1/h^0(\TT_{\cW_\ga/\cD_\ga})@>>>h^1/h^0(\TT_{\cW_\ga}).
\end{CD}
\eeq
Applying the argument analogs  to  \cite[Coro 2.9]{CL} (also see \cite[Prop 3]{KKP}),
we conclude that the two intrinsic normal cones
$$\fC_{\cW_\ga}\sub h^1/h^0(\TT_{\cW_\ga})\and  \fC_{\cW_\ga/\cD_\ga}\sub h^1/h^0(\TT_{\cW_\ga/\cD_\ga})
$$ satisfy
$$\lam\sta(\fC_{\cW_\ga})=\fC_{\cW_\ga/\cD_\ga}.
$$
Since the two cosections of $\Ob_{\cW_\ga/\cD_\ga}$ and of $\Ob_{\cW_\ga}$ are compatible under $\lambda$,
we conclude that the localized virtual classes using the cone $\fC_{\cW_\ga}$ and using
the cone $\fC_{\cW_\ga/\cD_\ga}$ are identical (also see \cite[Prop 3]{KKP}). This proves that the
the cycles $[\cW_\Ga]\virtloc$ defined using $\phi\dual_{\cW_\ga/\cD_\ga}$ and
using $(\iota_\ga\sta\phi_\cW)^T$ 
are identical.
\end{proof}

\begin{coro}
Let the notation be as stated. Then $$[\cW_{(\ga)}]\virtloc=\frac{1}{|\Aut(\ga)|}[\cW_{\ga}]\virtloc.$$
\end{coro}

We will prove Proposition \ref{prop-loc} by proving a series of Lemmas.
Let $\ga$ be a flat graph (not necessary regular); let $e\in E_0\cup E_\infty$ be a non-leaf edge of $\ga$,
and let $v\in V_1$ be the vertex of $e$ lying in $V_1$.
We let $\ga'$ be the decoupling of $\Ga$ along $a=(e,v)$. Let
$\phi: \cW_\ga\to\cW_{\ga'}$ be the isomorphism stated in \eqref{phi-de}.

\begin{lemm}\label{virvir}
We have identity $\phi\sta[\cW_{\ga'}]\virtloc=[\cW_{\ga}]\virtloc$.
\end{lemm}

We let $(\cC,\Si^\cC,\cL,\cdots)$, with $\pi: \cC\to \cW_\ga$ be the universal family of $\cW_\ga$;
let $(\cC',\Si^{\cC'},\cL',\cdots)$ with $\pi:\cC'\to\cW_{\ga'}$ be the universal family of $\cW_{\ga'}$.
By identifying $\cW_\ga$ with $\cW_{\ga'}$ using $\phi$, we can view both $\cC$ and $\cC'$ as families over
the same stack.
As argued before, $\cC'$ is derived from $\cC$ by partial resolution along the $\cW_\ga$-family of nodes
$\cQ_a\sub \cC$ associated
to $a$. Let $\eta:\cC'\to\cC$ be the partial resolution morphism.

We set
$\cL^{\log}=\cL(-\Si^\cC\lophi)$, $\cP^{\log}=\cL^{\vee\otimes 5}\otimes\omega_{\cC/S}^{\log}(-\Si^\cC\lorho)$,
and set $\cV$ be as in \eqref{VV} with subscript ``$\cW$" removed.
We set similarly $\cL^{\prime\log}$, $\cP^{\prime\log}$, and $\cV'$ on $\cC'$.

\begin{lemm}\label{G-fixed}
Let the notations be as stated. 
Then we have quasi-isomorphisms
$R\pi^T\lsta \cL^{\log}\cong R\pi^T\lsta\cL^{\prime\log}$, $R\pi^T\lsta \cP^{\log}\cong R\pi^T\lsta\cP^{\prime\log}$,
and a d.t.
\beq\label{pp}
R\pi^T\lsta \cV\lra \phi\sta R\pi^T\lsta \cV' \lra
\pi^T\lsta (\cL\otimes\cN\otimes\bL_1|_{\cQ_a}\oplus\cN|_{\cQ_a})
\mapright{+1}
\eeq
where the last term is canonically isomorphic to $\cO_{\cW_\ga}^{\oplus 2}$.
\end{lemm}

\begin{proof} We first prove the two quasi-isomorphisms stated. Because in the case $v\in V_1(\ga)$, two
sections $\cQ'_a$ and $\cQ\dpri_a\sub \cC'$ (lie in the preimage of $\cQ_a$ under $\eta:\cC'\to\cC$)
are marking with $(1,\varphi)$, we have short exact sequence
$$0\to \eta\lsta \cL^{\prime\log}\to \cL^{\log}\to \cL|_{\cQ_a}\to 0\and
0\to \cP^{\log}\to \eta\lsta\cP^{\prime\log}\to \cP|_{\cQ_a}\to 0.
$$
Here we used that 
$\cL'=\eta\sta\cL$ and $\cN'=\eta\sta\cN$.

Since $T$ acts non-trivially on
$\sL|_{\cQ_a}$, $R\pi\lsta^T (\sL|_{\cQ_a})=R\pi\lsta^R (\cP|_{\cQ_a})=0$.
Thus applying the functor $R\pi\lsta^T$ to the above two short exact sequences,
we obtain the desired quasi-isomorphisms.

Finally, since $\cV=\cL^{\log\oplus 5}\oplus \cP^{\log}\oplus
\cL\otimes\cN\otimes\bL_1\oplus \cN$, and since $(\cL\otimes\cN\otimes\bL_1\oplus \cN)|_{\cQ_a}\cong\cO_S^{\oplus}$
because $\nu_1$ and $\nu_2$ are nowhere vanishing $T$-equivariant sections, we obtain the d.t. stated.
\end{proof}

\begin{proof}[Proof of Lemma \ref{virvir}]
Obviously, we have the forgetting morphism
$q_\ga: \cW_\ga\to \cD_\ga$, and $q_{\ga'}: \cW_{\ga'}\to\cD_{\ga'}$.

We show that there is a tautological morphism $\cD_\ga\to\cD_{\ga'}$ that fits into the
commutative diagram of morphisms
\beq\label{WDga}
\begin{CD}\cW_\ga@>{\phi}>> \cW_{\ga'}\\
@VV{q_\ga}V@VV{q_{\ga'}}V\\
\cD_\ga@>>> \cD_{\ga'}.
\end{CD}
\eeq
Given a $\Ga$-framed curve $(\sC,\Si^\sC,\sL,\sN)$, following Definition \ref{cDga}, $\sC$ has a distinguished ($T$-unbalanced)
node $q_a$, and the two branches of $q_a$ are also labeled by data in $\Ga$.
Let $\pi:\sC'\to \sC$ be the partial normalization of $\sC$ along $q_a$,
and let $\Si^{\sC'}=\pi\upmo(\Si^\sC)\cup \pi\upmo(q_a)$, where we mark $\pi\upmo(\Si^\sC)$ by the legs of
$\ga'$ that come from legs in $\Ga$, and mark the two sections in $\pi\upmo(q_{a})$ via $(1,\varphi)$.
As $T$-unbalanced nodes remain nodes in any family of $T$-equivariant curves, this association defines
a morphism $\cD_\ga\to\cD_{\ga'}$ that fits into the commutative square \eqref{WDga}.

Without lose of generality, we can assume $\Ga$ is connected.
Then $\Ga'$ either is connected or has two connected components.
In case $\Ga'$ is connected, the lower horizontal line in \eqref{WDga} is a $(\CC\sta)^2$-torsor,
where the first factor of
the group $(\CC\sta)^2$ acts on $\cD_\ga$
by scaling the gluing of $\eta\lsta\cL'$ over the section of nodes $\cQ_{a}\sub \cC$,
and the second factor acts by scaling the gluing of $\eta\lsta\cN'$ over $\cQ_{a}\sub \cC$.
Therefore, $\cD_\ga\to\cD_{\Ga'}$ is smooth, and  
\beq\label{tempp}
q_\ga\sta \TT_{\cD_\ga/\cD_{\ga'}}\cong 
\sO_{\cW_\ga}^{\oplus 2}.
\eeq
In case $\ga'$ has two connected components, then $\cD_\ga\to
\cD_{\ga'}$ is a $(\CC\sta)^{2}$-gerbe. In this case, we also have isomorphism \eqref{tempp}.

On the other hand, as $\cW_{\ga'}$ is defined similar to $\cW_\ga$, it comes with a tautological relative
obstruction theory which takes the form
$$\phi_{\cW_{\ga'}/\cD_{\ga'}}\dual: \TT_{\cW_{\ga'}/\cD_{\ga'}}\lra \EE_{\cW_{\ga'}/\cD_{\ga'}}=
R\pi_{\ast}^T\cV'.
$$
We then form a morphism of d.t.s
\beq\label{hugge}
\begin{CD}
\EE_{\cW_\ga/\cD_\ga}@>{\gamma_1}>>\phi\sta \EE_{\cW_{\ga'}/\cD_{\ga'}}@>>>
\pi_{\ast}^T(\cV^c|_{\cQ_{a}}) @>+1>>\\
@AA{\phi_{\cW_\ga/\cD_\ga}\dual}A@AA{\phi\sta( \phi\dual_{\cW_{\ga'}/\cD_{\ga'}})}A@AA{\alpha}A\\
\TT_{\cW_\ga/\cD_\ga}@>{\gamma_2}>>\phi\sta\TT_{\cW_{\ga'}/\cD_{\ga'}}@>>>q_\ga\sta\TT_{\cD_\ga/\cD_{\ga'}}@>+1>>.\\
\end{CD}
\eeq
Here the top d.t. is \eqref{pp}, using that $\EE_{\cW_\ga/\cD_\ga}=R\pi_{\ast}^T\cV$, etc.;
the lower d.t. is from the functoriality of cotangent complexes using the arrows in \eqref{WDga}.
Because $\phi$ in \eqref{WDga} is an isomorphism, the construction of obstruction theories in \cite{CL} applied
to this case ensures that the left square is commutative. Let $\alpha$ be so that it makes the above  diagram a morphism
between d.t.s.

We claim that $\alpha$ is an isomorphism. Indeed, by the discussion after \eqref{pp}, and
by the isomorphism \eqref{tempp}, both $q_\ga\sta\TT_{\cD_\ga/\cD_{\ga'}}$ and
$\pi_{\ast}^T(\cV^c|_{\cQ_{a}})$ are isomorphic to $\sO_{\cW_\ga}^{\oplus 2}$.
By the construction of $\gamma_1$ and the explicit knowledge of $\phi$ in \eqref{WDga}, we conclude that
$$\coker H^0(\gamma_1)\cong \coker H^0(\gamma_2)
\and
\ker H^1(\gamma_1)\cong \ker H^1(\gamma_2).
$$
This implies that $\alpha$ is an isomorphism.

Finally, we check that the $\Gm$-invariant cosection $\sigma_{\cW}$ of the obstruction sheaf of $\cW$
lifts to a cosection $\sigma_\Ga$ of the obstruction sheaf of $\Wfix$, identical to that using the formula
(2.9) in \cite{CLLL}. Let $\sigma_{\ga'}$ be the cosection of the
obstruction sheaf of $\cW_{\ga'}$, defined similarly. Then it is straightforward to see that $\sigma_\ga$ and
$\sigma_{\ga'}$ are consistent under {$\alpha$}. Since $\TT_{\cD_\ga/\cD_{\ga'}}$ is a locally free sheaf,
$$\begin{CD}
h^1/h^0(\EE_{\cW_\ga/\cD_\ga})@>{h^1/h^0(\gamma_1)}>> h^1/h^0(\EE_{\cW_{\ga'}/\cD_{\ga'}})
\end{CD}
$$
is a smooth quotient. By the commutative diagram above,   the intrinsic normal cone
$\fC_{\cW_\ga/\cD_\ga}$ is the pullback of the intrinsic normal cone $\fC_{\cW_{\ga'}/\cD_{\ga'}}$.
Therefore, by the comparison of the two cosections $\sigma_\ga$ and $\sigma_{\ga'}$, we prove the
proposition.
\end{proof}

We continue to let $\ga\in\Del^\fl$; let $e\in E_0\cup E_\infty$ to be a leaf edge and $v$ to be its
connecting vertex. Let $\ga'$ be the trimming of $e$ from $\ga$,  and  $\phi: \cW_\ga\to\cW_{\ga'}$ be
the tautological morphism in \eqref{phip}. By the proof of Prop.  \ref{Phi1}, it is a $G_e$-gerbe.

\begin{lemm}\label{L3.10}
We have identity $\phi\sta[\cW_{\ga'}]\virtloc=[\cW_{\ga}]\virtloc$.
\end{lemm}

\begin{proof}
We will prove the case $v\in V_0$. The other cases are similar. First recall the construction of the morphism $\phi$.
Let $(\cC,\Si^{\cC},\cdots)$ with
$\pi:\cC\to\cW_{\ga}$ be the universal family of $\cW_{\ga}$;
let $(\cC,\Si^{\cC'},\cdots)$ with $\pi:\cC'\to\cW_{\ga'}$ be the universal family of $\cW_{\ga'}$.
Like the case of decoupling, we have a section of nodes $\cQ_a\sub\cC$ associated to the
flag $a=(e,v)$, so that the decomposition of $\cC$ along $\cQ_a$ results two subfamilies:
$\cC'\times_{\cW_{\ga'}}\cW_\ga$, and $\cC_e$,
where $\cC_{e}$ is a family of rational curves associated with $e$. Let $\eta:\cC'\to\cC$
be the tautological morphism.
Then the pullback via $\eta\sta$ of $(\cC,\Si^{\cC}\cup\cQ_a,\cL,\cdots)$
is the universal family of $\cW_{\ga'}$.


Let $\cL^{\log}$, $\cP^{\log}$, etc., $\cV$ and $\cV'$ be as defined before Lemma \ref{G-fixed}.
We claim that we have the following d.t.
\beq\label{dt5}
R\pi\lsta^T \cV\lra \phi^{\ast}R\pi\lsta^T \cV' \lra \pi\lsta^T \bl\cL\otimes\cN\otimes\bL_1|_{\cC_{e}}(-\cQ_a)\br
\mapright{+1},
\eeq
where the last term is isomorphic to $\cO_{\cW_{\ga}}$. \
Indeed, parallel to the proof of Lemma \ref{G-fixed}, we show that
$$R\pi\lsta^T \cL^{\log}\cong \phi^{\ast}R\pi\lsta^T \cL^{\prime\log},\ \
R\pi\lsta^T \cP^{\log}\cong \phi^{\sta} R\pi\lsta^T \cP^{\prime\log},\ \
R\pi\lsta^T \cN_{[v]}\cong \phi^{\ast} R\pi\lsta^T \cN'.
$$
Next, let $\cQ'_a\sub(\cC_{e})^T$ be the section of the $T$-fixed locus other than $\cQ_a$.
Since $\nu_2|_{\cQ_a'}$ is $T$-invariant and nowhere vanishing, we have
$$R\pi\lsta^T \bl\cL\otimes\cN\otimes\bL_1|_{\cC_{e}}(-\cQ_a)\br\cong R\pi\lsta^T \cO_{\cQ_a'}
\cong\cO_{\cW_{\ga}}.
$$
Thus, \eqref{dt5} follows from applying $R\pi\lsta^T$ to the short exact sequence
$$
0\to\cL\otimes\cN\otimes\bL_1\to \eta\lsta(\cL'\otimes\cN'\otimes\bL_1)
\to \cL\otimes\cN\otimes\bL_1|_{\cC_{e}}(-\cQ_a)\to0.
$$

We then form the commutative square \eqref{WDga}; form the commutative diagram of d.t.s
\eqref{hugge}, where the upper line is \eqref{dt5}. Then a direct calculation shows that
$q_\ga\sta\TT_{\cD_\ga/\cD_{\ga'}}\cong \cO_{\cW_\ga}$, and $\alpha$ (in \eqref{hugge}) is
an isomorphism. Repeating the same argument succeeding \eqref{hugge}, we conclude that
$\phi\sta[\cW_{\ga'}]\virtloc=[\cW_\ga]\virtloc$.
\end{proof}

\begin{proof}[Proof of Proposition \ref{prop-loc}]
Let $\gad$ be the result from $\ga$ after decoupling as done before Proposition \ref{Phi0}.
By Lemma \ref{virvir},
\begin{eqnarray}\label{VirCycle}
[\cM_\ga]\virtloc=[\cW_{\gad}]\virtloc=\prod_{A\in (\gad)_{\text{conn}}}[\cW_A]\virtloc.
\end{eqnarray}
Since $\ga$ is regular, $E_{0\infty}=\emptyset$. Thus any
$A\in (\gad)_{\text{conn}}$ takes the following form: it is a vertex $v\in V_0\cup V^S_1\cup V_\infty$
with certain edges $e$ attached. 
Like before, we denote such $A$ by $[v]$.
Specifically, when $v\in V_0\cup V_\infty$, $[v]$ is $v$ union with
all edges in $E_v$; when $v\in V_1$, $[v]$ is $v$ union with the leaf edges in $E_v$.
Thus \eqref{VirCycle} takes the form
\beq\label{W-cycle}
[\cM_\ga]\virtloc=\prod_{v\in V_0\cup V_1^S\cup V_\infty} [\cW_{[v]}]\virtloc.
\eeq

We analyze $[\cW_{[v]}]\virtloc$. In case $v\in V^S\cup V_0^U$, by Lemma \ref{L3.10} and that $\cW_e$ is a
$G_e$-gerbe over a point,
$$[\cW_{[v]}]\virtloc=\frac{1}{\prod_{e\in E_{[v]}} |G_e|} [\cW_v]\virtloc.
$$
In case $v\in V^U_\infty$, then a direct calculation shows that $[\cW_{[v]}]\virtloc=[\cW_{[v]}]$,
which is $(\prod_{e\in E_{[v]}} |G_e|)\upmo [pt]$.
Combined, this proves Proposition \ref{prop-loc}.
\end{proof}

For future references, we list explicitly the cycles $[\cW_v]\virtloc$ appear in Prop. \ref{prop-loc}.

\begin{prop}[cf. Prop. \ref{fix2}]\label{fix25}
The cycles $[\cW_v]\virtloc$ are given by
\begin{enumerate}
\item when $v\in V_0^S$, $[\cW_{v}]\virtloc=[\Mbar_{g_v, E_v\cup  
S_v}(\PP^4,d_v)^p]\virtloc
= (-1)^{5d_v+1-g_v}[\Mbar_{g_v, E_v\cup S_v}(Q_5,d_v)]\virt$ is the  
GW-invariants;
\item when $v\in V_1^S$, $[\cW_{v}]\virtloc=[\Mbar_{g_v,E_v\cup S_v}]$;
\item when $v\in V_\infty^S(\ga)$,
$[\cW_{v}]\virtloc= [\Mbar_{g_v,\gamma_v}^{1/5,5p}]\virtloc$ is the FJRW invariants.
\end{enumerate}
\end{prop}

\begin{lemm}\label{v0-1}
Let either $v\in V^U_0$ or $v\in V_0^S$ so that $g_v=0$ and $d_v=0$. Then
\beq\label{W0-2}[\cW_{v}]\virtloc=[-Q_5]\times[\overline M_{0,E_v\cup S_v}].
\eeq
When $|E_v\cup S_v|\le2$, we agree $[\overline M_{0,E_v\cup S_v}]=[pt]$.
\end{lemm}

\begin{proof}
Without loss of generality, we can assume that $\ga=[v]$. Let $(\cC,\Si^\cC,\cdots)$ with $\pi:\cC\to\cW_\ga$
be the universal family of $\cW_\ga$; let $\cC_v\sub\cC$ be the subfamily associated with the vertex
$v$: in case $v$ is unstable, $\cC_v$ is a section of $\cC\to\cW_\ga$;
in case $v$ is stable, $\cC_v$ is a family of genus zero curves.
In the following we assume $n_v=|E_v\cup S_v|\ge 3$. (The argument for the cases $n_v\le 2$ can easily be modified.)

First, the sections $(\varphi_{1},\cdots,\varphi_{5})|_{\cC_v}$ define a morphism
$\varphi_v: \cC_v\to\Pf$.
Because $d_v=0$, $\varphi_v$ is fiberwise constant along fibers of $\cC_v\to \cW_\ga$,
thus factors through a $\ti\varphi_v:\cW_\ga\to\Pf$. Following the proof of Lemma \ref{fix4}, we have
$\varphi_v\sta\sO_{\Pf}(1)\cong \cL|_{\cC_v}$. Because $g_v=0$,
$R^1\pi\lsta^T \cL=R^1\pi\lsta (\cL|_{\cC_v})=0$.

Next, by (2) of Lemma \ref{G-fixed},  
we get
$$R^1\pi_{\ast}^T \cP^{\log}=R^1\pi\lsta ((\cL_v|_{\cC_v})^{-5}\otimes\omega_{\cC_v/\cW_\ga})=
\pi\lsta \cL_v^{-5}\cong \ti\varphi_v\sta\sO_{\Pf}(-5).
$$
Thus $\cW_{\ga}$ has the obstruction sheaf $\ti\varphi_v\sta\sO_{\Pf}(-5)$.
This proves \eqref{W0-2}.
\end{proof}

\begin{rema}\label{v0-2}
Let $v\in V_0^S$ be such that $g_v\ge 1$ and $d_v=0$. Then 
$[\cW_{v}]\virtloc=e(\cH\dual\boxtimes T_{Q_5})$, 
where $\cH$ is the Hodge bundle of $\overline{M}_{g_v,E_v\cup S_v}$, and $e(\cH\dual\boxtimes T_{Q_5})$ is the Euler class of
the pullback of $\cH\dual$ and $T_{Q_5}$ on $\overline{M}_{g_v,n_v}\times Q_5$.
\end{rema}

%
%
%

\section{Virtual localization, part 2}\label{sec:localization-II}
In this section, we fix a $\Ga\in \Delta^{\text{reg}}$ and will work out the contribution to the localization formula
from the moving parts associated to $\Ga$. We  abbreviate $\cW=\cW\lggd$.
Let $\xi\in \cW_\Ga$, thus $\sC=\cC_\ga\times_{\cW}\xi$, etc..
Denote $\sV=\cV_\ga\otimes_{\sO_{\cC_\ga}}\sO_\sC$. We form
\begin{eqnarray*}
&B_1 = \Aut({\sC},\Si^{\sC})= \Ext^0(\Omega_{\sC}(\Si^{\sC}),\cO_{\sC}),
&B_2 = \Aut({\sL}) \oplus \Aut({\sN})= H^0(\cO_{\sC}^{\oplus 2}),\\
&B_3 = \Def(\varphi,\rho, (\nu_1,\nu_2)) = H^0({\sV}),
&B_4 = \Def({\sC},\Si^{\sC})= \Ext^1(\Omega_{\sC}(\Si^{\sC}),\cO_{\sC}),\\
&B_5 = \Def({\sL})\oplus \Def({\sN})= H^1(\cO_{\sC}^{\oplus 2}),
&B_6  = \Obs(\varphi,\rho, (\nu_1,\nu_2)) = H^1({\sV}),
\end{eqnarray*}

All $B_i$ are $T$-spaces. Let $B_i\umv$ be the moving parts of $B_i$. Then
the virtual normal bundle $N^\vir$ to $\cW_\Ga$ in $\cW$ restricted at $\xi$ is
$$
N^\vir |_\xi= T_\xi\umv-Ob_\xi\umv = -B_1\umv - B_2\umv + B_3\umv + B_4\umv + B_5\umv -B_6\umv.
$$
It will be clear later that the $B_i\umv$ for $\xi\in \cW_\Ga$ forms a vector bundle over $\cW_\Ga$.
By abuse of notation, we will view $B_i\umv$ as such a vector bundle. Then $\Gm$-equivariant Euler class 
$e_{\Gm}(N^\vir)$ of the virtual normal bundle is given by
\begin{equation}\label{eqn:virtual-normal}
\frac{1}{e_{\Gm}(N^\vir)}=\frac{e_{\Gm}(B_1\umv)e_{\Gm}(B_2\umv)
	e_{\Gm}(B_6\umv)}{e_{\Gm}(B_3\umv)e_{\Gm}(B_4\umv)e_{\Gm}(B_5\umv)}.
\end{equation}

The goal of this section 
is to derive an explicit formula of
$\displaystyle{\frac{1}{e_{\Gm}(N^\vir)} }$.

\subsection{The moving part of deforming $\Si^{\sC}\sub {\sC}$}\label{sec:curve}
\def\ggga{}

In this subsection, we compute
$$
{e_T(B_1\umv)}\cdot {e_T(B_4\umv)}^{-1}.
$$

Recall our conventions that $F=F(\Ga)$, $V^{0,2}=V^{0,2}(\ga)$, etc., and  on nodes:
$$\forall\, (e,v)\in F:\ y_{(e,v)}={\sC}_v\cap {\sC}_e;\quad \forall\, v\in V^{0,2}\ \text{and}\ E_v=\{e,e'\}:\ y_{(e,v)}=\sC_e\cap\sC_{e'}.
$$
Set $F^{0,1}=\{(e,v)\in F: v\in V^{0,1}\}$ and $F^S=\{(e,v)\in F:v\in V^S\}$. Then 
\begin{eqnarray*}
B_1\umv &=& \bigoplus_{(e,v)\in F^{0,1} } T_{\sC_v} \sC_e, \\
B_4\umv &=& \bigoplus_{(e,v)\in F^S }T_{y_{(e,v)}} \sC_e \otimes T_{y_{(e,v)}} \sC_v
\oplus \bigoplus_{\substack{v\in V^{0,2}\\ E_v=\{e,e'\}} } T_{y_{(e,v)}} \sC_e\otimes T_{y_{(e',v)}} \sC_{e'}.
\end{eqnarray*}
Let $(e,v)\in F$, then $T_{y_{(e,v)}} \sC_e$ forms a line
bundle over $\cW_e$. In case $v\in V^S$, then $T_{y_{(e,v)}}\sC_v$  forms a line
bundle over $\cW_v$.
By abuse of notations, we will view $T_{y_{(e,v)}}\sC_e$, $T_{y_{(e,v)}}\sC_v$, and
$T_{\sC_v}\sC_e$ as such line bundles. As $T$ acts trivially on $\cW_e$ and $\cW_v$,
$$
H^2_T(\cW_e;\QQ)=H^2(\cW_e;\QQ)\oplus \QQ \ft,\quad
H^2_T(\cW_v;\QQ)=H^2(\cW_v;\QQ)\oplus \QQ \ft.
$$
As $T$ acts trivially on $T_{y_{(e,v)}}\sC_v$,
$$
e_T(T_{y_{(e,v)}}\sC_v)= e(T_{y_{(e,v)}}\sC_v)= -\psi_{(e,v)},
$$
where $\psi_{(e,v)}\in H^2(\cW_v;\QQ)$ is the $\psi$-class associated to the pointed curves $y_{(e,v)}\in \sC_v$.
For  $(e,v)\in F$, let $\bT_{(e,v)}$ be the line bundle
$\bT_{(e,v)}= T_{y_{(e,v)}}\sC_e$ over $\cW_e$ 
with
$$
w_{(e,v)}:= e_T(\bT_{(e,v)})\in H^2_T(\cW_e;\QQ).
$$

$$
e_T(B_1\umv)= \prod_{(e,v)\in F^{0,1}} w_{(e,v)},
$$
$$
e_T(B_4\umv)=\prod_{(e,v)\in F^S} (w_{(e,v)}-\psi_{(e,v)})
\cdot\prod_{\substack{v\in V^{0,2},  E_v=\{e,e'\}} } (w_{(e,v)}+w_{(e',v)}),
$$


It remains to determine $w_{(e,v)}$ for $(e,v)\in F$. The formula
of $w_{(e,v)}$ will be given in Lemma \ref{wev} below.
To state Lemma \ref{wev}, we introduce some definitions.
 
In case $e\in E_0$,
let $\cE_e \to \cW_e=\sqrt[d_e]{\cO_{\PP^4}(1)/\PP^4}$ be the tautological  bundle such that
$\cE_e^{\otimes d_e}= \pi_e^*\cO_{\PP^4}(1)$,
where $\pi_e: \cW_e\to \PP^4$ is the projection to its coarse moduli space. Let $h\in H^2(\PP^4;\QQ)$ be the hyperplane
class, and let $h_e=\pi_e^*h \in H^2(\cW_e;\QQ)$.  

For $e\in E_0$,  let
$v\in V_0$ and $v'\in V_1$ be the two vertices of $e$.
For $e\in E_\infty$, let $v\in V_\infty$ and
$v'\in V_1$ be the two vertices of $e$.
Set
\begin{equation}\label{eqn:delta}
\delta' = \begin{cases}
-1, & v'\in  V^{0,1},\\
0, & v'\in V\setminus V^{0,1},
\end{cases}\quad
\delta  = \begin{cases}
-1, & v\in V^{0,1},\\
0, & v\in V\setminus V^{0,1}.
\end{cases}
\end{equation}
For $e\in E_\infty$,  set
\begin{equation}\label{eqn:re}
r_e=
1 \hbox{ when }d_e \in \ZZ; \quad \hbox{or } 
r_e=5 \hbox{ when } d_e\notin \ZZ.
\end{equation}
Then $\sC_e\cong \PP(r_e,1)$. Because $\Gamma$ is regular,
if $v\in V^S_\infty$, then $d_e\neq \ZZ$ and $r_e=5$. Note that
if $v\in V^{0,1}_\infty$, then $d_e\in \ZZ$ and $r_e=1$.

\begin{lemm}\label{wev} Let $v$ and $v'$ be the two vertices of an edge $e$.
\begin{enumerate}
\item When $v\in V_0$, then $w_{(e,v)}=\frac{h_e +\ft}{d_e}$ and $ w_{(e,v')}=-\frac{h_e+\ft}{d_e}$.
\item When $v\in V_\infty\setminus V_\infty^{0,1}$, then $w_{(e,v)}=\frac{\ft}{r_e d_e}$ and $ w_{(e,v')}=-\frac{\ft}{d_e}$.
\item When $v\in V_\infty^{0,1}$, then  $w_{(e,v)}=\frac{5\ft}{5d_e+1}$ and $ w_{(e,v')}=\frac{-5\ft}{ 5d_e+1}$.
\end{enumerate}
\end{lemm}

\begin{proof} We begin with the case $e\in E_0$.
For  $\cE_e$ and $h_e$, we have
$$
c_1(\cE_e)= {h_e}/{d_e} \in H^2(\cW_e;\QQ),
$$
and 
$\bT_{(e,v)} \cong \cE_e\otimes \bL_{1/d_e}\cong \bT_{(e,v')}^\vee$
as $T$-equivariant line bundles over $\cW_e$, where $\cE_e$ is equipped with
the trivial $T$-equivariant structure. Therefore
$w_{(e,v)}=e_T(\bT_{(e,v)})$ and $w_{(e,v')}$ have the expressions stated in part (1) of the Lemma.

Let $e\in E_\infty$. The coarse moduli of $\cW_e$ is a point, so
$H^2_T(\cW_e;\QQ)=\QQ \ft$. We have 
$$
w_{(e,v')}=e_T(T_{y_{(e,v')}}\sC_e),\and
w_{(e,v)}= e_T(T_{y_{(e,v)}}\sC_e) =  - {w_{(e,v')}}/{r_e},
$$
where $r_e$ is defined in \eqref{eqn:re}.   Let $\sP={ \sL^{\vee\otimes 5}\otimes \omega_\sC^{\log}}$.
By Remark \ref{weights},
$$
\sP|_{y_{(e,v)}} \cong (\sL\otimes \sN\otimes \bL_{1})|_{y_{(e,v)} }\cong
(\sL\otimes \sN\otimes \bL_{1})|_{y_{(e,v')} }\cong \sN|_{y_{(e,v')} } \cong \bL_0.
$$
Then for  $\delta$ and $\delta'$ defined in \eqref{eqn:delta}, we have 
$$
\sP|_{\sC_e}\cong \sL^{\vee\otimes 5}|_{\sC_e}\otimes
\cO_{\sC_e}(\delta y_{(e,v)} +\delta' y_{(e,v')}),
$$
\begin{eqnarray*}
e_T(\sP|_{y_{(e,v')}}) &=& -5 e_T(\sL|_{y_{(e,v')} })+\delta' w_{(e,v')} = 5\ft + \delta'  w_{(e,v')},\\
e_T(\sP|_{y_{(e,v')}}) &=& e_T(\sP|_{y_{(e,v)}}) + \deg(\sP|_{\sC_e}) w_{(e,v')}
= 0+  (-5d_e+ {{\delta}/{r_e} +\delta'})w_{(e,v')}.
\end{eqnarray*}
Therefore  $w_{(e,v')}= \frac{5\ft}{-5d_e+\delta/r_e}= \frac{5\ft}{-5d_e+\delta}$.
(Since $\delta=-1$ implies $d_e\in \ZZ\Leftrightarrow r_e=1$.)
So when $v\in V_\infty\setminus V_\infty^{0,1}$, then
$w_{(e,v')}=-\frac{\ft}{d_e}$ and $w_{(e,v)}= \frac{\ft}{r_e d_e}$;
when $v\in V^{0,1}_\infty$, then $w_{(e,v')} = \frac{-5\ft}{5d_e+1}$ and
$w_{(e,v)}=\frac{5\ft}{5d_e+1}$. This proves part (2) and part (3).
\end{proof}

For references, if $v\in V_\infty$, then
\begin{equation}\label{infty}
\begin{cases}
 e_T(\sL|_{y_{(e,v')}}) = -\ft, \quad  e_T(\sN|_{y_{(e,v')}})=0, \quad
e_T(\sP|_{y_{(e,v')}}) = (5-\displaystyle\frac{5\delta'}{5d_e-\delta})\ft , \\
 e_T(\sL|_{y_{(e,v)}})  
 =\frac{\delta \ft}{5d_e+1},\quad 
e_T(\sN|_{y_{(e,v)}})  =\displaystyle\frac{-5d_e\ft}{5d_e-\delta},  \quad e_T(\sP|_{y_{(e,v)}})  = 0. \\

\end{cases}
\end{equation}

\subsection{The moving part of deforming $(\sL, \sN, \varphi,\rho,\nu)$}\label{sec:fields}
We compute
\begin{equation}\label{eqn:BB}
{e_{\Gm}(B_2\umv)e_{\Gm}(B_6\umv)}\cdot {e_{\Gm}(B_3\umv)^{-1}e_{\Gm}(B_5\umv)}^{-1}.
\end{equation}

Because $\xi$ is $T$-invariant, the spaces $H^i(\sL(-\Si^{\sC}\lophi))$,
$H^i(\sP(-\Si^{\sC}\lorho))$, etc., are
$T$-vector spaces. As $T$ acts by scaling $\nu_1$, 
 as $T$-vector space   
$$
H^i(\sV)=H^i( \sL(-\Si^{\sC}\lophi) )^{\oplus 5} \oplus H^i(\sP(-\Si^{\sC}\lorho) )
\oplus H^i(\sL\otimes\sN)\otimes\bL_{1}\oplus H^i(\sN).
$$

To compute \eqref{eqn:BB}, we need to study the moving parts: 
$$\bl H^0(\sV)-H^0(\cO_{\sC}^{\oplus 2})\br\umv
\and \bl H^1(\sV)-H^1(\cO_{\sC}^{\oplus 2})\br\umv.
$$
To achieve this, we use the two long exact sequences of $T$-representations:
$$
\lra H^i({\sV})\lra \bigoplus_{v\in V^S\cup E}H^i({\sV}|_{{\sC}_v})
\lra \bigoplus_{a \in F^S \cup V^{0,2}\ggga} H^i({\sV}|_{y_a})\lra H^{i+1}(\sV)\lra,
$$
$$
\lra H^i(\cO_{\sC}^{\oplus 2})\lra
\bigoplus_{v\in V^S\cup E}H^i(\cO_{{\sC}_v}^{\oplus 2})\lra
\bigoplus_{a\in F^S\cup V^{0,2} \ggga}H^i( \cO^{\oplus 2}_{y_a})
\lra  H^{i+1}(\cO_{\sC}^{\oplus 2})\lra.
$$
We have
\begin{equation}\label{eqn:BBBB}
\frac{e_{\Gm}(B_2\umv)e_{\Gm}(B_6\umv)}{e_{\Gm}(B_3\umv)e_{\Gm}(B_5\umv)}
=\prod_{v\in V^S} A'_v \prod_{e\in E} A'_e \prod_{a\in F^S\cup V^{0,2}} A'_a,
\end{equation}
where $A'_v$, $A'_e$ and $A'_a$ are contributions from
$\bl H^i(\sV|_{\sC_v})-H^i(\cO^{\oplus 2}_{\sC_v})\br\umv$ ($i=0,1$),
$\bl H^i(\sV|_{\sC_e})-H^i(\cO^{\oplus 2}_{\sC_e})\br\umv$ ($i=0,1$),
and $\bl H^0(\sV|_{y_a})-H^0(\cO^{\oplus 2}_{y_a})\br\umv$ respectively.
We will derive  formulae of $A'_v$,
$A'_e$, and $A'_a$ in \S\ref{sec:vertex}, \S \ref{sec:edge}, and \S \ref{sec:node},
respectively.

\subsubsection{Contribution from stable vertices} \label{sec:vertex}
We first introduce some notations:
\begin{itemize}
\item Given a stable vertex $v\in V^S$, let $\pi_v: \cC_v\to \cW_v$ be the universal curve; let $\cL_v$ and $\cN_v$ be
the universal line bundle over $\cC_v$.
\item Given $v\in V^S_0$, let $\phi_v: \cC_v\to \PP^4$ be defined as in \eqref{eqn:phi-v}.
\item Given $v\in V^S_1$, let $\EE_v:= \pi_*\omega_{\pi_v}$ be the Hodge bundle, where
$\omega_{\pi_v}\to \cC_v$ is the relative dualizing sheaf. Then $\EE^\vee_v=R^1\pi_{v*}\cO_{\cC_v}$.
\end{itemize}
The contribution $A'_v$ from a stable vertex $v\in V^S$ is given by
the following lemma.
\begin{lemm}\label{lem:VS}
\begin{equation}\label{eqn:Av}
A'_v=\begin{cases}
\displaystyle{ \frac{1}{e_T(R\pi_{v*} \phi_v^* \cO_{\PP^4}(1)\otimes \bL_1)} },
& v\in V^S_0; \\
& \\
\displaystyle{\Big(\frac{e_T(\EE_v^\vee \otimes \bL_{-1})}{-\ft}\Big)^5
	\cdot \frac{5\ft}{e_T(\EE_v\otimes \bL_5)}
	\cdot(\frac{1}{5\ft})^{|E_v|}(\frac{-\ft^4}{5})^{|S_v^{(1,\varphi)}|} , }
& v\in V_1^S;\\
\displaystyle{ \frac{1}{e_T(R\pi_{v*}\cL_v^\vee \otimes \bL_{-1})} }, & v\in V_\infty^S.
\end{cases}
\end{equation}
\end{lemm}
\begin{proof}
We claim $ \big(H^i( \sV|_{\sC_v})- H^i( \cO_{\sC_v}^{\oplus 2})\big)\umv 
$
\begin{equation}\label{eqn:CvV}
\begin{aligned}
\quad =\begin{cases}
H^i(\sL|_{\sC_v})\otimes \bL_1= H^i(\sC_v, f_v^*\cO_{\PP^4}(1))\otimes \bL_1, & v\in V^S_0;\\
(H^i( \cO_{\sC}({-}\Si^{\sC}\lophi)|_{\sC_v})\otimes \bL_{-1})^{\oplus 5} \oplus
H^i( \omega_{\sC}(\Si^{\sC}\lophi)|_{\sC_v})\otimes \bL_{5}, &v\in V^S_1;\\
H^i( \sL^\vee|_{\sC_v})\otimes \bL_{-1}, &v\in V_\infty^S.
\end{cases}
\end{aligned}
\end{equation}
We now derive these formulae.

\noindent
(1.) In case $v\in V^S_0$, $\Si^{\sC_v}\subset \Si^{\sC}\lorho$ and $(\nu_1,\nu_2)|_{ \sC_v}= (0,1)$.
So $\sL(-\Si^{\sC}\lophi)|_{\sC_v}\cong \sL|_{\sC_v}$, ${\sN}|_{{\sC}_v}\cong \cO_{ {\sC}_v}$, and
$\omega_{\sC}^{\log}(-\Si^{\sC}\lorho)|_{\sC_v} =\omega_{\sC}|_{\sC_v}$. Thus,
$$
H^i(\sV|_{\sC_v}) = H^i( \sL|_{\sC_v})^{\oplus 5} \oplus H^i(\sL^{\vee\otimes 5}\otimes \omega_{\sC}|_{\sC_v} )\\
\oplus H^i(\sL|_{ \sC_v})\otimes \bL_1 \oplus H^i(\cO_{\sC_v}).
$$
Taking the moving parts, we obtain the formula \eqref{eqn:CvV} in case $v\in V_0^S$, which implies
the formula of $A'_v$ in case $v\in V_0^S$.
\vsp

\noindent
(2.) In case $v\in V^S_1$,  $\Si^{\sC_v}\subset \Si^{\sC}\lorho \cup \Si^{\sC}\lophi$ and
$(\nu_1,\nu_2)|_{{\sC}_v}=(1, 1)$.
So $(\sL\otimes \sN)|_{\sC_v}\otimes \bL_1\cong \cO_{\sC_v}$ and
$\sN|_{\sC_v}\cong \cO_{\sC_v}$, which implies $\sL|_{\sC_v}\cong \cO_{\sC_v}\otimes \bL_{-1}$.
We have $S_v = S_v^{(1,\rho)}\cup S_v^{(1,\varphi)}$ where $S_v^{(1,\rho)}$ and
$S_v^{(1,\varphi)}$ consist of markings in $\Si^{\sC_v}\lorho:=\sC_v\cap \Si^{\sC}\lorho$
and $\Si^{\sC_v}\lophi:=\sC_v\cap \Si^{\sC}\lophi$, respectively.
Since $\omega_{\sC}^{\log}(-\Si^{\sC}\lorho)|_{\sC_v} =\omega_{\sC}(\Si^{\sC}\lophi)|_{\sC_v}$,
$$
H^i( \sV|_{\sC_v}) = (H^i(\cO_{\sC_v}(-\Si^{\sC_v}\lophi))\otimes \bL_{-1})^{\oplus 5} \oplus
H^i( \omega_{\sC}(\Si^{\sC_v}\lophi)|_{\sC_v})\otimes \bL_{5} \oplus H^i(\cO_{\sC_v})^{\oplus 2}.
$$
Taking the moving parts, we obtain the formula \eqref{eqn:CvV} in case $v\in V_1^S$. To obtain the
formula of ${A'_v}$ in case $v\in V_1^S$, we need further simplification.

 Taking cohomologies of $
0\lra \cO_{\sC_v}(-\Si^{\sC_v}\lophi) \lra \cO_{\sC_v} \lra \cO_{\Si^{\sC_v}\lophi}\lra 0
$ we have
$$
		\big( H^0(\cO_{\sC_v}(-\Si^{\sC}\lophi)) -H^1(\cO_{\sC_v}(-\Si^{\sC}\lophi))\big)\otimes \bL_{-1}
		 = -H^1(\cO_{\sC_v}) \otimes \bL_{-1} - \bL_{-1}^{ \oplus (|S_v^{(1,\varphi)}|-1) }.
$$
By the similar argument as above, we also get
\begin{equation*}
	\begin{aligned}
		 \big(H^0(\omega_{ {\sC} }(\Si^{\sC}\lophi)|_{ {\sC}_v})
		-H^1(\omega_{ {\sC}}(\Si^{\sC}\lophi)|_{ {\sC}_v})\big)\otimes \bL_{5}
		 = H^0(\omega_{ {\sC}_v})\otimes \bL_{5}\oplus \bL_{5}^{\oplus(|S_v^{(1,\varphi)}|+|E_v|-1)}.
	\end{aligned}
\end{equation*}
Combining \eqref{eqn:CvV} (in case $v\in V_1^S$) and the above two formulae, 
we obtain the formula of ${A'_v}$ in case $v\in V_1^S$.
\vsp

\noindent
(3.) In case $v\in V^S_\infty$, $\Si^{\sC_v}\subset \Si^{\sC}_{\zeta_5}\cap\Si^{\sC}_{\zeta_5^2}$
since $\Gamma$ is regular and $(\rho,\nu_1)|_{{\sC}_v}=(1,1)$,
so  $\sL(-\Si^{\sC}\lophi)|_{\sC_v}=\sL|_{\sC_v}$, $\sP(-\Si^{\sC}\lorho)|_{\sC_v}
=\sP|_{\sC_v}\cong \cO_{{\sC}_v}$,
and $(\sL\otimes \sN)|_{\sC_v} \otimes \bL_1 \cong \sO_{\sC_v}$,
equivalent to $\sN|_{\sC_v}\cong \sL^\vee|_{\sC_v}\otimes \bL_{-1}$.
Thus
$$
H^i(\sV|_{\sC_v}) =H^i( \sL|_{\sC_v})^{\oplus 5} \oplus
H^i(\sO_{\sC_v}) \oplus H^i(\cO_{ \sC_v})\oplus H^i(\sL^\vee|_{\sC_v})\otimes \bL_{-1}.
$$
Taking the moving part, we get the formula \eqref{eqn:CvV} in case $v\in V_\infty^S$, which
implies the formula of ${A'_v}$ in case $v\in V_\infty^S$.
\end{proof}

 \begin{rema} The factors for $v\in V_0^S, V_1^S$ in \eqref{eqn:Av} are used in \cite{NMSP2,NMSP3} to construct local twisted CohFTs. In \cite{NMSP2} the large $N$ technique is introduced to reduce factors of $v\in V_0^S$ into scalars. And in \cite{NMSP2,NMSP3} we compose  Grothendieck-Riemann-Roch $R$ matrices (of \eqref{eqn:Av}'s factors of $v\in V_1^S$) into the localization $R$ matrices, so as the  total $R$ matrices matches BCOV propagators.
 \end{rema}
 
\subsubsection{Contribution  $A'_e$ from an edge $e\in E(\Ga)$}\label{sec:edge}
  We use
$$
H^i(\sV|_{\sC_e}) = H^i(\sL(-\Si^{\sC}\lophi)|_{\sC_e})^{\oplus 5} \oplus
H^i(\sP(-\Si^{\sC}\lorho)|_{\sC_e} )
\oplus H^i( \sN\otimes \sL |_{\sC_e}) \otimes \bL_1\oplus H^i(\sN|_{\sC_e}).
$$
Let $v$, $v'$ and $\delta,\delta'$ be that in \eqref{eqn:delta}, and
$q=y_{(e,v)}$ and $q'=y_{(e,v')}$.  We introduce  \begin{equation}\label{eqn:delta-phi}
	\delta_\varphi'=\begin{cases}
		-1, & v'\in V_1^{1,1}, q' \in \Si^{\sC}\lophi,\\
		0, & \textup{otherwise};
	\end{cases}\
	\delta_\varphi=\begin{cases}
		-1, & v\in V_\infty^{1,1}, q \in \Si^{\sC}\lophi,\\
		0, & \textup{otherwise}.
	\end{cases}
\end{equation}
\begin{equation}\label{eqn:delta-rho}
	\delta_\rho'=\begin{cases}
		-1, & v'\in V_1^{1,1}, q' \in \Si^{\sC}\lorho,\\
		0, & \textup{otherwise};
	\end{cases}\
	\delta_\rho=\begin{cases}
		-1, & v\in V_0^{1,1}, q \in \Si^{\sC}\lorho,\\
		0, & \textup{otherwise}.
	\end{cases}
\end{equation}
With the above definition, we have $\omega^{\log}_{\sC}|_{\sC_e} = \cO_{\sC_e}(\delta q + \delta' q')$ and
\begin{eqnarray*}
	&&\cO_{\sC}(-\Si^{\sC}\lophi)|_{\sC_e} = \cO_{\sC_e}(\delta_\varphi q +\delta'_\varphi q'),
	\cO_{\sC}(-\Si^{\sC}\lorho)|_{\sC_e} = \cO_{\sC_e}(\delta_\rho q +\delta'_\rho q'),\\
	&&\sL(-\Si^{\sC}\lophi)|_{\sC_e} =  \sL|_{\sC_e}\otimes \cO_{\sC_e}(\delta_\phi q+ \delta'_\phi q'),\\
	&&\sP(-\Si^{\sC}\lorho)|_{\sC_e} = (\sL^{\vee\otimes 5})|_{\sC_e}\otimes \cO_{\sC_e}( (\delta+\delta_\rho)q+(\delta'+\delta'_\rho)q').
\end{eqnarray*}

\begin{lemm}[{Contribution from an edge in $E_0$}]\label{lem:E-zero}
Given an edge $e\in E_0$, let $v\in V_0$ and $v'\in V_1$ be the two ends of $e$. 
Then
	$$
	A'_e =\frac{\prod_{j=1}^{5d_e-1-\delta'-\delta'_{\rho}}(-5 h_e +\frac{j(h_e+\ft)}{d_e})}{
		\prod_{j=1}^{d_e+\delta'_{\varphi}}(h_e-\frac{j(h_e+\ft)}{d_e})^5 \cdot \prod_{j=1}^{d_e}\frac{j(h_e+\ft)}{d_e} }.
	$$
	Note that $\delta'$ and $\delta'_{\varphi}$ can be both zero and cannot be both nonzero.
\end{lemm}

\begin{proof}
	When $e\in E_0$,  $(\rho,\nu_2)|_{\sC_e}=(0,1)$, so $\sN|_{\sC_e}\cong \cO_{\sC_e}$.
	Let $T^{1/d_e}=\CC^*\to T=\CC^*$ be defined by $\tilde{t}\mapsto \tilde{t}^{d_e}$. Then
	$T^{1/d_e}$ acts on $\sC_e=\PP^1$ by $\tilde{t}\cdot[x,y]=[\tilde{t}x,y]$, and $H^2(B(T^{1/d_e});\ZZ)=\ZZ(\ft/d_e)$.
	Let $q=[0,1]$ and $q'=[1,0]$ be the two $T^{1/d_e}$-fixed points
	on $\sC_e$. Any $T^{1/d_e}$-linearlized line bundle on $\sC_e\cong \PP^1$ is of the form
	$\cO_{\sC_e}(aq + bq')$ for some $a,b\in \ZZ$, characterized by
	$$
	e_{T^{1/d_e}}\big(\sO_{\sC_e}(aq +bq')|_q\big) = \frac{a \ft}{d_e},\quad
	e_{T^{1/d_e}}\big(\sO_{\sC_e}(aq +bq')|_{q'}\big) =\frac{-b\ft}{d_e}.
	$$
	We have the following isomorphisms of
	$T^{1/d_e}$-linearlized line bundles on $\sC_e$:
	\begin{eqnarray*}
		&& T\sC_e \cong \cO_{\sC_e}(q+q'),\quad \sL(-\Si^{\sC}\lophi)|_{\sC_e}\cong \cO_{\sC_e}((d_e+\delta_\varphi')q') ,\\
		&& \sP(-\Si^{\sC}\lorho)|_{\sC_e}\cong \cO_{\sC_e}((\delta+\delta_\rho) q +(-5d_e+\delta' +\delta'_\rho)q').
	\end{eqnarray*}
	
	Recall that $\pi_e:\cW_e\to \PP^4$ is the projection to
	the coarse moduli space, and $\cE_e \to \cW_e$ is the
	tautological line bundle such that $\cE_e^{\otimes d_e} = \pi_e^*\cO_{\PP^4}(1)$.
	By abuse of notations, we view $H^i(\sL (-\Si^{\sC}\lophi)|_{\sC_e})$ and
	$H^i(\sP(-\Si^{\sC}\lorho)|_{\sC_e})$ as $T$-equivariant vector bundles
	over $\cW_e$. Then
	$$
	H^0(\sL(-\Si^{\sC}\lophi)|_{\sC_e}) = \bigoplus_{j=0}^{d_e +\delta_\varphi'}
	\cE_e^{\otimes (d_e-j)}\otimes (\bL_{-1/d_e})^{\otimes j},\
	H^1(\sL(-\Si^{\sC}\lophi)|_{\sC_e}) =0;
	$$
	$$
	H^0(\sP(-\Si^{\sC}\lorho)|_{ \sC_e}) = 0 ,\quad
	H^1(\sP(-\Si^{\sC}\lorho)|_{\sC_e}) =
	\bigoplus_{j= 1+\delta + \delta_\rho}^{5d_e-1 -\delta' -\delta_\rho'} \cE_e^{\otimes (-5d_e+j)}
	\otimes \bL_{j/d_e},
	$$
	and $H^0(\cO_{\sC_e})\umv= H^1(\cO_{\sC_e})\umv=0$.
	So for $e\in E_0$, we have three equations
	\begin{eqnarray*}
		\bl H^0( \sL(-\Si^{\sC}\lophi)|_{\sC_e})-H^1( \sL(-\Si^{\sC}\lophi)|_{\sC_e})\br\umv
		=\bigoplus_{j=1}^{d_e +\delta'_\varphi}\cE_e^{\otimes (d_e-j)}\otimes \bL_{-j/d_e};
	\end{eqnarray*}
	\begin{equation*}
		\big(H^0( \sN\otimes \sL|_{\sC_e})\otimes \bL_1-H^1( \sN\otimes \sL|_{\sC_e})\otimes \bL_1\big)\umv
		= \bigoplus_{j=1}^{d_e} \cE_e^{\otimes j} \otimes \bL_{ j/d_e };
	\end{equation*}
	\begin{equation*}
		\bl H^0(\sP(-\Si^{\sC}\lorho)|_{\sC_e})-H^1(\sP(-\Si^{\sC}\lorho)|_{\sC_e})\br\umv = -
		\bigoplus_{j=1}^{5d_e -1- \delta'-\delta'_\rho} \cE_e^{\otimes (-5d_e+j)} \otimes \bL_{j/d_e} .
	\end{equation*}
	The Lemma follows from  the three formulae above. 
\end{proof}

\begin{lemm}[Contribution from an edge in $E_\infty$]\label{lem:E-infty}
	Given an edge $e\in E_\infty$, let $v\in V_\infty$ and $v'\in V_1$ be the two ends of $e$. 
Then
	$$
	A'_e = \begin{cases}
	\displaystyle{\frac{\prod_{j=1+\delta_{\varphi}'}^{\lceil-d_e\rceil-1}(-\ft-\frac{j\ft}{d_e} )^5 }{
			\prod_{j=1}^{{-5d_e+\delta'+\delta'_{\rho}}}(-\frac{j\ft}{d_e})
			\prod_{j=1}^{\lfloor-d_e \rfloor}(\frac{j\ft}{d_e})}
	},&  v\in V\setminus V^{0,1};\\
	& \\
	\displaystyle{\frac{\prod_{j=1+\delta_{\varphi}'}^{-d_e{-1}}(-\ft-\frac{5j\ft}{5d_e+1} )^5 }{
			\prod_{j=1}^{-5d_e-1+\delta'+\delta'_{\rho}}(\frac{-5j\ft}{5d_e+1})
			\prod_{j=1}^{-d_e}(\frac{5j\ft}{5d_e+1})}
	},&  v\in V^{0,1}.
	\end{cases}
	$$
	Note that $\delta'$ and $\delta'_\rho$ can be both zero and cannot be both nonzero.
\end{lemm}
\begin{proof}
 We need to calculate $H^i( \cV|_{\sC_e})-H^i( \cO_{\sC_e}^{\oplus 2})$, which is
 $$H^i( \sL(-\Si^{\sC}\lophi)|_{\sC_e}\otimes \bL_{-1}  )^{\oplus 5}		+ H^i( \sP(-\Si^{\sC}\lorho)|_{\sC_e}) + H^i( \sN|_{\sC_e})  -H^i( \cO_{\sC_e}).$$
	Let $r_e\in \{1,5\}$ be defined as in \eqref{eqn:re}.

	\medskip
	
	\noindent
	{\it Case 1: $v\in V\setminus V^{0,1}$.}
	From \eqref{infty}, since $\delta =0$, we obtain
	$$
	e_T(\sL(-\Si^{\sC}\lophi)|_{q'})= (-1-\frac{\delta'_\varphi}{d_e})\ft, \quad
	e_T(\sL(-\Si^{\sC}\lophi)|_q )= \frac{\delta_\varphi \ft}{r_ed_e};
	$$
	$$
	e_T(\sP(-\Si^{\sC}\lorho)|_{q'} )= (5-\frac{\delta'+\delta'_\rho}{d_e})\ft,\quad
	e_T(\sP(-\Si^{\sC}\lorho)|_q) = 0.
	$$
	These imply 
	$$
	H^0(\sL(-\Si^{\sC}\lophi)|_{\sC_e})=0, \quad
	H^1(\sL(-\Si^{\sC}\lophi)|_{\sC_e}) =\bigoplus_{j=1 + \delta'_{\varphi} }^{\lceil -d_e \rceil-1 -\delta_\varphi} \bL_{-(1+j/d_e)},
	$$
	\begin{equation}\label{eqn:E-infty-I}
		\bl H^0(\sL(-\Si^{\sC}\lophi)|_{\sC_e})
		-H^1(\sL(-\Si^{\sC}\lophi)|_{\sC_e})\br\umv =- \bigoplus_{j=1 + \delta'_{\varphi} }^{\lceil -d_e \rceil-1} \bL_{-(1+j/d_e)}.
	\end{equation}  	We also have
	\begin{equation}\label{eqn:E-infty-II}
		H^0(\sN|_{\sC_e}) -H^0(\cO_{\sC_e})=
		\bigoplus_{j=1}^{\lfloor -d_e \rfloor} \bL_{j /d_e},\quad
		H^1(\sN|_{\sC_e}) -H^1(\cO_{\sC_e})= 0
	\end{equation}
	Finally, $H^0(\sP(-\Si^{\sC}\lorho)|_{\sC_e} ) = \oplus_{j=5d_e-\delta' -\delta'_\rho}^0 \bL_{j /d_e}$ and
	$H^1( \sP(-\Si^{\sC}\lorho)|_{\sC_e} ) = 0$.
	So
\begin{equation}\label{eqn:E-infty-III}
 \bl H^0(\sP(-\Si^{\sC}\lorho)|_{\sC_e} )-H^1(\sP(-\Si^{\sC}\lorho)|_{\sC_e} )\br\umv
= \bigoplus_{j=5d_e-\delta'-\delta'_\rho}^{-1} \bL_{j/d_e}
			=\bigoplus_{j=1}^{-5d_e+\delta' +\delta'_\rho} \bL_{-j/d_e}.
\end{equation}
	Combining \eqref{eqn:E-infty-I}, \eqref{eqn:E-infty-II},
	and \eqref{eqn:E-infty-III}, we obtain the Lemma in case $v\in V\setminus V^{0,1}$.
	
\medskip
	
	\noindent
	{\it Case 2: $v\in V^{0,1}$.}
	In this case, $\delta=-1$, $\delta_\varphi=0$,
	$-d_e\in \ZZ_{>0}$, $r_e=1$.
	Using the same computations as above, we get the required formula. \end{proof}

For later convenience, we rewrite the total contribution of all edges as follows.
\begin{lemm}[Contribution from all edges]\label{lem:all-edges}  We define $$
	A_e = \begin{cases}
	\displaystyle{ \frac{\prod_{j=1}^{5d_e-1}(-5 h_e +\frac{j(h_e+\ft)}{d_e})}{
			\prod_{j=1}^{d_e}(h_e-\frac{j(h_e+\ft)}{d_e})^5 \cdot \prod_{j=1}^{d_e}\frac{j(h_e+\ft)}{d_e} }},
	& e\in E_0; \\
	&\\
	\displaystyle{\frac{\prod_{j=1}^{\lceil-d_e\rceil-1}(-\ft-\frac{j\ft}{d_e} )^5 }{
			\prod_{j=1}^{{-5d_e}}(-\frac{j\ft}{d_e})
			\prod_{j=1}^{\lfloor-d_e \rfloor}(\frac{j\ft}{d_e})}
	},& e\in E_\infty,\ (e,v)\in F,\ v\in V_\infty \setminus V_\infty^{0,1};\\
	& \\
	\displaystyle{\frac{\prod_{j=1}^{-d_e{-1}}(-\ft-\frac{5j\ft}{5d_e+1} )^5 }{
			\prod_{j=1}^{{ -5d_e-1}}(\frac{-5j\ft}{5d_e+1})
			\prod_{j=1}^{-d_e}(\frac{5j\ft}{5d_e+1})}
	},& e\in E_\infty,\ (e,v)\in F, \ v\in V_\infty^{0,1}.
	\end{cases}
	$$
	then 
$$
	\prod_{e\in E}A'_e = \prod_{e\in E} A_e \prod_{v\in V_1^{0,1}\cup V_1^{1,1}} A'_v
	$$
where  $A'_v$ is    as follows.
 	$$
	A'_v= \begin{cases}
   	5\ft, & v\in V_1^{0,1},\hbox{ or } v\in V_1^{1,1},\ S_v\subset \Si^{\sC}\lorho;\\
 	-\ft^5, & v\in V_1^{1,1},\  S_v\subset \Si^{\sC}\lophi.\\
	\end{cases}
	$$
\end{lemm}
\begin{proof} It follows from Lem. \ref{lem:E-zero} and \ref{lem:E-infty},
	and the definitions of $\delta'$, $\delta'_{\varphi}$, and $\delta'_{\rho}$.
\end{proof}

\subsubsection{Contributions from nodes} \label{sec:node}
\begin{lemm}[Contribution from a flag in  $F^S$]\label{lem:FS}
	If $(e,v)\in F^S$ then
	$$
	A'_{(e,v)}=\begin{cases}
	h_e+\ft, & v\in V_0;\\
	-5\ft^6, & v\in V_1;\\
	1, & v\in V_\infty.
	\end{cases}
	$$
\end{lemm}
\begin{proof}
	Let $y_{(e,v)}=\sC_e\cap \sC_v$ be the node associated to
	the flag $(e,v)\in F^S$.  It is a scheme point when $v\in V_0\cup V_1$;
	when $v\in V_\infty$, it is a stacky point $B\bmu_5$ because $d_e\notin\ZZ$ as $\Ga$ is regular. Note
	that $H^0(\sV|_{y_{(e,v)}})=H^0(\cO_{y_{(e,v)}}^{\oplus 2})$ when $y_{(e,v)}=B\bmu_5$.
	\begin{equation}\label{eqn:FS}
		\bl H^0(\sV|_{y_{(e,v)}})-H^0(\cO_{y_{(e,v)}}^{\oplus 2}) \br\umv=
		\begin{cases}
			\ev_{(e,v)}^*\cO_{\PP^4}(1)\otimes \bL_1, & v\in V_0;\\
			\bL_{-1}^{\oplus 5}\oplus \bL_{5}, & (e,v)\in V_1;\\
			0, & (e,v)\in V_\infty.
		\end{cases}
	\end{equation}
	When $v\in V_0$, $\cW_v$ is the moduli of stable maps to $\Pf$ with $p$-fields,
	and $ev_{(e,v)}$ is the evaluation map at the marking labeled by $e$.
	The Lemma follows from \eqref{eqn:FS}.
\end{proof}
\begin{lemm}[Contribution from a vertex in $V^{0,2}$] \label{lem:V-zero-two}
	If $v\in V^{0,2}$, then
	$$
	A'_v= \begin{cases}
	h_e +\ft = h_{e'}+\ft, & v\in V_0^{0,2}\ \text{and}\ E_v=\{e,e'\},\\
	-5\ft^6, & v\in V_1^{0,2},\\
	(-\ft)^{{6}\epsilon(d_e)} , & v\in V_\infty^{0,2}\ \text{and}\ E_v=\{e,e'\}.\\
	\end{cases}
	$$
	Here we define $\ep(x)=1$ when $x\in \ZZ$, and $\ep(x)=0$ otherwise.
\end{lemm}
\begin{proof}
	If $v\in V^{0,2}(\Ga)$ and $E_v=\{e,e'\}$,  then $\sC_v=\sC_e\cap \sC_{e'}$ is a node,
	which is a stacky point $B\bmu_5$ if and only if $v\in V^{0,2}_\infty(\Ga)$ and $d_e\neq \ZZ$.
	(The balance condition implies $d_e+d_{e'}\in\ZZ$.) We have
	$$
	\bl H^0(\sV|_{\sC_v})-H^0(\cO_{\sC_v}^{\oplus 2})\br\umv=
	\begin{cases}
	\cO_{\PP^4}(1)\otimes \bL_1 , & v\in V_0^{0,2} \textup{ (so that $\cW_v=\PP^4$)}; \\
	\bL_{-1}^{\oplus 5}\oplus \bL_{5}, & v\in V_1^{0,2};\\
	0, & v\in V_\infty^{0,2},\ d_e\notin \ZZ ;\\
	\end{cases}
	$$
	This proves the Lemma. 
\end{proof}
\begin{theo}\label{final}
We define $A'_v$ by \eqref{eqn:Av} if $v\in V^S$. Let
	$A'_v$ be given by  Lemma \ref{lem:all-edges} (resp. Lemma \ref{lem:V-zero-two}) if $v\in V_1^{0,1}\cup V_1^{1,1}$ (resp. $v\in V^{0,2}$). If
	$v\in V_0^{0,1}\cup V_0^{1,1}\cup V_\infty^{0,1}\cup V_\infty^{1,1}$, we define $A'_v=1$.
	For $e\in E$, let $A_e$ be as in Lemma \ref{lem:all-edges}. Then
	$$
	\frac{1}{e_T(N^\vir_{\vGa})} =\prod_{(e, v)\in F^S} \frac{A'_{(e,v)} } {w_{(e,v)}-\psi_{(e,v)}}\prod_{v\in V(\Ga)}A_v \prod_{e\in E(\Ga)} A_e,
	$$
	where
	$$
	A_v=\begin{cases}
	A'_v\ , &  v\in V^S \hbox{ or }v\in V^{1,1};\\
	\displaystyle{ \frac{A'_v}{w_{(e,v)}+w_{ (e',v)}} }, & v\in V^{0,2}, E_v=\{e,e'\};\\
	A'_v\cdot w_{(e,v)}, & v\in V^{0,1}, E_v=\{e\}.
	\end{cases}
	$$
\end{theo}

\def\Gm{{\mathbb G_m}}
\section{Effective relations between GW and FJRW of quintics}
We will show that for a regular graph $\Ga\in \Delta^{\text{reg}}$, the contributions to $[\cW_\ga]\virtloc$
from vertices in $V^S_\infty(\Ga)$ are (dual-twisted) FJRW invariants, 
and those from vertices in $V_0^S(\Ga)$ are (twisted) GW invariants of the quintic.
We then investigate the relations obtained from MSP field theory between GW invariants of the quintic threefold and
FJRW invariants of $([\CC^5/\ZZ_5],\fw_5)$.

\subsection{FJRW invariants of $([\CC^5/\ZZ_5],\fw_5)$}

 The  group $\bmu_5$ 
 acts on $\CC^5$
 by diagonal action.  The function $x_1^5+\cdots+x_5^5$ on $\CC^5$ is invariant under such $\bmu_5$ action, and thus descends to $\fw_5:[\CC^5/\ZZ_5]\to \CC$. The FJRW invariants of $([\CC^5/\ZZ_5],\fw_5)$
were constructed in \cite{PV,Chi,Mo,FJR1,FJR2}. We follow the construction in \cite{CLL}.

Let $\Mbar_{g,\gamma}^{1/5}$ be the stack of families of $(\sC,\Si^\sC, \sL:\sL^{\otimes 5}\cong \omega_\sC^{\log})$, of $\ell$-marked genus-$g$ twisted nodal curves together with a line bundle $\sL$ satisfying $\sL^{\otimes 5}\cong \omega_\sC^{\log}$ and having monodromy $\gamma_i$ along the marking $\Si_i^\sC$. Let $\Mbar_{g,\gamma}^{1/5,5p}$ be the moduli of $5$-spin twisted curves with five $p$-fields: i.e. its closed points are $(\sC,\Si^\sC,\sL,\varphi_1,\cdots,\varphi_5)$ with $(\sC,\Si^\sC,\sL)\in \Mbar_{g,\gamma}^{1/5}$ and $\varphi_i\in H^0(\sL)$. The setup endows
$\Mbar_{g,\gamma}^{1/5,5p}$ a perfect obstruction theory relative to the (smooth and proper) $\Mbar_{g,\gamma}^{1/5}$, along with a cosection. Then the cosection localization induces a cycle $[\Mbar_{g,\gamma}^{1/5,5p}]\virtloc\in A\lsta (\Mbar_{g,\gamma}^{1/5})$.




We fix  conventions that  we will use throughout this section. We write
\beq\label{mi}
\gamma=(\zeta_5^{m_1},\cdots,\zeta_5^{m_\ell})\in (\bmu_5^\times)^\ell,\quad m_i\in [1,4].
\eeq
We will also use the standard abbreviation $\gamma=(\zeta_5,\zeta_5,\zeta_5^2,\zeta_5^3)=(1123)=(1^223)$.
We assume $2g+\ell>3$, unless otherwise stated.
The following two vanishings are known to experts (c.f. \cite[Def.\,2.7, Lem.\,2.8, (3.3)]{CLL}).

\begin{prop}\label{Van-1}
When $2g-2-\sum_{i=1}^\ell(m_i-1)\not\equiv 0\,(5)$,  $\Mbar_{g,\gamma}(G_5)^p=\emptyset$.
When $\Mbar_{g,\gamma}(G_5)^p$ is not empty, its virtual dimension is
$\sum_{i=1}^\ell (2-m_i)$.
\end{prop}

\begin{prop}\label{possibility} The cycle $[\Mbar_{g,\gamma}(G_5)^p]\virtloc =0$ unless $(g,\gamma)$ is one of the following cases:
\begin{enumerate}
\item $g\ge 0$ and $\gamma_i\in \{\zeta_5,\zeta_5^2\}$ for all $i$;
\item $g=0$ and $\gamma=(1^{1+k}23)$ or $(1^{2+k}4)$, $k\ge 0$.
\end{enumerate}
\end{prop}

\begin{rema}\label{123114}
In case $g=0$ and $\gamma=(123)$ or $(1^24)$, one easily checks that
$\Mbar_{g,\gamma}^{1/5}\cong B\bmu_5$. Thus $\deg[\Mbar_{g,\gamma}^{1/5,5p}]\virt=\frac{1}{5}$.
  \end{rema}

\subsection{Dual-twisted FJRW invariants}
Let's recall FJRW invariants 
with descendants. Given
$\gamma\in (\bmu_5^\times)^\ell$, and $2g-2+\ell> 0$,   define
\beq\label{twist}
\langle\tau_{a_1}(\gamma_1)\cdots\tau_{a_\ell}(\gamma_\ell)\rangle_{g}^{G_5}:=
\int_{[\Mbar_{g,\gamma}^{1/5,5p}]\virtloc} \psi_1^{a_1}\cdots\psi_\ell^{a_\ell},
\eeq
where $\psi_i$ is the first Chern class of the relative cotangent bundle of the universal curve
of $\Mbar_{g,\gamma}^{1/5,5p}$($=\Mbar_{g,\gamma}(G_5)^p$) along its $i$-th marking. Later, we will use $\bar{\psi}$ to denote
the $\psi$-class given by cotangent line bundle of the coarse marked point.
In our case, as all $\gamma_i\in\bmu_5^\times$,  $\bar \psi_i=5 \psi_i$.



Given $k\ge 0$ with $2g-2\equiv k\,(5)$ and $\gamma=(2^k)$, then $[\Mbar_{g,(2^k)}^{1/5,5p}]\virtloc$
is zero dimensional. 
We define the primitive FJRW invariant to be
\beq\label{prim}
\varTheta_{g,k}=\int_{[\Mbar_{g,(2^k)}^{1/5,5p}]\virtloc} 1\in\QQ.
\eeq
Note that all
$\varTheta_{0,k\ge 3}$ were calculated in \cite[Thm.\,1.1.1]{CR}.

Let $(\Si^\cC,\cC,\cL)$ with the universal family  $\pi:\cC\to \Mbar_{g, \gamma}^{1/5}$  of  $\Mbar_{g, \gamma}^{1/5}$.
Set 
\beq\label{ee}\ee_T(\cL\dual):=e_{\Gm}(R\pi_{\ast}\cL\dual\otimes \bL_{-1})
\in H_{\Gm}\sta\bl \Mbar_{g, \gamma}^{1/5}\br .
\eeq
We define the dual-twisted FJRW invariants to be
\beq\label{twist11}
\langle\tau_{a_1}(\gamma_1)\cdots\tau_{a_\ell}(\gamma_\ell)\rangle_{g}^{\mathrm{d.t.},G_5}:=
\int_{[\Mbar_{g,\gamma}^{1/5,5p}]\virtloc} \ee_T(\cL\dual)\upmo  \cdot \psi_1^{a_1}\cdots\psi_\ell^{a_\ell}.
\eeq
By Proposition \ref{prop-loc}, Theorem \ref{final} and \eqref{eqn:Av}, they are the relevant contributions associated to $v\in V_0^S(\Gamma)$ in the localization formula \eqref{eqn:Av}. By Proposition \ref{possibility}, when $g\geq 1$, \eqref{twist} and \eqref{twist11} vanish unless
all $\gamma_i\in\{\zeta_5,\zeta_5^2\}$.
\medskip

We now use Grothendieck-Riemann-Roch formula to express \eqref{twist11} in terms of $\varTheta_{g,k}$.
We first calculate the Chern character of $R\pi_{\ast}(\cL\dual)$ over $\Mbar_{g,\gamma}^{1/5}$.
Following  the notations introduced in \cite[Theo 1.1.1]{Chi},   let $\Upsilon$
be the stack of objects $(\Si^\sC, \sC, \sL,z, \sU_z)$, where $(\Si^\sC, \sC, \sL)\in \Mbar_{g,\gamma}^{1/5}$,
$z$ is a node of $\sC$,  and $\sU_z$ is a branch of the formal completion $\hat\sC_z$ of $\sC$ at $z$.
Clearly, $\Upsilon$ is a smooth DM-stack.

Let $\psi_z\in A^1 \Upsilon$ be the first Chern class
(the $\psi$ class) of the relative tangent bundle (i.e. $T_z\sU_z$) of the
distinguished branch along the distinguished node of the universal family of $\Upsilon$. Let
$\bar{\psi}_z$ be the coarse moduli analogue of $\psi_z$. 
Thus $\bar{\psi}_z$ equals $5\psi_z$ when the node is non-scheme, and equals $\psi_z$ otherwise.

Let $\jmath: \Upsilon\to \Mbar_{g,\gamma}^{1/5}$
be the forgetful morphism. 
The image $\jmath(\Upsilon)\sub \Mbar_{g,\gamma}^{1/5}$ is
the divisor of spin curves with singular domains, and
$\jmath: \Upsilon\to \jmath(\Upsilon)$ is finite.
Let $\Upsilon_l\sub \Upsilon$ be the open and close substack of objects $(\Si^\sC, \sC, \sL,z, \sU_z)$
so that the monodromy of $\sL|_{\sU_z}$ along $z$ is $\zeta_5^l$, and let
$\jmath_l=\jmath|_{\Upsilon_l}:\Upsilon_l\to \Mbar_{g,\gamma}^{1/5}$.
Let $\sigma:\Upsilon\to\Upsilon$ be the involution that sends $(\Si^\sC, \sC, \sL,z, \sU_z)$
to $(\Si^\sC, \sC, \sL,z, \sU_z')$, where $\sU_z$ and $\sU_z'$ are the two branches of the distinguished node of $\sC$.

For $(\sC,\Si^\sC,\sL)\in \Mbar_{g,\gamma}^{1/5}$,  let $\bar\sC$ be the twisted curve after taking the coarse moduli of $\sC$
at all its markings, and let $\Si^{\bar\sC}_i\sub\bar\sC$ be the associated $i$-th (scheme) marking.
We abbreviate $\Mbar\defeq \Mbar_{g,\gamma}^{1/5}$. Then the collection of such
curves ($\Si_i^{\bar\sC}, \bar\sC$) form a family of curves $\bar\pi:\bar\cC\to  \Mbar$, with $\bar\Si_i\sub \bar\cC$ being
the associated sections of markings. Let $B_m(x)$ be the Bernoulli polynomials.
Let
\beq\label{compare} \bar\psi_i=c_1(T_{\bar\cC/\Mbar}|_{\bar\Si_i}),\quad
\kappa_h=\pi\lsta(c_1(\omega^{\log}_{\cC/ \Mbar})^{h+1})=\bar\pi\lsta(c_1(\omega^{\log}_{\bar\cC/ \Mbar})^{h+1}).\eeq
As in \eqref{mi},  {$\gamma_i=\zeta_5^{m_i}$ where  $m_i$ are integers in $[1,4]$.}

\begin{lemm}  
The $h$-th total Chern character 
\begin{align}\label{GRR}
		& ch_h(R\pi_{\ast} \cL\dual)=  \frac{B_{h+1}(-1/5)}{(h+1)!}\kappa_h-\sum_{i=1}^\ell\frac{B_{h+1}((5-m_i)/5)}{(h+1)!}(\bar\psi)_i^h+
		\qquad\qquad\\
		&\quad\qquad\qquad\qquad+\frac{1}{2}\sum_{k=0}^{4}\frac{5B_{h+1}(k/5)}{(h+1)!}5^{-\delta_{k,0}}(\jmath_k)\lsta
		\Bigl(\sum_{i+j=h-1}(-\bar\psi)^i\sigma\sta(\bar\psi)^j\Bigr). \nonumber
\end{align}
\end{lemm}

\begin{proof}
We prove this lemma by applying the formula in \cite[Theo 1.1.1]{Chi}.
We first recall its setup. Following \cite[Definition 2.1.1]{Chi}, a $5$-stable $\ell$-pointed curve $(\sA,x_1,\cdots,x_\ell)$  is
an $\ell$-pointed twisted nodal curve 
with all of its markings being scheme points, all of its nodes being $\bmu_5$-balanced,   and its automorphisms group
being finite. Following \cite[Theorem 2.2.1]{Chi}, we form the groupoid $\overline\cN$ of $(\sA,x_i,\sS)$, where $(\sA,x_i)$ is a
$5$-stable $\ell$-pointed curve, and $\sS$ is an  invertible sheave on $\sA$ so that
$\sS^{\otimes 5}\cong (\omega^{\log}_{\cA})^{-1}(-\sum_{i=1}^\ell(5-m_i) x_i))$.
The stack $\overline \cN$ is a smooth proper DM stack. Let $\hat\pi:\cA\to\overline\cN$, $\cX_1,\cdots,\cX_\ell\sub\cA$ and
$\cS$ be the universal family of $\overline\cN$, which comes with the isomorphism
$$ \cS^{\otimes 5}\mapright{\cong}   (\omega^{\log}_{\cA/\overline\cN})^{-1}\bl-\sum_{i=1}^\ell(5-m_i) \cX_i\br.
$$

Let $\eta:\cC\to\bar\cC$ be the tautological partial coarse moduli morphism, and  let $\bar\cS=\eta\lsta(\cL\dual)$ which is a line bundle on $\bar \cC$.
Using $\cL^{\otimes 5}\cong \omega^{\log}_{\cC/\Mbar}$, we obtain
\begin{eqnarray*} &\bar\cS^{\otimes r}\ {\cong}\ (\omega^{\log}_{\bar\cC/\Mbar})^{-1}\bl -\sum_{i=1}^\ell(5-m_i)\bar\Si_i\br.
\end{eqnarray*}
Thus the family $(\bar\cC,\bar\Si_i,\bar\cS)$ would induce a morphism $\overline\cN\to\Mbar$  if  {$\bar \cS$} were representable along
all nodes of the base curves.

To remedy this, to each $(\sA,x_i,\sS)\in \overline\cN$,  let $\hat\sA$  be the twisted curve obtained from $\sA$ by forgetting the stacky structure
 at every node of $\sA$ where $\sS$ has trivial monodromy,  and  let $\zeta: \sA\to \hat\sA$ be the obvious morphism.
Then $(\hat\sA,x_i,\zeta\lsta\sS)\cong (\bar\sC,\Si_i^{\bar\sC},\bar\sL)$ for a unique point $(\sC,\Si_i^{\sC},\sL)\in\Mbar$.
This association produces  a morphism $f: \overline\cN\to\Mbar$, isomorphic on a dense open subset, so that
\beq\label{ff} R\hat\pi_{\ast} \cS=f\sta(R\bar\pi_{\ast} \bar\cS)= f\sta(R\pi_{\ast} \cL\dual).
\eeq

We   recall the notations in \cite[page 3]{Chi}. As  the case for $\Upsilon\to \Mbar$,     let
$\Upsilon'=\cup_{k=0}^4\Upsilon'_k$ be
the smooth DM stack of objects $(\sA,x_i, \sS,z, \sU_z)$, where $(\sA,x_i,\sS)\in \overline\cN$, and
$(z,\sU_z)$ is the pair of a node and a branch of $\sA$. Here $\Upsilon'_k\sub \Upsilon'$ is the
substack characterized by  the monodromy of $\sS|_{\sU_z}$ along $z$  being  $\zeta_5^k$.  Let $j:\Upsilon'\to \overline\cN$ be the forgetful morphism.
Also let $j_k=j|_{\Upsilon'_k}$, and let $\ti\sigma:\Upsilon'\to \Upsilon'$ be the evolution.

By \cite[Theo 1.1.1]{Chi}, the GRR formula takes the form
\begin{eqnarray*}
&ch_h(R\hat\pi_{\ast} \cS)=  \frac{B_{h+1}(-1/5)}{(h+1)!}\ti\kappa_h-\sum_{i=1}^\ell\frac{B_{h+1}((5-m_i)/5)}{(h+1)!}(\ti\psi_i)^h+\qquad\\
&\qquad\qquad\qquad+\frac{1}{2}\sum_{k=0}^{4}\frac{5B_{h+1}(k/5)}{(h+1)!}(j_k)\lsta
		\Bigl(\sum_{i+j=h-1}(-\ti\psi_z)^i\ti\sigma\sta(\ti\psi_z)^j\Bigr),
\end{eqnarray*} 	
where $\ti\psi_i$ is the $i$-th $\psi$-class of $\overline\cN$, and $\ti \psi_z$ is the $\psi$-class of the distinguished node-branch
of $\Upsilon'$ (cf. \cite[page 3 line 20]{Chi}). By our construction we have $\ti\psi_i=f\sta \bar \psi_i$,
$\ti \psi_z=\lam_k\sta \bar{\psi}_z$, where $\lam_k:\Upsilon'_k\to \Upsilon_k$ is defined similar to that of $f: \overline \cN\to\Mbar$.
 $\ti\kappa_k$ is the standard $\kappa$-class of $\overline\cN$.

Applying $f\lsta$ to the above formula, one checks that the first two terms on the right hand side of the identity
coincide with that in the statement of the lemma.
For the third term  we use the commutative square
$$
\begin{CD}
\Upsilon'_k @>j_k>> \overline\cN\\
@VV{\lam_k}V@VV{f}V\\
\Upsilon_k@>\jmath_k>>  \Mbar=\Mbar_{g,\gamma}^{1/5}.
\end{CD}
$$
For $k\ne 0$, $\lam_k$ is birational, one has
$$f\lsta j_{k\ast}\Bigl((-\ti\psi)^i\sigma\sta(\ti\psi)^j\Bigr)=\jmath_{k\ast}\lam_{k\ast}\Bigl((-\ti\psi)^i\sigma\sta(\ti\psi)^j\Bigr)=\jmath_{k\ast}\Bigl((-\bar\psi)^i\sigma\sta(\bar\psi)^j\Bigr).$$

In case $k=0$, $\lam_0$ is generically a $\bmu_5$-gerbe (c.f. \cite[Def 2.1.1]{Chi}, \cite[Thm. 7.1.1]{ACV}).
Therefore  
$$f\lsta j_{0\ast}\Bigl((-\ti\psi)^i\sigma\sta(\ti\psi)^j\Bigr)=\jmath_{0\ast}\lam_{0\ast}\lam_0\sta\Bigl((-\bar\psi)^i\sigma\sta(\bar\psi)^j\Bigr)=\frac{1}{5}\jmath_{0\ast}\Bigl((-\bar\psi)^i\sigma\sta(\bar\psi)^j\Bigr).$$	
This gives the term $5\upmo=5^{-\delta_{0,0}}$. This completes the proof.
\end{proof}

  Using the  GRR formula \eqref{GRR} we can reduce dual-twisted FJRW invariants  to primitive  FJRW invariants. We first need a simple fact for genus one case.

\begin{lemm}\label{g1vir} The moduli $\bcM_{1,\zeta_5}^{1/5}$ is a disjoint
union of two  {one dimensional smooth} DM-stacks $M_0$ and $M_1$, characterized by that
$(\Si^\sC,\sC,\sL)\in M_0$ (resp. $\in M_1$) if $\sL(-\Si^\sC)\cong \sO_\sC$ (resp. $\sL(-\Si^\sC)$ has no section).
Furthermore,
\beq\label{g1virt}[\bcM_{1,\zeta_5}^{1/5,5p}]\virtloc=-4^5[M_0]+[M_1].\eeq
\end{lemm}

\begin{proof}By the assumption on the monodromy,
$\sL(-\Si^\sC)$ has trivial monodromy along $\Si^\sC$,
and $\sL(-\Si^\sC)^{\otimes 5}\cong \sO_\sC$. Thus $\barM_{1,\zeta_5}^{1/5}$ is isomorphic to the moduli
of $(\sC,\Si,\sL')$ so that $(\sC,\Si)$ is a one-pointed genus 1 twisted curve and $\sL^{\prime\otimes 5}\cong \sO_C$. This gives
$\bcM_{1,\zeta_5}^{1/5}=M_0\cup M_1$ (cf. \cite{AJ}).

For \eqref{g1virt},
we first note that since $H^0({\sL'})=0$ for $(\Si^\sC,\sC,{\sL'})\in M_1$, 
$$[\Mbar_{1,\zeta_5}^{\frac{1}{5},5p}]\virtloc=c[M_0]+[M_1],\quad c\in\QQ.
$$
Since for $(\Si^\sC,\sC,{\sL'})\in M_0$, $h^i({\sL'})=1$ for {$i=0$, $1$},
\cite[Prop 4.17]{CLL} gives $c=-4^5$.
\end{proof}

\begin{lemm}\label{un-tw} Given $a_1,\cdots,a_\ell\in\ZZ_{\geq 0}$ and $\gamma_1,\cdots,\gamma_\ell
\in\{\zeta_5, \zeta_5^2\}$, letting $k=\#\{\gamma_i\,|\,\gamma_i=\zeta_5^2 \}$,
there is a $\fc \in\QQ[\ft,\ft\upmo]$, depending on $\{g, \ell, k, a_1, \ldots, a_k\}$ and effectively calculable, such that
\beq\label{twistedFJRW} \langle\tau_{a_1}(\gamma_1)\cdots\tau_{a_\ell}(\gamma_\ell)\rangle_{g}^{\mathrm{d.t.},G_5}
=\varTheta_{g,k}\cdot \fc.
\eeq
\end{lemm}

\begin{proof}
Assume $(g,k)\ne (1,0)$. From \cite[Theo 4.5]{CLL}, we have $\hat\pi: \bcM_{g,\gamma}^{1/5} \lra \Mbar_{g,(2^k)} ^{1/5}$ which forgets all $\zeta_5$-markings.
By \cite[Theo 4.10]{CLL}, we have $\hat\pi\sta [\Mbar_{g,(2^k)}^{1/5,5p}]\virtloc
=[\Mbar_{g,\gamma}^{1/5,5p}]\virtloc$.
 Since $\bcM_{g,\gamma}^{1/5}$ is smooth, we have the diagram 
$$
\begin{CD}
A\lsta  (\Mbar_{g,(2^k)} ^{1/5})@>\hat\pi\sta>> A\lsta (\Mbar_{g,\gamma}^{1/5} )\\
@VV{\ti\iota\lsta}V@VV{\iota\lsta}V\\
H\lsta  (\Mbar_{g,(2^k)} ^{1/5})@>\hat\pi\sta>> H\lsta (\Mbar_{g,\gamma}^{1/5} ).\\
\end{CD}
$$
Due to the  dimension reason, $\ti\iota\lsta [\Mbar_{g,(2^k)}^{1/5,5p}]\virtloc=\sum_j a_j[\xi_j]$,
for $a_j\in \QQ$ and $\xi_j\in \Mbar_{g,(2^k)}^{1/5}$ general such that $\sum a_j=\varTheta_{g,k}$.
Consequently,
$$\iota\lsta [\Mbar_{g,\gamma}^{1/5,5p}]\virtloc={\hat\pi\sta} \ti\iota\lsta [\Mbar_{g,(2^k)}^{1/5,5p}]\virtloc
=\sum a_j[\hat\pi^{-1}(\xi_j)].
$$
Hence using the definition \eqref{twist11}, we obtain
$$
\langle\tau_{a_1}(\gamma_1)\cdots\tau_{a_\ell}(\gamma_\ell)\rangle_{g}^{\mathrm{d.t.},G_5}=
\sum a_j\int_{[\hat\pi^{-1}(\xi_j)]}  \ee(\cL\dual)\upmo  \cdot \psi_1^{a_1}\cdots\psi_\ell^{a_\ell} .
$$
 As $ \ee_T(\cL\dual)$ is a polynomial of   $ch_h(R\pi_{\ast} \cL\dual)$'s, applying \eqref{GRR} we know that
$${\fc}\defeq\int_{[\hat\pi^{-1}(\xi)]}  \ee_T(\cL\dual)  \cdot \psi_1^{a_1}\cdots\psi_\ell^{a_\ell}\in \QQ[\ft,\ft\upmo]
$$
is independent of the choice of the general $\xi\in \Mbar_{g,(2^k)}^{1/5}$, and is effectively calculable.
Finally applying $\varTheta_{g,k}=\sum a_j$, we obtain \eqref{twistedFJRW}.
\vsp

It remains to look at the case when $(g,k)=(1,0)$. By stability requirement, we have $\ell\ge 1$.
Let $\hat\pi$ be the forgetful morphism forgetting all but the first marking.
Using   \eqref{GRR} and that the Hodge class $\kappa$ is representable by $\psi$ classes,
the correlator (\ref{twistedFJRW})  is a sum of multiples of
$\int_{[\Mbar_{1,\zeta_5}^{1/5,5p}]\virtloc}\psi$, and the boundary classes of $\Mbar_{1,\zeta_5}^{1/5}$,
with the multiplicities  explicitly calculable.

The boundary classes can be easily calculated.
To calculate $\int_{[\Mbar_{1,\zeta_5}^{1/5,5p}]\virtloc}\psi$, we apply Lemma \ref{g1vir}.
Because $M_0\to \Mbar_{1,1}$ is quasi-finite and generically a $\bmu_5$-gerbe,
$\int_{[M_0]}\psi=\frac{1}{5\cdot5\cdot 24}$, where the {additional} $5$ comes from that $\psi$ is of
$\bmu_5$-markings.

To calculate the contribution from $[M_1]$, we need to determine the degree of $M_1\to \Mbar_{1,1}$.
Since $\Mbar_{1,\zeta_5}^{1/5}\to \Mbar_{1,1}$ is flat,  a $\bmu_5$-gerbe generically, and has $5^2$
pre-images over a general point of $\Mbar_{1,1}$, the degree we intend to calculate is $5^2\cdot \ofth-\ofth$.
Thus $\int_{[M_1]}\psi=\frac{25-1}{24\cdot 5\cdot5}=\frac{1}{25}$.
This proves the proposition.
\end{proof}

\begin{lemm}\label{exc}
For the case $g=0$, and $\gamma=(1^{k+1}23)$ or $(1^{k+2}4)$ in Proposition \ref{possibility},  
$ \langle\tau_{a_1}(\gamma_1)\cdots\tau_{a_\ell}(\gamma_\ell)\rangle_{g}^{\mathrm{d.t.},G_5}$ is calculable.
\end{lemm}
\begin{proof} First, for $\gamma_0=(123)$ or $(1^24)$, the moduli $\barM_{0,\gamma_0}^{1/5}\cong B\bmu_5$.
For $k>0$, we use the  marking-forgetful  morphism $\barM_{0,\gamma}^{1/5}\to \barM_{0,\gamma_0}^{1/5}\cong B\bmu_5$.
Combined with the argument in the previous proof we obtain the lemma.
\end{proof}

 \begin{rema}
  We expect similar quantization formulae for dual-twisted FJRW potentials as that in  \cite{CG,CZ}. \end{rema}
  \medskip

 Parallelly, we need to treat the following term in the localization formula \eqref{eqn:Av}.  
\beq\label{mono} \int_{[\cW_v]\virtloc} \frac{1}
{e_G(R\pi_{v\ast} ev\sta_v\sO_\Pf(1)\otimes \bL_1)}  \prod_{i=1}^{n_v}\psi_{i}^{a_i}ev_i\sta(h^{k_i}),\quad k_i\in \ZZ_{\geq 0},
\eeq
where $h\in H^2(\Pf,\ZZ)$ is the positive generator, $v \in V_0^S(\ga)$,
$(\pi_v,ev_v): \cC_v\to \cW_v\times\Pf$ is the universal family of $\cW_v$, and $ev_i: \cW_v\to \Pf$ is the evaluation
at $i$-th markings.

\begin{lemm}\label{twGW}   Suppose $(g_v,d_v)\neq (0,0),(1,0)$.  Let $\langle\cdots\rangle^{GW}_{g,d}$ be  the genus-$g$ degree-$d$ GW invariant of $Q_5$. The integral \eqref{mono} is a monomial in $\ft$, of the form
\beq\label{GW-} (-1)^{d_v+1-g_v}
\cdot \langle\tau_{a_1}(h^{k_1})\cdots \tau_{a_{n_v}}(h^{k_{n_v}})\rangle^{GW}_{g_v,d_v}\cdot \ft^{d_v+1-g_v}.
\eeq
\end{lemm}


\begin{proof}   After fixing an ordering of $E_v$, we have isomorphism
$\cW_v\cong \Mbar_{g_v,n_v}(\Pf,d)^p$, and $[\cW_v]\virtloc=[\Mbar_{g_v,n_v}(\Pf,d)^p]\virtloc$.
By \cite{CL},
\beq\label{CL0}
[\Mbar_{g_v}(\Pf,d)^p]\virtloc=(-1)^{d_v+1-g_v} [\Mbar_{g_v}(Q_5,d_v)]\virt\in A_0 \Mbar_{g_v}(Q_5,d_v).
\eeq
Let $\theta:\Mbar_{g_v,n_v}(Q_5,d_v)\to \Mbar_{g_v}(Q_5,d_v)$ be the forgetful map. Then we have
$$\theta\sta [\Mbar_{g_v}(\Pf,d_v)^p]\virtloc
= [\Mbar_{g_v,n_v}(\Pf,d_v)^p]\virtloc.
$$
{
By zero dimensionality, we can assume
$[\Mbar_{g_v}(\Pf,d_v)^p]\virtloc=\sum c_i [\xi_i]$,
where $\xi_i$, $i=1,\cdots m$ are closed points in $\Mbar_{g_v}(\Pf,d_v)^p$, and $c_i\in\QQ$.

Let $W_i=\theta\upmo(\xi_i)\sub \Mbar_{g_v,n_v}(\Pf,d_v)^p$. Because $W_i$ consists of a fixed stable
map $\phi_i:C_i\to\Pf$ together with all possible marking $p_1,\cdots, p_{n_v}\sub C$, we conclude
$R\pi_{v\ast} ev\sta_v\sO_\Pf(1)|_{W_i}\cong \sO_{W_i}^{\oplus e}$, where
$e=d_v+1-g_v$. Thus
$$e_G(R\pi_{v\ast} ev\sta_v\sO_\Pf(1)\otimes \bL_1)\upmo \cap [\cW_v]\virtloc= \ft^{-e} \sum c_i[W_i]=
 \ft^{-e} \cdot[\Mbar_{g_v,n_v}(\Pf,d_v)^p]\virtloc.
$$ }
\end{proof}

\begin{rema}\label{well}  Using string, dilaton and divisor equations, \eqref{GW-}
can be calculated knowing the
GW invariants of quintics  $N_{g,d}$. 
\end{rema}

\subsection{ From FJRW invariants   to GW-invariants} 

We use MSP moduli spaces to produce an algorithm calculating the GW invariants of the  quintic from the
FJRW invariants.


\begin{proof}[Proof of Theorem \ref{main-1}]
We first prove the case $g\ge 2$.
Choose $\bd=(d,0)$ and $\ell=0$, and consider the moduli of stable MSP fields $\cW_{g,\bd}$
of numerical data $(g,\emptyset,\bd)$. 
We apply the vanishing \eqref{0} to the equivariant
cycle $[\cW_{g,\bd}]\virtloc$ to obtain
\beq\label{AA}
0=\sum_{\Ga\in\Del^\reg_{g,\bd}}\contri(\Ga),
\quad \contri(\Ga)=\Bigl[\ft^{\delta(g,\bd)} \cdot\frac{[(\cW\lgd)^{T}_{(\Gamma)}]\virtloc}
{e(N_{(\cW\lgd)^{T}_{(\Gamma)}/\cW\lgd})}\Bigr]_0.
\eeq
For $\Ga\in\Del^\reg_{g,\bd}$, applying
Proposition \ref{prop-loc} we see that $\contri(\Ga)$ is a polynomial expression in terms of the
twisted GW invariants of $Q_5$ and the dual-twisted
FJRW invariants of $([\CC^5/\ZZ_5],\fw_5)$. By the previous subsections, they can all be expressed in
terms of $N_{g',d'}$ and $\varTheta_{g',k'}$.

We now investigate possible $N_{g',d'}$ and $\varTheta_{g',k'}$ that appear in $\contri(\Ga)$.
By Proposition \ref{prop-loc} and the previous discussion,
we see that each $v\in V_0^S(\Ga)$ contributes an entry $N_{g_v,d_v}$
in $\contri(\Ga)$, and each non-exceptional $v\in V_\infty^S(\Ga)$ contributes an entry $\varTheta_{g_v,k_v}$ in $\contri(\Ga)$,
where $k_v$ is the number of  edges $e\in E_v$ so that the monodromy of $\sL|_{\sC_v}$ along $y_{(e,v)}=\sC_e\cap \sC_v$
is $\zeta_5^2$.

We first show that if $N_{g',d'}$ appears in $\contri(\Ga)$, then $g'\le g$ and $d'\le d$, and if $N_{g,d}$ appears in
$\contri(\Ga)$, then $\Ga$ is a one vertex graph. Indeed, because the MSP fields all have domain curves genus $g$, and
$g=\sum_{v\in V^S(\Ga)}g_v+h^1(\Ga)$, we have $g_v\le g$ for all $v\in V^S(\Ga)$. Furthermore  when $g_v=g$ for one
$v\in V_0$, then all other
$g_{v'}=0$ and $h^1(\Ga)=0$.
For  the degree bound, we recall that (cf. \cite{CLLL}) by the non-vanishing properties of $\nu_1$ and $\nu_2$, both
$\sN\otimes\sL|_{\sC_1\cup\sC_{1\infty}\cup\sC_\infty}$ 
and $\sN|_{\sC_0\cup\sC_{01}\cup\sC_1}$ are trivial line bundles.
Thus using $\deg\sN=d_\infty=0$, 
we get $\deg\sL|_{\sC_0\cup\sC_{01}}=d_0=d$. Because $\deg\sL|_{\sC_e}>0$ for every
$e\in E_{0}(\Ga)$, we have $d_v=\deg\sL|_{\sC_v} \le d$. Finally,
suppose $N_{g',d}$ appears in $\contri(\Ga)$, then $E_{0}(\Ga)=\emptyset$ and $V_0^S(\Ga)\ne\emptyset$.
Since $\Ga$ is connected and regular, $\Ga$ must be a single vertex graph, say $\{v\}=V_0^S(\Ga)$,  and then $g_v=g$.
For such graph, we have $\contri(\Ga)=N_{g,d}$. This proves (1).

We now look at the possible $\varTheta_{g',k'}$ appearing in $\contri(\Ga)$.
For $v\in V_\infty(\Ga)$, we adopt the convention that
$$\sC_{[v]}=\sC_v\cup\bl \cup_{e\in E_v} \sC_e\br,
$$
which is a union of $\sC_v$  with all $\sC_e$'s intersecting $\sC_v$.
As before, we let $V_{\mathrm{exc}}(\Ga)$ be the set of exceptional vertices (c.f. before Definition \ref{regu}), namely $v\in V^S_\infty(\Gamma)$ with $g_v=0$ and $\gamma_v=(1^{2+k}4)$ or $(1^{1+k}23)$. 
For each $v\in V^S_{\mathrm{exc}}(\Ga),$
\beq\label{v1}\deg \sN|_{\sC_{[v]} }=-\bl d_v+\sum_{e\in E_v}d_e\br \geq
-\Bigl(\frac{k+1}{5}-\frac{6+k}{5}\Bigr)=1.
\eeq
For each $v$ in 
$V^S_\infty(\Gamma)-V_{\mathrm{exc}}(\Ga)$, since $\Ga$ is regular,    for every $e\in E_v,$  $\gamma_{(e,v)}=\zeta_5$ or $\zeta_5^2$.
If $\gamma_{(e,v)}=\zeta_5$,  then $d_e\le -1/5$, and  if $\gamma_{(e,v)}=\zeta^2_5$,  then  $d_e\le -2/5$.
Let $s_v$ (resp. $k_v$) be the number of $e\in E_v$ whose $\gamma_{(e,v)}=\zeta_5$ (resp. $\zeta_5^2$).
Adding that both $\nu_1|_{\sC_v}$ and $\rho|_{\sC_v}$ are nowhere vanishing, we have
$$\deg \sN^{\vee\otimes 5}|_{\sC_v} 
=\deg \omega_{\sC_v}(\sum_{e\in E_v} y_{(e,v)})=
(2g_v-2+|E_v|) .
$$
Therefore, $5\deg\sN|_{\sC_v}=-(2g_v-2)-|E_v|$.

 As $\sN|_{\sC_1\cup\sC_{01}\cup\sC_0}$ is trivial and
$\deg \sN|_{\sC_{[v]}}>0$ 
for unstable $v\in V^{}_\infty$, \eqref{v1} implies  
 \begin{align*} 0&=d_\infty  =\sum_{v\in V_\infty(\Ga)}\deg \sN|_{\sC_{[v]}}  
\ge \sum_{v\in V^S_\infty(\Gamma)-V_{\mathrm{exc}}(\ga)}  \deg\sN|_{\sC_{[v]}}\\
&\geq \sum_{v\in V^S_\infty(\Gamma)-V_{\mathrm{exc}}(\ga)}\frac{k_v-(2g_v-2)}{5}.
\qquad\qquad
\end{align*}
 
 Let $v\in V_{\infty}^S(\Ga)$. First we have $g_v\leq g$.  
If  $g_{v}=g$, then all other 
$g_{v'}=0$. Thus the above inequality shows
that $k_v\le 2g_v-2$. This proves the first part of (2).

We prove the second part of (2). If $g_v\le g-1$ for all $v\in V_{\infty}^S(\Ga)$, then
$\sum_{v\in V^S_\infty(\Gamma)-V_{\mathrm{exc}}(\ga)} (2g_v-2) \leq 2g-4$. Combined with the previous inequality, we obtain
$k_v \leq 2g-4$ for every $v\in v\in V^S_\infty(\Gamma)-V_{\mathrm{exc}}(\ga)$. This proves (2).

Finally, the contribution from $v\in V_1^S(\Ga)$ are integrals of $\psi$ classes on $\Mbar_{g',n'}$,
which can be effectively calculated.

Combined, this shows that \eqref{AA} is a polynomial expression of terms as stated in (1) and (2) with coefficients involving
Hodge integrals on $\Mbar_{g',n'}$ and other calculable terms. 
This proves the case $g\ge 2$.

In case $g=1$, a similar argument using Proposition \ref{un-tw} shows that
we can determine all $N_{1,d}$, knowing the terms specified in the theorem. 
\end{proof}

\begin{rema} The proof shows that \eqref{0} gives an algorithm to determine all genus $g$ invariants $\{N_{g,d}\}_{d=0}^\infty$ from a finite set of initial data. For $g=1$ no initial data is needed.
For $g=2$, the initial data are $N_{2,1}$ and $\varTheta_{2,2}$. As $N_{2,1}$ is classical,
only $\varTheta_{2,2}$ is unknown. Similarly, for $g=3$ we only need $\varTheta_{2,2}$ and $\varTheta_{3,4}$
since $N_{3,1}$ and $N_{3,2}$ are classical.
(Recall that $\varTheta_{g,k}=0$ unless $k+2-2g\equiv 0(5)$.)
\end{rema}

\subsection{Algorithm for FJRW invariants}
We let $\bd=(0,d)$ so that $d+1-g>0$; we choose $\gamma=\emptyset$, and look at the polynomial relation from the vanishing \eqref{0}. Since $d_0=0$, no GW invariants $N_{g',d'>0}$ appear in this relation. 
Thus it provides a relation among FJRW invariants $\varTheta_{g',k'}$.
\begin{proof}[Proof of Theorem \ref{main-2}] 
We need to show that for $g\geq 1$ and $k\ge 7g-2$, the relation \eqref{0} with $\gamma=\emptyset$, $d_0=0$ 
provides an effective algorithm to evaluate $\Theta_{g,k}$,
provided that $\{\varTheta_{g,k' }\}_{k'< k}$ and $\{\Theta_{g',k'}\}_{g'<g,k'\le k-2}$ are known.

We use the same notations as in the prior proof. As $\varTheta_{g,k}=0$ when $5 \nmid k-7g+2$,
we only need to look at $k=7g-2+5m$, $m\ge  0$.
Pick $d_\infty=g+m$ and let {$\bd=(0,d_\infty)$}.  Applying \eqref{0} to the
cycle $[\cW_{g,\bd}]\virtloc$, as in \eqref{AA},  we obtain
\begin{align}\label{con-3}
\sum_{\Ga\in\Del^\reg_{g,\bd}}\contri(\Ga)=0.
\end{align}
Here $\contri(\Ga)$ is a polynomial expression of a collection
of $N_{g',d'}$ and $\varTheta_{g',k'}$. Let $\xi=(\sC,\Sigma,\sL,\sN,\cdots)\in \cW_\ga$. Since
$\sN\otimes \sL|_{\sC_1\cup\sC_{1\infty}\cup\sC_\infty}$ and $\sN|_{\sC_0\cup\sC_{01}\cup\sC_1}$ are trivial line
bundles,  $\deg\sL|_{\sC_0\cup\sC_{01}}=d_0$. As the degree of $\sL$  is nonnegative on each component of $\sC_0$  and positive on each component of $\sC_{01}$, $d_0=0$ implies that $V_0(\Ga)=E_0(\Ga)=\emptyset$. Therefore no $N_{g',d'}$ occurs in $\contri(\Ga)$.

As argued before, only $\varTheta_{g'\le g,k'}$ can possibly appear in $\contri(\Ga)$.
We now examine when does $\varTheta_{g,k'}$ appear. Following the notation in the proof of Theorem \ref{main-1}, we get
\begin{align}\label{eq5.16}
g+m&=d_\infty  =\sum_{v\in V_\infty(\Ga)}\deg \sN|_{\sC_{[v]}} \geq \sum_{v\in V^S_\infty(\Gamma)-V_{{\mathrm{exc}}}(\Gamma)}\frac{k_v-(2g_v-2)}{5} +\nonumber \\
&\quad+\sum_{v\in V_{\mathrm{exc}}(\Gamma)}\deg\sN|_{\sC_{[v]}}
+\sum_{v\in V_{\infty}^U(\Gamma)}\deg\sN|_{\sC_{[v]}} 
 \geq \sum_{v\in V^S_\infty(\Gamma)-V_{{\mathrm{exc}}}(\Gamma)}\frac{k_v-(2g_v-2)}{5}.
\end{align}
Here we used that for $v\in V_{{\mathrm{exc}}}(\Gamma) \cup V_{\infty}^U(\Gamma)$,
$\deg\sN|_{\sC_{[v]}} \geq 1$ (c.f. \eqref{v1}). Thus, using $k=7g-2+5m$,
\begin{align}\label{combine-ine}\sum_v k_v\le 5g+5m+ \sum_v (2g_v-2)=  k
-\big[ 2g-2 - \sum_v (2g_v-2)\big],
\end{align}
where the sum is taken over $v\in V_\infty^S(\Ga) -V_{{\mathrm{exc}}}(\Gamma)$. 

In case $g_v\le g-1$ for all $v\in V_{\infty}^S(\Ga)$, then
$\sum_v (2g_v-2) \leq 2g-4$; thus \eqref{combine-ine} implies that for every $v\in V^S_\infty(\Gamma)$,
$k_v\leq k-2$; and $\Theta_{g_v,k_v}$ appears in $\{\Theta_{g',k'}\}_{g'<g,k'\le k-2}$.
Now suppose there is a $v\in V^S_\infty(\ga)$ so that $g_{v}=g$.
Then \eqref{combine-ine} implies $k_v\le  k$, thus either  $\Theta_{g_v,k_v}$
appears in $\{\varTheta_{g,k' }\}_{k'< k}$, or $k_v=k$.

We examine the remainder case $(g_v,k_v)=(g,k)$. Because $k=7g-2+5m$ and $d_\infty=g+m$, 
we conclude that $V_\infty(\Ga)=V_\infty^S(\ga)=\{v\}$, $V_{{\mathrm{exc}}}(\Gamma)=\emptyset=V_{\infty}^U(\Gamma)$, and $E_v=\{e_1,\cdots,e_{k}\}$ with $\deg\sN_{\sC_{e_i}}=-2/5$.
%
By Proposition \ref{prop-loc},
$[\cW_{(\Ga)}]\virtloc$ is a $\frac{1}{2^{k}\cdot k!}$ multiple of 
$[\Mbar_{g,(2^k)}^{1/5,5p}]\virtloc$.
We calculate $\contri(\Ga)$.

Recall $\ee_T(\cL\dual)=e_{\Gm}(R\pi_{\ast}\cL\dual\otimes \bL_{-1})$ (cf. \eqref{ee}); applying Theorem \ref{final} and
the formulae in the previous section, we get
\begin{eqnarray*}
  \contri(\Ga)&=&\Bigl[\ft^{\delta(g,\bd)} \cdot\frac{[(\cW\lgd)^{T}_{(\Gamma)}]\virtloc}
{e(N_{(\cW\lgd)^{T}_{(\Gamma)}/\cW\lgd})}\Bigr]_0\\
&=&\Bigl[ \ft^{1+m} (5\ft\cdot \frac{5\ft}{2})^{k}\cdot (\frac{1}{\frac{5\ft}{2}\cdot 5\ft})^{k} \cdot  
\displaystyle{ \frac{[\Mbar_{g,\gamma}^{1/5,5p}]\virtloc }{k!2^{k} \ee_T(\cL\dual)}}\frac{1}{\prod_{i=1}^{k} (-\frac{\ft}{2}-\psi_i)} \Bigr]_0\\
&=&\ft^{1+m}\frac{\varTheta_{g,k}}{k!2^{k}(-\ft)^{3-4m-7g}}\cdot\frac{1}{(-\frac{\ft}{2})^{k}}=(-1)^{m+1}(k!)^{-1}\varTheta_{g,k},
\end{eqnarray*}
where we have used $\vdim   [\Mbar_{g,\gamma}(G_5)^p]\virtloc=0$ and $\rank R\pi_{\ast}\cL\dual=3-4m-7g$.

This proves that the relation \eqref{con-3} expresses
$\Theta_{g,k}$ as a polynomial of terms indicated in the statement of the theorem.
\end{proof}

As a corollary, the collection $\{\varTheta_{1,k}\}_k$ (resp. $\{\varTheta_{2,k}\}_k$; resp. $\{\varTheta_{3,k}\}_k$)
can be determined effectively (resp. once $\Theta_{g=2,2}$ and $\Theta_{2,7}$ are known;
resp. once $\Theta_{g=3,4}$, $\Theta_{3,9}$, $\Theta_{3,14}$, $\Theta_{2,2}$ and $\Theta_{2,7}$ are known).

\subsection{After NMSP}

 In \cite{NMSP1,NMSP2,NMSP3} a package of (N)MSP fields leads to Feynman graph structures of quintic's GW potentials $F_g$'s. In particular this provides a very short determination of $F_1$ and $F_2$ in \cite{NMSP3}. There are two natural questions afterwards.
 \begin{enumerate}
 \item By Theorem \ref{main-1} the primitive FJRW invariant $\Theta_{2,2}$ determines $F_2$. Does the information of $F_2$ determine  $\Theta_{2,2}$ backwards? Note that $\Theta_{2,2}$ could be more difficult to be calculated, compared to GW invariants $N_{2,k}, k=0,1,2,3$ which admit interpretations as countings of maps. Similarly does $F_3$ determine $\Theta_{3,4}$?
 \item As the result in \cite{NMSP1,NMSP2,NMSP3} is a package of Theorem \ref{main-1}, is there analogous package for Theorem \ref{main-2}?
 If so, there should be Feynman structure among FJRW potentials $F^{LG}_g:=\sum_k \Theta_{g,k}q^k$ as  well.  \end{enumerate}
  
 We expect the answers to both above questions are  affirmative.

\end{document}